\documentclass[11pt]{article}
\usepackage{amsfonts}
\usepackage{amsmath,amsthm,amscd,amssymb,mathrsfs,setspace, textcomp}
\usepackage{cite}
\usepackage{latexsym,epsf,epsfig,float}
\usepackage{color}
\usepackage[hmargin=2.25cm,vmargin=2.75cm]{geometry}

\usepackage{hyperref}
\usepackage{array,multirow}
\usepackage{multicol}

\usepackage{empheq}

\usepackage{enumitem,mathtools}
\usepackage{multimedia}
\usepackage{tikz}
 
\usepackage{pgfplots}
\usepackage{tikz}
\usetikzlibrary{patterns,matrix,arrows,decorations,calc,decorations.pathreplacing,backgrounds,shapes}
\usetikzlibrary{decorations.pathmorphing,decorations.text}
\usetikzlibrary{scopes,decorations.markings,positioning,shapes.arrows} 
\usepackage{graphicx}  
\usepackage{bm}
\usepackage{array}
\usepackage{geometry}  
\usepackage{booktabs} 
\usepackage{subfig}
\usepackage{tabularx}
\usepackage{caption}
\usepackage[english]{babel}
\usepackage{enumitem}
\newlist{steps}{enumerate}{1}
\setlist[steps, 1]{leftmargin=1cm, label = Step \arabic*:}

\usepackage{titlesec}
\titlespacing*\section{0pt}{4pt plus 2pt minus 2pt}{3pt plus 2pt minus 2pt}
\titlespacing*\subsection{0pt}{4pt plus 2pt minus 2pt}{3pt plus 2pt minus 2pt}
\titlespacing*\subsubsection{0pt}{4pt plus 2pt minus 2pt}{3pt plus 2pt minus 2pt}

\newcommand{\ds}{\displaystyle}

\newcommand{\cA}{{\mathcal{A}}}

\newcommand{\cD}{\mathcal{D}}

\newcommand{\cE}{{\mathcal{E}}}

\newcommand*{\rom}[1]{\expandafter\@slowromancap\romannumeral #1@}
\theoremstyle{plain}
\newtheorem{theorem}{Theorem}[section]

\newtheorem{definition}{Definition}

\newtheorem{corollary}[theorem]{Corollary}

\theoremstyle{remark}
\newtheorem{remark}{Remark}[section]
\begin{document}

\title{Theory of Solutions for An Inextensible Cantilever}
 \author{\normalsize \begin{tabular}[t]{c@{\extracolsep{.8em}}c}
            { \large Maria Deliyianni} &{ \large Justin T. Webster} \\
 \it University of Maryland, Baltimore County   \hskip.4cm  & \hskip.4cm \it University of Maryland, Baltimore County   \\
 \it Baltimore, MD &\it Baltimore, MD\\
    \it  mdeliy1@umbc.edu &  \it websterj@umbc.edu \\
\end{tabular}}
\maketitle

\begin{abstract} {\noindent Recent equations of motion for the large deflections of a cantilevered elastic beam are analyzed. In the traditional theory of beam (and plate) large deflections, nonlinear restoring forces are due to the effect of stretching on bending; for an inextensible cantilever, the enforcement of arc-length preservation leads to quasilinear stiffness effects and  inertial effects that are both nonlinear and nonlocal. For this model, smooth solutions are constructed via a spectral Galerkin approach. Additional compactness is needed to pass to the limit, and this is obtained through a complex procession of higher energy estimates. Uniqueness is obtained through a non-trivial decomposition of the nonlinearity. The confounding effects of nonlinear inertia are overcome via the addition of structural (Kelvin-Voigt) damping to the equations of motion. Local well-posedness of smooth solutions is shown first in the absence of nonlinear inertial effects, and then shown with these inertial effects present, taking into account structural damping. With damping in force, global-in-time, strong well-posedness result is obtained by achieving exponential decay for small data.
  \\[.15cm]
\noindent {\bf Key terms}: inextensible cantilever,  elasticity, quasilinear, structural damping, well-posedness
 \\[.15cm]
\noindent {\bf MSC 2010}: 74B20, 35L77, 35B65, 74H20}
\end{abstract}
\maketitle

\section{Introduction}
\subsection{Motivation and Overview}
The large deflections of elastic beams and plates have broad applicability in engineering and other physical sciences, and they have been intensely studied from the modeling, analytical, and computational points of view (see, e.g., \cite{book,springer,ciarlet, lions}). Specifically, with respect to fluid-structure interaction models, the large deflections of panel, airfoil, and flap structures are of particular interest \cite{dowell,survey2} (and references therein). In these circumstances, the presence of a fluid flow can act as a destabilizing mechanism, giving rise to self-excitation instabilities (i.e., aeroelastic flutter \cite{survey2,McHughIFASD2019}) that manifest as {\em limit cycle oscillations} (LCOs). In such applications, relevant large deflection models require nonlinear restoring forces that take into account higher order effects,  typically appearing via a  potential energy above the ``quadratic" level.  The choice of nonlinearity  dictates the qualitative features of the post-onset dynamics---which is to say, the dynamics in the  nonlinear regime of interest. Traditional large deflection theory for {\em panels} (i.e., fully restricted boundary conditions) is that of von Karman \cite{springer}, producing semilinear, cubic-type nonlinearities based on a quadratic strain-displacement law \cite{ciarlet,lagleug}.

The configuration of a {\em cantilever in axial flow}, whereby an elastic beam (or thin plate) has a flow of gas running {\em along its principal axis}, has been historically overlooked. Until about 15 years ago, interest in this configuration
was minimal  \cite{huang}, while interest in airfoil and panel flutter has been {\em immense} for more than 75 years \cite{survey2,dowell}.  A cantilever in axial flow is particularly prone to aeroelastic instability, with the bifurcation leading to sustained LCOs. This fact is useful in the development of vibration-based energy harvesting devices \cite{DOWELL,energyharvesting}. In such applications, dynamic instability is encouraged to extract energy from LCOs of the elastic cantilever, after the onset of flutter.  The main idea for large displacement harvesters is to capture  mechanical energy via piezoelectric laminates or patches (for which oscillating strains induce current \cite{energyharvesting}). The feasibility of such a system has been recently demonstrated with affixed piezo (SMART) materials \cite{DOWELL,energyharvesting,fab-han:01:DCDS,piezomass}. 
\begin{figure}[htp]
\begin{center}
\includegraphics[width=2in]{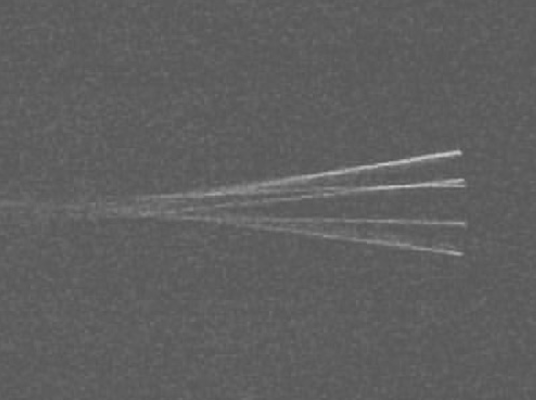}
\hspace*{0.55in}
\includegraphics[width=2in]{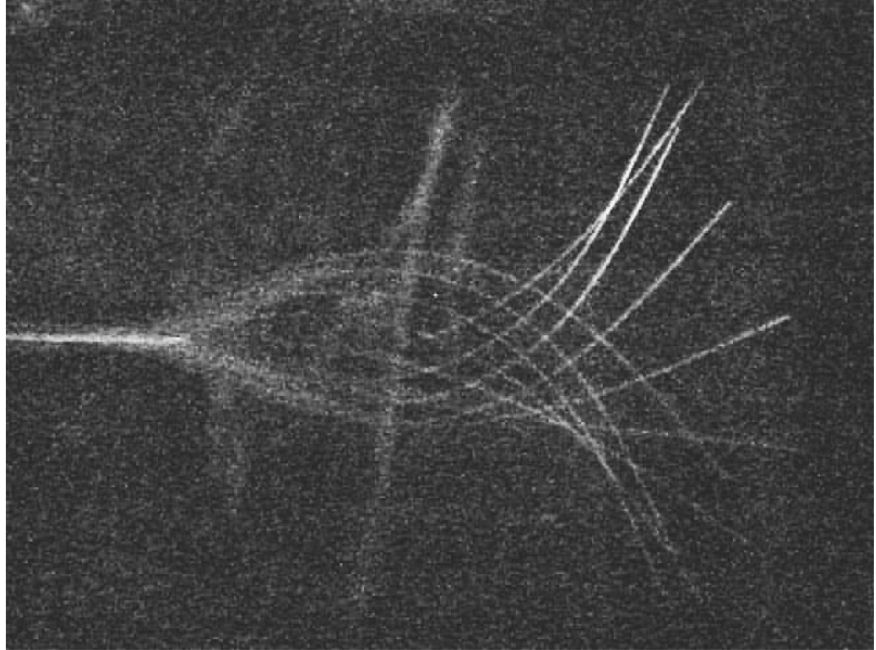}
\vskip-.15cm
\caption{Temporal snapshots of post onset, small amplitude LCO (left) and large amplitude LCO (right) for a cantilever in axial flow. Captured from wind-tunnel simulations \cite{dowell4,inext2}. In these experiments, the airflow runs from left to right.}
\end{center} 
\end{figure}

To effectively and efficiently harvest energy in this manner, one must be able capture and predict the post-onset behaviors of a cantilever, and thus, {\em one must have a viable model for the large deflections of cantilever}. Tradition nonlinear elastic models are based on local stretching effects, which are dominant when the entire structural boundary is restricted. However, for a cantilever, nonlinear effects are decidedly not due to stretching \cite{DOWELL,dowellmaterial,inext1,inext2}. An appropriate nonlinear cantilever model, then, should {account for in-plane displacements and variable stiffness and  inertia}. Thus, dominant nonlinear effects should come from the cantilever's {\em inextensibility}, rather than extensible effects (stretching). Inextensible cantilever models are rather recent \cite{inext1,dghw,inext2}, and have not been addressed---even in vacuo---in the rigorous mathematical literature. Thus, in this paper, we discuss the model derivation for an inextensible cantilever, and we produce a rigorous theory of solutions for the corresponding PDE model. The treatment at hand is a rigorous follow up to the recent \cite{dghw}, where an inextensible cantilever is discussed and analyzed numerically. In that paper, the results proven here were announced.

\subsection{Large Deflections of Cantilevers}
Cantilever flutter is associated with large deflections on the order of the beam's length \cite{dowell4,inext2}. For such deflections, structural nonlinearity arises in modeling as a byproduct of the inclusion of higher order terms in strain and energetic expressions. From a mathematical point of view, the presence of non-conservative flow-effects gives rise to a beam bifurcation (this is flutter \cite{survey2,HHWW,HTW,dghw}), which would yield exponential growth in time for a linear model, according to destabilized eigenvalues. To ensure that flow-destabilized trajectories remain bounded, one must consider a nonlinear restoring force, active for large displacements, slopes, or curvatures. 
	
The typical way to achieve this in the theory of elasticity is through the inclusion of cubic-type forces that arise from the effect of local stretching on bending  \cite{HTW,lagleug} . Since extensibility is not physically dominant for cantilevers, engineers have posited that the prevailing nonlinear forces result from inextensibility \cite{semler2,inext1,inext2}. Although enforcing inextensibility---as a nonlinear constraint---can be quite challenging, the recent modeling work \cite{inext1,inext2} utilizes a simplified approach that accounts for both nonlinear stiffness and nonlinear inertia effects. The result is a beam theory that is both quasilinear and nonlocal in space, as well as  implicit in the transverse acceleration, which is to say the dynamics are not a traditional second order evolution. The model was recently considered in the mathematical paper \cite{dghw}, where solutions were defined and qualitatively investigated from a numerical point of view under the influence of non-conservative flow effects. 
{\bf In the paper at hand we consider this inextensible elastic cantilever model and develop a rigorous well-posedness theory for smooth solutions}.

For the remainder of this treatment, let ~$(u,w) \in \mathbb R^2$ denote the Lagrangian displacement of a beam whose centerline equilibrium position is $x \in [0,L]$. This is to say that $u(x,t)$ is the axial (longitudinal) displacement from equilibrium and $w(x,t)$ is the transverse beam deflection.
 \begin{center}
\begin{tikzpicture}[scale=1.1,>=latex']
\draw[->,thick](0,0)--(6,0) node[below]{\footnotesize $x$};
\draw[->,thick](0,0)--(0,2.25) node[left]{\footnotesize $z$};
\draw[gray,line width=3pt] (0,0)--(5.2,0);
\draw (5.2,0.25)--(5.2,-0.25)node[below]{\footnotesize $L$};
\draw (0,0.25)--(0,-0.25)node[below]{\footnotesize $0$};
\draw[line width=3pt,smooth] plot[domain=0:4.5]({\x},{1.25-1.25*cos(0.125*pi*\x r)});
\draw[<->] (4,0)--node[pos=0.5,left]{\footnotesize $w$}(4,1.25);  
\draw[thick,dashed](4.5,0)--(4.5,2.2);
\draw[thick,dashed](5.2,0)--(5.2,2.2);
\draw[<->] (4.5,1.5)--node[pos=0.5,above]{\footnotesize $u(L)$}(5.2,1.5);
\end{tikzpicture}
\end{center}

Then, the {\bf inextensible equations of motion} of interest---derived later in Section \ref{modeling}---are:
\begin{align}\label{dowellnon1}
&\begin{cases} w_{tt} + D \partial^4_{x}w +k_2 \partial^4_{x}w_{t} +\mathbf A [w,u_{tt}] =p(x,t)& \text{ in } (0,L) \times (0,T);\\
w(0)=w_x(0) = 0;~~w_{xx}(L)=w_{xxx}(L)=0  & \text { in } (0,T);\\
w(t=0)=w_0,~~  w_t(t=0)=w_1, \end{cases}\\ \label{dowellnon2}
\text{with}~~~~&~~~~\mathbf A [w,u_{tt}] =-D\partial_x\big[w_{xx}^2w_x \big]+D\partial_{xx}\big[w_{xx}w_x^2\big] +\partial_x\left[w_x\int_x^L u_{tt}d\xi \right] \\
& \hskip1cm u_{tt}(x)=~-\int_0^x[w_{xt}^2+w_xw_{xtt}]d\xi.
\end{align}

We denote by $D$ the standard (mass-normalized) beam stiffness coefficient \cite{inext1}, $L$ is the beam's length, and $k_2 \ge 0$ corresponds to structural damping of Kelvin-Voigt type (discussed further in Section \ref{modeling}). The RHS $p(x,t)$ constitutes a given transverse pressure differential across the deflected beam.

We now mention the only (to the best of our knowledge) large deflection beam models in the PDE literature that accommodate a cantilever configuration. First is
 the extensible system found in \cite{lagleug}; in addition to standard elasticity assumptions, it invokes a quadratic strain-displacement law consistent with von Karman theory \cite{wonkrieg}. As a system, it is nonlinearly coupled in $u$ and $w$ via the beam's {\em extensionality}: ~$[u_x+\frac{1}{2}w_x^2]$. 
 \begin{equation} \label{LLsystem}
\begin{cases} u_{tt} -D_1\partial_x\left[u_x+\frac{1}{2}(w_x)^2\right]=0; \\\ds
(1-\alpha\partial_x^2)w_{tt}+D_2\partial_x^4 w+(k_0-k_1\partial_x^2) w_t -D_1\partial_x\left[ w_x(u_x+\frac{1}{2}w_x^2)\right] = p(x,t); \\[.1cm]
u(0)=0,~w(0)=w_x(0)=0;~~\left[u_x(L)+\frac{1}{2}w_x^2(L)\right]=0,~w_{xx}(L)=0, ~\alpha\partial_xw_t(L)=D_2w_{xxx}(L);\\
u(t=0)=u_0,~u_t(t=0)=u_1;~~w(t=0)=w_0,~w_t(t=0)=w_1. \end{cases}
\end{equation}
The above {\em Lagnese-Leugering} system is the beam analog of the so called {\em full von Karman plate equations} \cite{koch}.
Above, $D_1,D_2 >0$ are two different mass-normalized stiffness parameters and $\alpha\ge 0$ represents (linearized) rotational inertia in the filaments of the plate\footnote{$D_1$, $D_2$, and $\alpha$ are not necessarily independent in the presentation of these equations with physical coefficients.}. The coeffcients $k_i\ge 0$ correspond to damping of various strengths. The paper \cite{lagleug} considers a variant of this model that allows for boundary feedbacks and takes $k_0=k_1=0$; the results include nonlinear semigroup well-posedness, as well as a stabilization result.

One further consideration can be made as a simplification of the above system when we take negligible in-plane accelerations, $u_{tt} \approx 0$. Elementary simplifications then produce a scalar extensible cantilever, as was studied in \cite{HTW}.
In that reference, well-posedness and long-time behavior of the scalar system are analyzed in the presence of a non-conservative $p(x,t)$ representing an inviscid potential flow. A principal consideration in that analysis is whether $\alpha>0$ or $\alpha=0$. In the case where $\alpha>0$, stabilization-type estimates require damping strength to be tailored to the inertia, i.e., ~ $\alpha>0~\implies~k_1>0$. 

Each of these extensible beam models above is reasonable under certain modeling hypotheses, in particular contexts. However, {\bf as is clear from the engineering literature discussed above,  large deflections of a flow-driven cantilever should be appropriately modeled with inextensibility}. Lastly, we point to some work that addresses the 2 or 3-D deflections of inextensible rods in the seminal reference \cite{antman} (and many references therein), as well as \cite{antman2}. The principal inextensible models therein are linear of wave type (or their linearizations are, i.e., second order in space, perhaps with strong damping), and hence fundamentally distinct from the nonlinear models considered here which account for nonlinear inertia and stiffness effects.

\subsection{The Analysis At Hand}
In this section we outline the remainder of the paper and state our results informally. 

We conclude the introduction in Section \ref{tech}, where we discuss the novel mathematical contributions of this treatment and technical challenges of the analysis. The remaining preliminary sections are \ref{modeling} and \ref{funx}. Section \ref{modeling} presents a derivation of the equations of motion, as well as a brief discussion of structural damping central to this analysis. Section \ref{funx} presents the functional setup for the analysis and technical definitions of solutions used in subsequent well-posedness and stability proofs. 

Each of the remaining sections corresponds to a main result. In those latter sections, we will: (i) state a main result technically, using the terminology and concepts established in Section \ref{funx}; (ii) outline the proof briefly; and (iii) execute the proof in detail. Below, we give a short, nontechnical description of each of these main sections.

Section \ref{noinertia} provides a {\em local\footnote{local, in the sense that the time of existence depends on the size of the initial data in the associated solution topology.} well-posedness for strong solutions} for \eqref{dowellnon1} {\em in the absence of nonlinear inertia with no imposed damping}. The resulting system is a conservative, quasilinear beam system.  The result---Theorem \ref{withoutiota}---is built upon higher order energy estimates used to obtain additional compactness needed in executing a Galerkin procedure with cantilever eigenfunctions..

Section \ref{inertia} provides a {\em local well-posedness result---in the presence of nonlinear inertia---for strong solutions} in Theorem \ref{withiota}; in this case, {\em some damping is required} ($k_2>0$) to obtain estimates in the construction of solutions. The damping addresses the nonlocal and implicit nature of the inertial terms. 

Section \ref{global} provides our final main result: {\em global existence of strong solutions for small data} (in the presence of both inertia and damping). This result is typical for quasilinear hyperbolic dynamics, whereby the presence of damping and small data allow for stabilization estimates that ensure exponential decay, yielding an arbitrary time of existence.

Lastly, Section \ref{open} gives a brief discussion of open problems related to the model at hand, and Section \ref{ack} gives the authors' declarations and acknowledgements.

\subsection{Novel Contributions and Technical Challenges} \label{tech}
The model we focus on here only appeared for the first time in the context of elastic cantilevers in the recent papers \cite{inext1,inext2}. These (and other earlier works focusing on inextensible pipes conveying fluid such as \cite{semler2,paidoussis}) are largely engineering-oriented, making use of finite dimensional analyses via modal truncation or the Rayleigh-Ritz method at the energetic level. And, although the present authors' recent work \cite{dghw} discusses solutions and states well-posedness theorems, it is numerically-focused, without proofs. Thus, {\bf the existing body of work on inextensible elasticity does not address}:
\begin{itemize}
\item a construction of PDE solutions (at the infinite dimensional level)
\item (Hadamard) well-posedness
\item  the effects of damping in relation to nonlinear inertial terms
\item  the time of existence for solutions or quantitative restrictions on data.
\end{itemize}
To the knowledge of the authors, {\em this is the first treatment to rigorously address the theory of solutions for inextensible elasticity.}  

Although the central problem here is a 1-D beam, the following issues render the analysis quite challenging. Some of these issues are common for quasilinear dynamics, but many are not (e.g., those associated with nonlinear inertia), and we also point to the non-trivial {\em interaction} between (high order) free boundary conditions, nonlinear stiffness, and nonlocal inertial terms. 

The {\bf technical challenges faced in the analysis} are:
\begin{itemize}\setlength\itemsep{.1em}
\item Despite a good, conservative structure for the baseline equations of motion, quasilinear and semilinear terms do not straightforwardly admit (semigroup or fixed point) perturbation methods.
\item The term ~{\small $\ds \partial_{x}[w_{xx}^2w_x]$} precludes weak limit point identification at the baseline energy level. 
\item Nonlinear terms and free boundary conditions (i) do not readily permit differentiation of the equations to obtain higher energy estimates, and (ii) convolute the standard technique of {\em going back through the equations} to trade time and space regularity. 
\item Nonlinear inertial terms (i) present themselves at a level above finite energy, (ii) are also nonlocal, and (iii) are implicit terms in $w_{tt}$, and hence do not constitute a traditional evolution. The truncated version of the dynamics is in fact {\em quasilinear in time} \eqref{sep2}.
\end{itemize}
\noindent In addressing the issues above, we note the following specific {\bf novelties of this analysis}:
\begin{itemize}\setlength\itemsep{.2em}
\item The sequence of multipliers used to close estimates in obtaining compactness are  non-standard, including the use of stabilization-type multipliers.
\item A novel decomposition of nonlinear differences exploits polynomial symmetry for a  non-obvious uniqueness proof, relying critically on  smooth trajectory estimates obtained earlier. 
\item The inclusion of damping to permit appropriate estimates for well-posedness of the full model is a peculiarity, one that, at present, we cannot avoid. On the other hand, including damping in the full model \eqref{dowellnon1}--\eqref{dowellnon2} successfully obtains global solutions for small data.
\end{itemize}
\section{PDE Model Derivation}\label{modeling}

Recall that $w(x,t)$ is the transverse deflection and $u(x,t)$ is the in-axis displacement from equilibrium of a beam at $t \in [0,T]$ and a spatial point $x \in [0,L]$. Let ~$\varepsilon(x,t)$ describe the {\em axial strain} along the centerline of the beam. In this section we derive the in-vacuo equations of motion via Hamilton's principle. The inextensibility condition is simplified to an {\em effective inextensibility constraint}, which is enforced via a Lagrange multiplier. Our derivation tracks the one first appearing in \cite{inext1}, and we point to the earlier references \cite{semler2,paidoussis} for inextensibility treated in the context of pipes conveying fluid.

	\subsection{Inextensibility}
According to classical work (e.g., \cite{stoker,semler2}) we have the Lagrangian strain relation \cite{semler2}
$$\big[1+\varepsilon]^2=(1+u_x)^2+w_x^2.$$
When the beam is {\em inextensible}, we take $\varepsilon(x,t)=0$, which immediately yields the  condition
\begin{equation}\label{realinext}
1=(1+u_x)^2+w_x^2.
\end{equation}

From \cite{semler2,inext1,ciarlet}, large deflections dictate that higher order nonlinear terms should be retained, namely, {\em up to cubic order}. (For variational purposes, then, energetic expressions will be accurate up to quartic order.) By expanding the inextensibility condition \eqref{realinext}, we see  that if $w_x \sim \epsilon$, we will have $u_x \sim \epsilon^2$:
$$2u_x+u_x^2+w_x^2=0.$$
As in \cite{inext1},  we drop $u_x^2 \sim \epsilon^4$, owing to its relative order being above cubic. Approximating, then
 $$0=2u_x+w_x^2 ~~\implies~~ u_x=-\frac{1}{2}w_x^2.$$ This yields what we henceforth refer to as the {\em effective inextensibility constraint}, providing a direct relationship between $u$ and $w$: \begin{equation}\label{inext} u(x,t)=-\frac{1}{2}\int_0^x[w_x(\xi,t)]^2d\xi.\end{equation}

	\subsection{Nonlinear Elasticity}
	
Define the elastic potential energy ($E_P$) via beam curvature $\kappa$ and constant stiffness $D$ (flexural rigidity) \cite{semler2} in the standard way
 $$E_P \equiv \frac{D}{2}\int_0^L\kappa^2dx.$$ 
 Owing to inextensibility, we may take the beam's displaced state, $\{(x+u(x),w(x))~:~x \in [0,L]\}$, as a parametrized curve. The standard expression for curvature in this scenario is: 
 $$\kappa = \frac{(1+u_x)w_{xx} - u_{xx}w_x}{[(w_x)^2 + (1+u_x)^2]^{3/2}}.$$
 
From inextensibility \eqref{realinext} (without approximation), we see that the denominator is one. From \eqref{realinext},  we can also write $u_{x}= \sqrt{1 - w^2_{x}} - 1$ which leads to $u_{xx}=-w_{x}w_{xx} (1-w^2_{x})^{-1/2}$. Substituting in $\kappa$, we obtain:
\begin{align*}
\kappa = (1+u_{x})w_{xx} - w_{x}u_{xx} =&~ (1-w^2_{x})^{1/2}w_{xx} + w_{x}(w_{x}w_{xx}(1-w^2_{x})^{-1/2}) =  \frac{w_{xx}}{(1-w^2_{x})^{1/2}} .
\end{align*} 
  To be consistent with the approximation that yields \eqref{inext}, we must retain terms at the level of $w_x^2$ in approximating $E_P$ \cite{inext1,semler2}. Via a Taylor expansion, we take  ~$\ds \kappa \approx w_{xx}\sqrt{1+w_x^2}.$
 \begin{remark} This point distinguishes the derivation from linear elasticity in $w$, where $\kappa \approx w_{xx}$.\end{remark}

 Finally, the effective potential energy for the problem at hand becomes
 \begin{equation}\label{potentialE} E_P = \frac{D}{2}\int_0^L w_{xx}^2 \left( 1+ w_{x}^2 \right)dx.\end{equation}
 The kinetic energy ($E_K$) for the dynamics taken in the standard way for a mass-normalized beam:
 \begin{equation}\label{kineticE} E_K = \frac{1}{2}\int_0^L\left( u^2_{t} + w^2_{t} \right) dx.\end{equation}
	
	\subsection{Hamilton's Principle}
	
To derive the equations of motion and the associated boundary conditions, we utilize Hamilton's Principle \cite{inext1,lagleug}. We consider displacements $u$ and $w$ (and hence virtual displacements $\delta u$ and $\delta w$) which are smooth and respect the essential boundary conditions at $x=0$, namely:
\begin{align*} &w,~w_x,~\delta w, ~\delta w_x:~ 0 ~\text{ at }~x=0;&~~~u,~\delta u:~0 ~\text{ at }~x=0.& \end{align*}
The effective inextensibility constraint, $f \equiv u_{x} + (1/2)w^2_{x}=0$, will be appended to the system via a {\em Lagrange multiplier $\lambda$}. Thus, we express the Lagrangian in the usual way:
\begin{equation}
\label{Lagrangian}
\mathcal{L} = E_{K} - E_{P} + \int_0^L \lambda f dx.
\end{equation}

Taking the variation of $(\ref{Lagrangian})$ and performing the necessary integration by parts with respect to both time and space, Hamilton's principle provides the Euler-Lagrange equations of motion and the associated boundary conditions. Virtual changes are considered for both displacements, $u$ and $w$.\footnote{Note that virtual change in $\lambda$ simply produces the effective inextensibility constraint.}

To minimize the Lagrangian, we set ~{\small $\delta \int_{t_1}^{t_2}\mathcal L  dt \equiv 0$}~ and utilize the arbitrariness of the virtual changes $\delta u$ and $\delta w$. For interior terms, we gather virtual changes and set the totals equal. The relevant calculation pertains to the $E_P$:
\begin{equation}\label{delE} \delta E_P = D\int_0^L\big[(1+w_x^2)w_{xx}\big]\delta w_{xx}+\big[(w_xw_{xx}^2)\big]\delta w_xdx.\end{equation}
Integrating by parts until only $\delta w$ appears, and utilizing the arbitrariness of the virtual changes, we obtain the unforced equations of motion:
\begin{align}
\label{deltau}
\text{from }~\delta u:& ~~~~u_{tt} + \lambda_x = 0\\
\label{deltaw}
\text{from }~\delta w:& ~~~~w_{tt} - D \partial_{x} \left( w^2_{xx} w_{x} \right) + D \partial_{xx} \left( w_{xx} \left [ 1 + w^2_{x} \right ] \right) + \partial_x \left( \lambda w_{x} \right) = 0.
\end{align}

For the (natural) boundary conditions at $x=L$, the relevant calculations pertain to $w$ (the $u$ and $\lambda$ conditions can then be inferred). In the integration by parts proceeding from \eqref{delE}, we obtain by the arbitrariness of $\delta w$, $\delta w_x$ and $\delta u$ at $x=L$:
\begin{align}\label{firstnatural}
& \lambda(L) = 0;~~~
(1+w_x^2(L))w_{xx}(L) = 0 ; ~~~
(1+w_x^2(L))w_{xxx}(L)+w_x(L) w^2_{xx}(L)=0.&
\end{align}
From \eqref{firstnatural}, we infer that $w_{xx}(L)=w_{xxx}(L)=0$---the standard free boundary conditions.
\begin{remark} This fact is both critical and somewhat surprising, as the nonlinear effects (and their previously discussed simplifications) do not alter the standard {\em linear} boundary conditions for a cantilever. Note that in extensible elasticity, this is not always the case \cite{springer,HTW}.\end{remark}

Now, using the equation \eqref{deltau} we can formally write 
\begin{equation*}
\lambda(x) = - \int_0^x u_{tt}(\xi) d\xi + \lambda(0).
\end{equation*}
We then utilize the fact that $\lambda(L)=0$ to conclude ~$\ds
\lambda(0) = \int_0^L u_{tt}(\xi) d \xi.$
From this we  deduce:
\begin{equation*}
\lambda(x) =  \int_x^L u_{tt}(\xi) d\xi.
\end{equation*}

Substituting the above expression in \eqref{deltaw} we finally obtain the equations of motion \eqref{dowellnon1}--\eqref{dowellnon2}, and the corresponding boundary conditions for $w$, as well as for $u$ and $\lambda$ at $x=L$.

	\subsection{Damping}
Discussion of damping in beams goes far back in both the engineering literature \cite{bolotin,semler2} as well as the mathematical literature \cite{beamdamping,che-tri:89:PJM}. In the treatment at hand, some additional velocity regularization is needed to address the nonlinear inertial terms; namely $w_t$ must be ``better" than $C([0,T];L^2(0,L))$. We obtain this by imposing Kelvin-Voigt type structural damping. Note, this type of damping is in fact invoked in the engineering-oriented references \cite{paidoussis,semler2} for improving numerical simulations. The recent \cite{dowellmaterial} addresses local damping and stiffness in a cantilever from a modeling and experimental point of view.
 
Let us here refer to the damped, linear Euler-Bernoulli beam equation
$$w_{tt}+D\partial_x^4 w+[k_0-k_1\partial_x^2+k_2\partial_x^4]w_t = p.$$ Weak (frictional) damping has the form ~{\small $ k_0 w_t$},  providing no velocity regularization. In the elasticity context,  Kelvin-Voigt damping ~{\small $k_2\partial_x^4 w_t$} is strain-rate type, and mirrors the principal (linear) operator, providing a strong dissipative effect. In fact, this damping transforms the underlying dynamics to be of parabolic type \cite{che-tri:89:PJM,redbook}.  Square root-like damping, {\small $-k_1\partial_x^2w_t$} \cite{fab-han:01:DCDS}, interpolates between the previous two damping types. (See \cite{dghw,HTW} for more discuss of damping in the context of nonlinear cantilevers.)
 
 \begin{remark}
Square root-type damping corresponds to modal damping models \cite{dowell}, as one finds frequently in the engineering literature \cite{bolotin,follower, McHughIFASD2019}.  However, the boundary conditions for a given problem affect the physical interpretation of square-root type damping; in \cite{beamdamping} it is noted that square-root type damping has a questionable physical interpretation for a cantilevered configuration. See also \cite{HHWW} for more recent discussion.  In the analysis here, we utilize the (strong) Kelvin-Voigt  damping.
	\end{remark}
	\begin{remark}
	It is of course of interest to discuss damping in the context of the stiffness-only model $\iota=0$. On the other hand, in this treatment the damping is primarily included to mitigate the effects of nonlinear inertia. We discuss this further in Section \ref{open}.	\end{remark}

\section{Functional Setup and Key Notions}\label{funx}
\subsection{Equations of Motion} \label{EquationofMotion}
With the derivation above, we recall the equations of motion, allowing for Kelvin-Voigt damping $k_2 \ge 0$, and including {\em flags} for the nonlinear terms:
\begin{equation}\label{dowellnon*}
\begin{cases} \displaystyle w_{tt}+D\partial_x^4w+ k_2 \partial_x^4 w_t+\mathbf A_{\iota,\sigma} (w,u_{tt}) =p(x,t)& \text{ in } (0,L) \times (0,T)
\\ w(t=0)=w_0(x), w_t(t=0)=w_1(x) \\ w(x=0)=w_x(x=0)=0;~w_{xx}(x=L)=w_{xxx}(x=L)=0,
\end{cases}
\end{equation}
\begin{align}\label{dowellnon2*}
\mathbf A_{\iota,\sigma} (w,u_{tt}) =&- \sigma D\partial_x\big[w_{xx}^2w_x \big]+ \sigma D\partial_{xx}\big[w_{xx}w_x^2\big] + \iota \partial_x\left[w_x\int_x^L u_{tt}(\xi)d\xi \right] \\
u(x)=&~-\frac{1}{2}\int_0^x \left [w_x(\xi) \right ]^2d\xi.\label{dowellnon3*}
\end{align}
To simplify terminology, we use the following language  from here on:
\begin{align*}\text{\bf   [NL Stiffness]}=&~ -D\partial_x\big[w_{xx}^2w_x \big]+D\partial_{xx}\big[w_x^2w_{xx}\big] \\ \text{\bf   [NL Inertia]}=&~ \partial_x\left[w_x\int_x^L u_{tt}(\xi)d\xi \right],\end{align*} the latter of which is nonlocal, when written in $w$ through \eqref{dowellnon3*}.
The flags, $\iota,\sigma=0 \text{ or } 1$, in \eqref{dowellnon2*}, easily isolate particular nonlinear effects. This is to say, when $\iota=0$, we say that {\bf \small [NL Inertia]} is turned off.

\begin{remark}\label{quasi}
For convenience, we note two expansions.
First
$$\text{\bf  [NL ~Stiffness]}=D[w_{xxx}^3+4w_xw_{xx}w_{xxx}+w_x^2\partial_x^4 w],$$
which highlights the quasilinear nature of the PDE (with high order semilinearity).
Secondly,
\begin{equation} \label{NLInertia}
{\bf  [NL ~Inertia]} =-w_xu_{tt}+ w_{xx}\int_x^L u_{tt}d\xi,~~\text{with}~~~~ u_{tt}=-\int_0^x[w_{xt}^2+w_xw_{xtt}]d\xi,
\end{equation}
which highlights that, when closed in $w$, (i) there is high temporal regularity required to interpret the strong form of the PDE, and (ii) the equation is implicit in the acceleration $w_{tt}$.
\end{remark}

\subsection{Notation and Conventions}

For a given spatial domain $D$,
its associated $L^{2}(D)$ will be denoted as $||\cdot ||_D$ (or simply $%
||\cdot||$ when the context is clear). Inner products in a Hilbert space  are written $(\cdot ,\cdot)_{H}$ (or simply $(\cdot ,\cdot)$ when $H=L^2(D)$ and the context is clear). We will also denote pertinent duality pairings as $%
\left\langle \cdot ,\cdot \right\rangle _{X\times X^{\prime }}$, for a given
Banach space $X$, as well as the general notation for a norm, $||\cdot||_X$. The open ball of radius $R$ in $X$ will be denoted $B_R(X)$. The space $H^{s}(D)$ will indicate the standard Sobolev space of
order $s$, defined on domain $D$, and $H_{0}^{s}(D)$ will be the closure
of $C_{0}^{\infty }(D)$ in the $H^{s}(D)$-norm $\Vert \cdot \Vert
_{H^{s}(D)}$, also written as $\Vert \cdot \Vert _{s}$. For $\Gamma \subset \partial D$, boundary restrictions $u\big|_{\Gamma}$ are taken in the sense of the trace theorem for $u \in H^{1/2^+}(D).$

The constant $C$ we take to mean a generic constant that may change from line to line. In estimates where dependencies are critical, we will write $C(q_i)$, where $q_i$ are relevant quantities. Additionally, in our involved estimates below, for situations where ~$\ds ||q_1||_X \le C ||q_2||_Y$ for some quantities $q_1,q_2$ in spaces $X$ and $Y$, with $C$ having no critical dependencies, {\em we will simply write ~{\small $\ds ||q_1||_X \lesssim ||q_2||_Y.$}}

Finally, we will frequently make use of standard Sobolev embeddings (in particular, that of \newline  $H^{1/2^+}(0,L) \hookrightarrow L^{\infty}(0,L)$) as well as the Sobolev interpolation inequalities \cite{evans}.

\subsection{Energies}\label{energiessec}
With reference to Section \ref{modeling}, we employ the following energies:
\begin{equation}\label{energiesdef}
E(t) \equiv E_K(t)+E_P(t) \equiv \frac{1}{2}\left[||w_t||^2+\iota ||u_t||^2\right]+\dfrac{D}{2}\left[||w_{xx}||^2+\sigma ||w_xw_{xx}||^2\right].
\end{equation}
{\em The energies now include the nonlinear flags.} This can be written in $w$ explicitly using ~{\small $u_t=-\int_0^xw_{x}w_{xt}d\xi$}.

In the unforced situation, with $p(x,t) \equiv 0$, the formal energy identity is obtained by the velocity multiplier $w_t$ on \eqref{dowellnon*} taken with the relation \eqref{dowellnon3*}, yielding
$$E(t) +k_2 \int_s^t||w_{xxt}||_{L^2(0,L)}^2d\tau= E(s),~~0 \le s \le t.$$

Higher order energies corresponding to smooth solutions will be defined in later sections.

	\subsection{Spaces and Operators}\label{spacessec}

The principal state space for cantilevered beam displacement takes into account the clamped conditions:
\[
  H^2_* = \{ v \in H^2(0,L) :  v(0) =0,\quad v_x(0) = 0 \}.
\]
This space is equipped with an $H^2(0,L)$ equivalent inner product:
\begin{equation}\label{H2*prod}
   (v,w)_{H^2_*} = D (v_{xx}, w_{xx}).
 \end{equation}
Denoting $R$ as the Riesz isomorphism $H^2_*\to [H^2_*]'$, we see it is given by:
 \begin{equation}\label{def:R}
 R(v)(w) \equiv  (v,w)_{H^2_*}\,.
 \end{equation}
 
This framework is conveniently induced by the  generator of the linear cantilever dynamics:
\[
  \cA : \mathcal D(\cA)\subset L^2(0,L) \to L^2(0,L),~~ \cA f \equiv D\partial_x^4 f,  
\]
\begin{equation}\label{def:cA0}
\mathcal D(\cA) = \{ w \in H^4(0,L) : w(0) =w_x(0)=0; \; w_{xx}(L)=w_{xxx}(L) = 0\}.
\end{equation}
From this we have in a standard fashion \cite{redbook}: 
\[
  \mathcal D(\cA^{1/2}) = H^2_*, \quad \mathcal D(\cA^{-1/2}) = [H^2_*]' ~\text{ and }~ \cA^{1/2} = R \quad \text{ in \eqref{def:R}}.
\]
Then $(u,\cdot)_{H^2_*}$ is the extension of $(\cA u, \cdot)$ from $\mathcal D(\cA)$ to $H^2_*$ which gives \eqref{H2*prod}.

Using the above spaces we can define the appropriate state space(s) for our dynamics. The finite energy space will be denoted as: $$\mathscr H \equiv \cD(\cA^{1/2})\times L^2(0,L) = H^2_*\times L^2(0,L),$$ with the inner product $y=(y_1,y_2),~ \tilde y=(\tilde y_1,\tilde y_2) \in \mathscr H$
\begin{equation}\label{Hal-prod}
  (y,\tilde y)_{\mathscr H} = (y_1,\tilde y_1)_{H^2_*} + (y_2,\tilde y_2)_{L^2(0,L)}.
\end{equation}

In our discussions, we will also require stronger state spaces (corresponding to {\em strong} solutions):
\begin{align}\label{strongspace}
\mathscr H_s \equiv & ~\mathcal D (\mathcal A) \times \mathcal D(\mathcal A^{1/2}), ~\text{ for} ~ \iota=k_2=0, \\[.1cm] 
\mathscr H_s^I \equiv &~\mathcal D (\mathcal A) \times \mathcal D(\mathcal A), ~\text{ for }~ \iota=1, k_2>0.
\end{align}
The norm in $\mathscr H_s$ is taken (equivalent\footnote{The topological equivalences on $\cD(\cA)$ follow from repeated applications of Poincar\'e.} to the natural operator-induced norm) to be:
\begin{align*} ||y||_{\mathscr H_s}^2 =&~ ||\partial_x^4 y_1||^2+||\partial_{x}^2y_2||^2,  ~\text{ for} ~ \iota=k_2=0, \\[.1cm] 
||y||_{\mathscr H^I_s}^2 =&~ ||\partial_x^4 y_1||^2+||\partial_{x}^4y_2||^2, ~\text{ for }~ \iota=1, k_2>0.
\end{align*}
	
	\subsection{Mode Functions}\label{modefunx}
We will utilize the so called {\em in vacuo modes} (eigenfunctions) associated to the operator $\cA$. Specifically, we work with the Euler-Bernoulli cantilever eigenfunctions as our approximants in $H^2_*$; namely, the eigenvalues and eigenfunctions {\small $\{\lambda_n,s_n(x)\}_{n=1}^{\infty}$} of $\mathcal A$ on $L^2(0,L)$. These modes and associated eigenvalues are computed in an elementary way. 
The $C^{\infty}([0,L])$ mode shapes take the form
\begin{equation}\label{modefunctions}s_n(x)\equiv c_n[\cos(\kappa_nx)-\cosh(\kappa_nx)]+ C_n[\sin(\kappa_nx)-\sinh(\kappa_nx)], ~\kappa_n^4=\lambda_n,\end{equation}
where the $\kappa_n$  are obtained (numerically) by solving the associated characteristic equation $$\cos(\kappa_nL)\cosh(\kappa_nL)=-1.$$  The $C_n$ are obtained by invoking the boundary conditions:
\[ C_n=\dfrac{-c_n\big(\cos(\kappa_nL)+\cosh(\kappa_nL)\big)}{\sin(\kappa_nL)+\sinh(\kappa_nL)},\]
and the $c_n$ values are chosen to normalize the  functions in  $L^2(0,L)$.  

Via the spectral theorem, these functions are {\em complete} and {\em orthonormal} in $L^2(0,L)$, as well as complete and orthogonal in $H_*^2$ (with respect to $(\cdot,\cdot)_{H^2_*}$). These eigenvalues have the property that $0 < \lambda_1 < \lambda_2 <... \to \infty.$
	
	\subsection{Definition of Solutions} \label{sols}

We provide the natural setting for the weak formulation of the problem; this will yield the appropriate starting point for our Galerkin procedure to construct solutions. Ultimately, we will construct weak solutions that possess additional regularity; these, in turn, will be strong solutions.

We begin with the {\em weak form} of \eqref{dowellnon*} which we define for functions that are smooth in time:
\begin{align} \label{weakform}
\left( w_{tt}, \phi \right) +D(w_{xx},\phi_{xx}) +k_2 (w_{xxt},\phi_{xx}) + \sigma D\big(w_{xx}w_x,w_x\phi_{xx}\big) +  \sigma D(w_{xx}w_x&,w_{xx}\phi_x) 
- \iota \left( w_x \int_x^L u_{tt}, \phi_x \right) \nonumber \\ & = (p,\phi), ~\forall~\phi \in H^2_*.
\end{align}
When $\sigma>0$, the {\bf \small [NL Stiffness]} is in force; similarly, when $\iota>0$, {\bf \small [NL Inertia]} is in force. When $k_2>0$, Kelvin-Voigt damping is imposed. 

We now give precise definitions of solutions making reference to the weak form \eqref{weakform} above:
\begin{definition}\label{sol1}
We say a {\em weak} solution to \eqref{dowellnon*}, with $k_2 = \iota = 0$ and $ \sigma=1$ is a function $w$, with
\begin{equation*}
w \in L^2\left(0,T;H_*^2\right);~w_t \in L^2\left(0,T;L^2(0,L)\right);~w_{tt} \in L^2\left(0,T;[H_*^2]'\right)
\end{equation*}
that satisfies \eqref{weakform}, replacing  $L^2(0,L)$ inner products with $(H^2_*,[H^2_*]')$ duality pairings where necessary.

Moreover, for any $\chi \in H^2_*$, $\psi\in L^2(0,L)$, we require
  \begin{equation}\label{weak-ic}
    (w,\chi)_{H^2_*}\big|_{t\to 0^+} = (w_0, \chi)_{H^2_*},\quad (w_t,\psi)\big|_{t\to 0^+} = (w_1,\psi).
  \end{equation}

\end{definition}

\begin{definition}\label{sol2}
A {\em weak} solution to \eqref{dowellnon*} with $k_2>0$ and $\iota=\sigma=1$ is a function $w$, with
\begin{equation*}
w \in L^2\left(0,T;H_*^2\right); ~ w_t \in L^2\left(0,T;H^2_*\right); ~ w_{tt} \in L^2\left(0,T;[H_*^2]'\right),
\end{equation*}
such that \eqref{weakform} holds, replacing  $L^2(0,L)$ inner products with $(H^2_*,[H^2_*]')$ duality pairings where necessary.

Moreover, for any $\chi \in H^2_*$, $\psi\in L^2(0,L)$, we require
  \begin{equation}\label{weak-ic*}
    (w,\chi)_{H^2_*}\big|_{t\to 0^+} = (w_0, \chi)_{H^2_*},\quad (w_t,\psi)\big|_{t\to 0^+} = (w_1,\psi).
  \end{equation}
\end{definition}

\begin{remark}
For $k_2>0$ and $\iota>0$, the definition of weak solution is self-consistent; this is to say, for such a function $w$, all terms in \eqref{weakform} are well-defined. 
We note that for $k_2=0$, there are complications with the a priori regularity of $w_t \in L^2(0,T;L^2(0,L))$ and the interpretation of the {\bf \small [NL Inertia]} terms. \end{remark}

Now, we define strong solutions as weak solutions with additional regularity.
\begin{definition}\label{strongsol1}
A {\em strong} solution to \eqref{dowellnon*} with $k_2 = \iota = 0$ and $ \sigma=1$ is a weak solution (as in Definition \ref{sol1}) with the additional regularity
$$w \in L^2 \left(0,T; \mathcal D(\mathcal A)\right);~w_t \in L^2(0,T;H^2_*);~w_{tt} \in L^2 \left( 0,T; L^2(0,L) \right).$$ 
\end{definition}

\begin{definition}\label{strongsol2}
A {\em strong} solution to \eqref{dowellnon*} with $k_2>0$, $\iota=\sigma=1$ is a weak solution (as in Definition \ref{sol2}) with the additional regularity
$$w \in L^2 \left(0,T;\mathcal D(\mathcal A) \right);~w_{t} \in L^2 \left( 0,T; \mathcal D(\mathcal A))\right);~w_{tt} \in L^2 \left( 0,T; H^2_{*}  \right).$$
\end{definition}
As we will show below in Corollaries \ref{strongisstrong} and \ref{StrongInertia}, strong solutions will satisfy the pointwise form of the PDE in \eqref{dowellnon*} as well as the higher order boundary conditions at $x=L$.

\section{The Case of Only Stiffness Effects: $\sigma=1$, $\iota=k_2=0$}\label{noinertia}
\subsection{Precise Statement of the Theorem}
\begin{theorem}\label{withoutiota}
Take $\sigma =1$ with $\iota=k_2=0$, and consider $ p \in H^2_{loc}(0,\infty;L^2(0,L))$. For smooth data $(w_0,w_1) \in \mathscr H_s = \mathcal D(\mathcal A)\times H_*^2$, strong solutions exist up to some time $T^*(w_0,w_1,p)$. For all $t \in [0,T^*)$, the solution $w$ is unique and obeys the energy identity
$$E(t) = E(0)+\int_0^t (p,w_t)_{L^2(0,L)}d\tau.$$ 

Restricting to $B_R(\mathscr H_s)$, for any $T<T^*(R,p)$ solutions depend continuously on the data in the sense of $C([0,T];\mathscr H)$ with an estimate on the difference of two trajectories, $z=w^1-w^2$:
$$\sup_{t \in [0,T]}\big|\big|(z(t),z_t(t))\big|\big|_{\mathscr H} \le C(R,T)\big|\big|\big(z(0),z_t(0)\big)\big|\big|_{\mathscr H},~~\forall~t \in [0,T].$$
\end{theorem}
\begin{remark}
The time of existence $T^*$ depends on the data in the sense of $$T^*=T^*\left(||(w_0,w_1)||_{\mathscr H_s}, ||p||_{H^2(0,T;L^2(0,L))}\right),$$ namely, the size of the data in the appropriate space, rather than the individual data itself.
\end{remark}

\subsection{Proof Outline}

We will commence with a Galerkin procedure, using the mode functions $\{s_j\}_{j=1}^{\infty}$ described above. This will yield approximate solutions, with the baseline energy identity providing associated weak limit points. Identifying the nonlinear weak limits is non-trivial, hence, two higher-order multipliers will be used to provide more regular a priori bounds; one is an energy estimate corresponding to the time-differentiated version of the equation, and the other is a stability type estimate resulting from the multiplier $\partial_x^4w$.  Additional compactness is obtained through these estimates with appropriately smooth initial data. With a weak solution in hand corresponding to smooth data, we will show that this strong solution satisfies the PDE pointwise, along with all four cantilever boundary conditions. Lastly, we will tackle the uniqueness and continuous dependence in this case through a particular decomposition of the polynomial structure of the nonlinear stiffness. 

	\subsection{Proof of Theorem \ref{withoutiota}} \label{SectionProofOfStiffness}

\subsubsection{Existence} \label{ExistenceSection}
\begin{proof}[\unskip\nopunct]
Consider the positive eigenfunctions of $\mathcal{A}$ described in Section \ref{modefunx}, with $\lambda_n \rightarrow \infty$; these constitute an {\em orthonormal} basis for $L^2(0,L)$ and {\em orthogonal} basis for any $\mathcal{D}(\mathcal{A}^s)$, $s \in \mathbb R$. Now, for each $n=1,2, \ldots$, we denote
\begin{equation}
\label{SnDefinition}
S_n \equiv \text{span}\{s_1, s_2, \ldots, s_n\}.
\end{equation}

\vskip.1cm
\noindent {\bf Step 1 - Approximants:} For fixed smooth data, $w_0 \in \mathcal{D}(\mathcal{A})$ and $w_1 \in H^2_{*}$, we can construct two approximating sequences $\{w_0^n\}_{n=1}^{\infty}$ and $\{w_1^n\}_{n=1}^{\infty}$ such that
\begin{equation}
\label{sequencedef}
w_{0}^n \coloneqq \sum_{j=1}^{n} \left( w_0, s_j \right) s_j ~\in S_n ~~~\text{ and}~~~ w_{1}^n \coloneqq \sum_{j=1}^{n} \left( w_1, s_j \right) s_j ~\in S_n.
\end{equation}
\begin{align}\label{InitialData}
\text{ By construction:} \hskip1cm&~\hskip1cm~w_0^n  \rightarrow w_0~~\text{in } \mathcal{D}(\mathcal{A}),~~  & ~w_1^n \rightarrow w_1 ~~\text{in } H^2_{*}. \hskip4cm
\end{align}
and we can proceed to define smooth finite-dimensional approximations,
\begin{equation*}
w^n(x, t) \coloneqq \sum_{j=1}^{n} q_j(t) s_j(x),
\end{equation*}
where each $q_j(t)$ is a smooth function of time. 

From the weak form, $(\ref{weakform})$, we construct a corresponding matrix system by taking $\phi=s_j$. We define the following spatial four tensor for ease of writing:
\begin{align}\label{Stensor}
\mathcal{S}_{ijkl}=~(\phi_{i,xx}\phi_{j,xx},\phi_{k,x}\phi_{l,x}).\end{align}
Interpreting $q_is_i$ as a sum, we have the separated form of the equations:
\begin{align}\label{sep1}
\left[q''_{i}(s_i,s_j)\right]+Dq_i\left[\kappa_i^4(s_i,s_j)\right]+Dq_i^3\left[\mathcal S_{iiij}+\mathcal S_{jiii}\right]=(p,s_j),
\end{align} where primes represent $\partial_t$. Initialization is given by
\begin{equation*}
q_{j}(0) =  \left( w_0, s_j \right), ~~~q'_{j}(0) =  \left( w_1, s_j \right), ~~~~j=1,2, \ldots ~ .
\end{equation*}

We may then invoke standard ODE existence and uniqueness for this finite dimensional system. Noting the hypotheses on $p_t,p_{tt}$, we obtain a solution $\{q_j\}_{j=1}^n \in C^3(0,t^*)$, for some small $t^*(n)$.

\vskip.1cm
\noindent\label{Level1}{\bf Step 2 - Energy Level 0:} The  estimate below for \eqref{sep1} on the approximant $w^n$ follows immediately using $w_t$ as the multiplier in the equations \eqref{dowellnon*}--\eqref{dowellnon3*}, taken with $\iota=k_2=0$:
\begin{equation*}
E^n_0(t) = E^n_0(0) + \int_0^t  \left(p , w^n_{t} \right) d \tau ~~\text{for all}~~ t>0,
\end{equation*}
where
\begin{equation}
\label{First Energy}
E^n_0(t) = \frac{1}{2} \left[  ||w^n_{t}||^2 + D||w^n_{xx}||^2 + D||w^n_xw^n_{xx}||^2 \right].
\end{equation}

Now, via Young's inequality and Gr\"{o}nwall applied to \eqref{First Energy}, 
and noting that by \eqref{InitialData} that $\left\{ E^n_0(0)\right\}_{n=1}^{\infty}$ is uniformly bounded in terms of the initial data ~$||(w_0,w_1)||^2_{\mathscr H}$, we obtain:
\begin{equation}
\label{FirstLevel}
E^n_{0}(t) \leq f_0 \left( ||p||_{L^2\left(0,t;L^2(0,L)\right)}, ||w_0||_{\mathcal D(\mathcal A)},||w_1||_{H^2_{*}} \right) e^{t/2} ~~\text{for all}~~ t>0.
\end{equation}
The function $f_0$ is increasing in its arguments.
The estimate in \eqref{FirstLevel} ensures that the time of existence for the approximants, $t^*$, is independent of $n$.

\vskip.1cm
\noindent \label{RemarkStiff}{\bf Step 3 - Boundedness of $w^n_{tt}(0)$:} We will consider $E_{1}(t)$ as the natural ``energy" corresponding to the {\em time}-differentiated version of the stiffness-only equation $(\iota=k_2=0)$. For this calculation it is pivotal to establish boundedness of the sequence $\{w^n_{tt}(0)\}_{n=1}^{\infty}$ in $L^2(0,L)$. To that end, it is true that the following holds for all $\phi \in S_n$, $n=1,2, \ldots$:
\begin{equation}
\label{approximate1}
\left( w^n_{tt} + D \partial^4_x w^n-D\partial_x\left[(w^n_{xx})^2w^n_x \right]+D\partial_{xx}\left[w^n_{xx}(w^n_x)^2\right] - p, \phi \right) = 0.
\end{equation}

We consider $\phi = s_j(x),~j=1,2,\ldots, n$. Then, multiplying \eqref{approximate1} by $q''_j(t)$, summing over the $j's$, and rearranging the terms we obtain: 
\begin{equation}
||w^n_{tt}||^2 =(p, w^n_{tt}) -D(\partial_x^4 w^n,w^n_{tt})+ D(\partial_x{([w^n_{xx}]^2 w^n_x)}, w^n_{tt}) - D(\partial_{xx}{([w^n_x]^2 w^n_{xx})}, w^n_{tt}).  \label{weakfort2}
\end{equation}

Owing to the $C^3$ temporal regularity of $w^n$, we can take $t=0$ in the above expression. Therefore, using (i) the expanded version of \text{\bf  [NL ~Stiffness]} shown in {\em Remark} \ref{quasi}, (ii) the Sobolev embedding into $H^1\hookrightarrow L^{\infty}$, (iii) and Poincar\'{e} for various derivatives, we have: 
\begin{align*}
 ||w^n_{tt}(0)|| 
\lesssim &~ ||p(0)|| +  ||w^n_{xxx}(0)||~||\partial_x^4w^n(0)||^2+||w^n_{xx}(0)||~||w^n_{xxx}(0)||^2+\left( 1 + ||w^n_{xx}(0)||^2 \right) ||\partial_x^4 w^n(0)||.
\end{align*}
The expression on the right-hand side is bounded. Indeed, by \eqref{InitialData}, $||\partial_x^4 w^n(0)|| \lesssim ||w_0||_{\mathcal D(\mathcal A)}$. Moreover, by hypothesis, since $p,p_t \in L^2(0,T;L^2(0,L))$, $||p(0)||$ is interpreted as a temporal trace \cite{evans}, with $||p(0)|| \lesssim ||p||_{H^1(0,T;L^2(0,L))}$. Hence we conclude that 
\begin{equation}
\label{boundednessofwtt}
||w^n_{tt}(0)|| \leq f \left( ||p||_{H^1(0,T;L^2(0,L))}, ||w_0||_{\mathcal D(\mathcal A)} \right).
\end{equation}

\vskip.1cm
\noindent\label{Level2}{\bf Step 4 - Energy Level 1:} Our goal now is to form the $E_1(t)$ energy which will correspond to time differentiation of the stiffness dynamics. We note that time differentiation does not affect the boundary conditions for $w^n(x,t)= \sum_{i=1}^n q_i(t)s_i(x)$. Hence, after proceeding with appropriate integration by parts, isolating conserved quantities, and gathering similar terms, we obtain the a priori identity: \begin{align}\label{energy1}
\frac{1}{2} \frac{d}{dt}\big[  ||w_{tt}||^2 +  D||w_{xxt}||^2 &+D ||w_{xx}w_{xt}||^2 +  D ||w_{xxt}w_{x}||^2\big]  \\ \nonumber
=-\dfrac{d}{dt}\Big[4D(w_{x} w_{xx} , & w_{xt}w_{xxt})\Big]+\left( p_{t}, w_{tt} \right) + 3D(w_{xx} w_{xxt} ,w^2_{xt} ) +3 D ( w_{x}w_{xt},w^2_{xxt}) .
\end{align}

We have omitted the superscript $n$ here and in the estimation below for ease of presentation. The identity above is integrated in time on $(0,t)$, with an eye to utilize a version of Gr\"{o}nwall's inequality. 
\begin{remark} As an a priori estimate, the equality above holds for approximate solutions, which are appropriately smooth; this can be seen by operating directly on the ODE system \eqref{sep1}, differentiating in time, multiplying by $q_j''$ and integrating in time. \end{remark}

Accordingly, we  define the energy $E^n_1(t)$ precisely, corresponding to smoother norms for a solution:
\begin{equation}
\label{secondenerg}
E^n_{1}(t) = \frac{1}{2} \left[ ||w^n_{tt}||^2 +  D||w^n_{xxt}||^2 + D||w^n_{xt}w^n_{xx}||^2 + D||w^n_{x}w^n_{xxt}||^2 \right].
\end{equation} 

Now, we must bound/absorb the unsigned quantities in the energy identity \eqref{energy1} above. We first note some important intermediate inequalities. (We have freely used: Young's inequality, Poincar\'{e}, Sobolev interpolation, and the continuous embedding $H^{1/2^+}(0,L) \hookrightarrow L^{\infty}(0,L)$.)
\begin{enumerate}\setlength\itemsep{.5em}
\item 
$
\!
\begin{aligned}[t]
3D \left| (w_{xx} w_{xxt} ,w^2_{xt} ) \right|
    \leq 3D ||w_{xt}||^2_{L^{\infty}}||w_{xx}||~||w_{xxt}|| 
    &\lesssim  ||w_{xx}||^4 + ||w_{xxt}||^4  \\
\end{aligned}
$

\item 
$
\!
\begin{aligned}[t]
3 D \left| ( w_{x}w_{xt},w^2_{xxt}) \right|
    \leq 3D ||w_{x}||_{L^{\infty}}||w_{xt}||_{L^{\infty}} ||w_{xxt}||^2 
    &\lesssim  ||w_{xx}||^4 + ||w_{xxt}||^4  \\
\end{aligned}
$ 

\item 
$
\!
\begin{aligned}[t]
4D \left| (w_{x} w_{xx} ,  w_{xt}w_{xxt}) \right |
    &\leq\varepsilon_1 ||w_{x}w_{xxt}||^2 + C_{\varepsilon_1} ||w_{xx}w_{xt}||^2 \leq\varepsilon_1 ||w_{x}w_{xxt}||^2 + C_{\varepsilon_1} ||w_{xt}||^2_{L^{\infty}}||w_{xx}||^2. \label{five}
\end{aligned}
$ 
\end{enumerate}

To continue our estimation of $\ref{five}$ above, we interpolate the term $||w_{xt}||^2_{L^{\infty}}$ as follows:
$$ ||w_{xt}||^2_{L^{\infty}} \lesssim ||w_{t}||^2_{3/2 + \epsilon}
    \lesssim  ||w_{t}||^{1/2 - \epsilon}~||w_{xxt}||^{3/2 + \epsilon}. $$
Substituting the above in \ref{five} and then utilizing Young's inequality in the $(p,q)$ setting we obtain:
\begin{align} \label{5cont}
 4D \left| (w_{x} w_{xx} ,  w_{xt}w_{xxt}) \right | & \leq  \varepsilon_1 ||w_{x}w_{xxt}||^2 + C_{\varepsilon_1}||w_{t}||^{1/2 - \epsilon}~||w_{xxt}||^{3/2 + \epsilon}~||w_{xx}||^2\\ 
 & \leq  \varepsilon_1 ||w_{x}w_{xxt}||^2 + C_{\varepsilon_1} C_{\varepsilon_{p}}||w_{t}||^{(1/2 - \epsilon)q}~~||w_{xx}||^{2q} + C_{\varepsilon_1} \varepsilon_{p}||w_{xxt}||^{(3/2 + \epsilon)p}. \nonumber
\end{align}
We choose $p>1$ such that $(3/2 + \epsilon)p =2$. Hence, by fixing $\epsilon=1/4$, we obtain $p=8/7$ and $q=8$. Inequality in \eqref{5cont} becomes:
\begin{align*}
4D \left| (w_{x} w_{xx} ,  w_{xt}w_{xxt}) \right | 
&\leq \varepsilon_1 ||w_{x}w_{xxt}||^2 + C_{\varepsilon_1} C_{\varepsilon_{p}}||w_{t}||^{4} + C_{\varepsilon_1} C_{\varepsilon_{p}}||w_{xx}||^{32} + C_{\varepsilon_1} \varepsilon_{p}||w_{xxt}||^2.
\end{align*}

Choosing $\varepsilon_1$ and $\varepsilon_p$ sufficiently small, we can absorb terms by $E_1^n(t)$ on the LHS of \eqref{energy1}. Thus, using \eqref{InitialData} in passing to the limit on the RHS, and invoking the result from \eqref{boundednessofwtt}, we arrive at the estimate:
\begin{equation}
\label{timediff}
E^n_1(t) \leq f_1\left(p_t, ||w_0||_{\mathcal D(\mathcal A)}, ||w_1||_{H^2_{*}} \right) + f_2\left(p,||w_0||_{\mathcal D(\mathcal A)},||w_1||_{H^2_{*}} \right)t + C \int_0^t \left [ E^n_1(\tau) \right ] ^2d\tau.
\end{equation}

Note that $C>0$ above {\em does not depend} on $w_0, w_1$ or $p$. The functions $f_1$ and $f_2$ are smooth, real-valued functions, increasing in their arguments. {In particular, the function $f_2$  is obtained after we apply \eqref{FirstLevel} to the norms $||w_{xx}||^4, ||w_t||^4$ and $||w_{xx}||^{32}$ that appear on the RHS of the estimates (1)--(4).} Dependence on $p$ we take to mean dependence on the norm $\displaystyle ||p||_{L^2(0,t; L^2(0,L))}$ (mutatis mutandis for $p_t $), as in the previous step.  

Hence, using the {\em nonlinear} version of Gr\"{o}nwall's inequality \cite{gronwall}, we obtain a local-in-time estimate:
\begin{equation}
\label{NonlinGronw}
E^n_1(t) \leq  \frac{ f_1  + f_2t }{1- C \left[ f_1 t + f_2 t^2 \right]} \equiv M_1(t),~ ~~0 \leq t < T^* ~~\text{where}~~ T^* = \sup_{t >0} \left \{ C \left[ f_1 t + f_2 t^2 \right] <1 \right \}. 
\end{equation}
\begin{remark}
Following the assumptions of Theorem \ref{withoutiota}, requiring $p \in H^2_{loc}(0,\infty;L^2)$ is done here since the version of Gr\"{o}nwall  we utilize for \eqref{timediff} requires $f_1$ and $f_2$ to be continuous functions in time. 
\end{remark}

Then, for any fixed $T<T^*$, we have that \eqref{NonlinGronw} constitutes a uniform-in-$n$ a priori bound on $ E_1^n(t)<M_1^*(T),~~t \in [0,T]$, where \begin{equation}M_1^*(T) = \max_{t \in [0,T]} M_1(t);\end{equation} this quantity depends only on fixed norms of the data and $T$. 
\begin{remark}\label{stupiddependencies} It is also important to note that, for a fixed $t$, $M_1(t)$ is an {\em increasing} function in ~$||(w_0,w_1)||_{\mathscr H_s}$ that vanishes when $p=w_0=w_1=0$; this is used for continuous dependence.
\end{remark}

From \eqref{NonlinGronw}, we conclude that the Galerkin approximations satisfy a local-in-time bound by the data on any interval $[0,T]$ with $T<T^*$\footnote{Conversely, given $T>0$, there is a ball of data small in the sense of $\sum E_i(0)$ for which solutions exist up to $T$.}). 

Whenever the initial data $(w_0,w_1) \in \mathscr H$, as well as $p$, are fixed, then $T^*$ is fixed; {\em hence, for the existence portion of the proof of Theorem \ref{withoutiota}, we take $T<T^*$ fixed and consider $t \in [0,T]$.}
\vskip.1cm
\noindent {\bf Step 5 - Additional Spatial Regularity:} Unlike the standard approach, we cannot obtain the needed additional boundedness of $\partial_x^4w$ by going back through the equation (with additional regularity of $w_{tt}$ established). 
To obtain further regularity of solutions, spatial differentiation is used. 
\begin{remark} 
Owing to the high order boundary conditions, one must take care in this process. 
We note energy identities associated with one spatial differentiation result in problematic trace terms that cannot be controlled by the conservative energetic terms. Moreover, as spatial differentiation produces mixed time-space terms, we do not proceed to obtain an energy estimate in this scenario; rather, we utilize an equipartition  multiplier and integrate in space-time, which will provide control of the term
$$||\partial_x^4 w||_{L^2(0,t;L^2(0,L))}^2-||w_{xxt}||^2_{L^2(0,t;L^2(0,L))}$$ the latter term is controlled by the estimate in the previous step. \end{remark} 

To obtain the a priori bound, we multiply the equation by $\partial_x^4w$ and estimate. 
\begin{equation}
\label{spacediff}
\left( w_{tt}, \partial_x^4w \right) + D \left( \partial^4_{x}w, \partial_x^4w \right) - D \left( \partial_{x}[w_{xx}^2 w_x], \partial_x^4w \right) + D \left( \partial^2_{x}[w_{xx}w_x^2], \partial_x^4w \right)= \left( p, \partial_x^4w \right).
\end{equation}
Note that as in Step 3, this can be justified by multiplying the weak ODE form \eqref{sep1} by $\lambda_jq_j$ (see \cite{old}).
 We integrate the above in time on $(0,t)$. 
 For the first term of \eqref{spacediff} we integrate by parts:
\begin{equation*}
\int_0^t \left( w_{tt}, \partial_x^4w \right)= \int_0^t \left( w_{xxtt}, w_{xx} \right) =(w_{xxt},w_{xx})\big|_0^t - \int_0^t ||w_{{xxt}}||^2.
\end{equation*}

For the remaining of the terms in \eqref{spacediff} we identify positive quantities and gather terms. 
\begin{align}
\label{spacetwice}
D \int_0^t &\left[  ||\partial_x^4 w||^2 +  ||w_x \partial_x^4 w||^2 \right] d\tau - \int_0^t  ||w_{{xxt}}||^2d\tau \\
=&~  \int_0^t \left[ \left( p, \partial_x^4 w \right)   - D \left( w^3_{xx}, \partial_x^4 w \right)  - 4D \left( w_{x}w_{xx}w_{xxx}, \partial_x^4 w \right) \right] d \tau   -(w_{xxt},w_{xx})\Big|_0^t .\nonumber
\end{align}
We now bound the expressions that appear on the RHS above.

\begin{enumerate}\setlength\itemsep{.5em}
\item 
$
\!
\begin{aligned}[t]
\left | (w_{xxt},w_{xx})\big|_0^t \right| \le&~ ||w_{xxt}(t)||~||w_{xx}(t)||+||w_{xxt}(0)||~||w_{xx}(0)|| \\
\end{aligned}
$ 

\item $\ds \left| \left( p, \partial_x^4 w \right) \right| \leq~ C_{\varepsilon} ||p||^2 + \varepsilon||\partial_x^4 w||^2$

\item 
$
\!
\begin{aligned}[t]
  D \left|  (w_{xx}^3,\partial_x^4 w) \right| \le &~ \delta||\partial_x^4 w||^2+C_{\delta}||w_{xx}||^6_{L^6}\\[.2cm]
    \le &~ \delta||\partial_x^4 w||^2+C_{\delta}||w_{xx}||_{L^{\infty}}^4||w_{xx}||^2 \le  ~\delta||\partial_x^4 w||^2+C_{\delta}||w_{xx}||_{1/2+\epsilon}^4||w_{xx}||^2\\[.2cm]
      \le &~ \delta||\partial_x^4 w||^2+C_{\delta}\big[||w||_4^{1+2\epsilon}||w_{xx}||^{3-2\epsilon}\big]||w_{xx}||^2 ~~\text{(take } \epsilon=1/4) \\[.2cm]
      \le &~ \delta||\partial_x^4 w||^2+C_{\delta} \delta_p ||w||_4^{(3/2)p} + C_{\delta_1, \delta_p}||w||_2^{(9/2)q} ~~\text{(take } p=4/3)\\[.2cm]
      \le &~ \left( \delta  +C_{\delta} \delta_p \right)||\partial_x^4 w||^2+C_{\delta, \delta_p}||w_{xx}||^{18}  
\end{aligned}
$

\item 
$
\!
\begin{aligned}[t]
  4D \left|  (w_xw_{xx}w_{xxx},\partial_x^4 w) \right| \le & ~ \eta||\partial_x^4 w||^2+C_{\eta}||w_x||^2_{L^{\infty}}||w_{xx}||_{L^{\infty}}^2||w_{xxx}||^2\\[.2cm]
   \le &~ \eta||\partial_x^4 w||^2+C_{\eta}||w_{xx}||^2 \left[||w_{xx}||^{3/2 - \epsilon}||w||_4^{1/2 + \epsilon}\right]\left [  ||w_{xx}||~||w||_4 \right ] ~~\text{(take } \epsilon=1/4)\\[.2cm]
   \le &~\eta ||\partial_x^4 w||^2+C_{\eta} \eta_p ||w||_4^{(7/4)p} + C_{\eta, \eta_p}||w_{xx}||^{(17/4)q} ~~\text{(take } p=8/7)\\[.2cm]
   \le &~ \left( \eta +C_{\eta} \eta_p \right) ||\partial_x^4 w||^2 + C_{\eta, \eta_p}||w_{xx}||^{34}.
\end{aligned}
$
\end{enumerate}

We choose $\varepsilon, \delta, \delta_p, \eta, \eta_p$ so that, upon integration, ~$\int_0^t||\partial_x^4 w||^2d\tau$ is absorbed by the LHS of \eqref{spacetwice}. Hence, by denoting 
\begin{equation*}
V(t) = D \left| \left| \partial_x^4 w \right|\right|^2 + D\big|\big|w_{x}\partial_x^4 w\big|\big|^2 ,
\end{equation*}
and $V^n(t)$ the above functional evaluated on $w^n$, we estimate \eqref{spacetwice} as:
\begin{equation}
\label{spacedifff}
\int_0^t V^n(\tau) d\tau \leq f_3 \left(t,p,E^n_0(0), E^n_1(0),E^n_0(t),  E^n_1(t) \right) ~~~ \text{for all } t \in [0, T],
\end{equation}
where we have invoked the estimates from the previous level \eqref{FirstLevel} and \eqref{NonlinGronw}, and $T<T^*$. Again, $f_3$ is increasing in its arguments, and dependence on $p$ is taken as in the previous sections. 
\begin{remark} Note that \eqref{spacedifff} {\em is not a true energy estimate} in the sense of pointwise-in-time control of an ``energy". The estimate above highlights the need to first close the higher time estimate for solutions in order to use the equipartition approach. \end{remark}

Based on  the boundedness of of $E^n_0(0)$ and $E^n_1(0)$, along with the combination of \eqref{FirstLevel},\eqref{NonlinGronw},  \eqref{spacedifff}, we deduce that 
\begin{equation}
\label{spacedifff1}
\int_0^t V^n(\tau) d\tau \leq f_3 \left(t,p, p_t, ||w_0||_{\mathcal D(\mathcal A)}, ||w_1||_{H^2_{*}}, M_1^*(T) \right) ~~~~\text{for all } t \in [0, T].
\end{equation}

Combining \eqref{NonlinGronw} and \eqref{spacedifff1}, we arrive at the final energy estimate for boundedness of 
\begin{equation}
\label{finalreg}
||w^n||_{L^2(0,T;\mathcal D(\mathcal A))}+||w^n_t||_{L^{\infty}(0,T;\cD(\cA^{1/2}))}+||w^n_{tt}||_{L^{\infty}(0,T;L^2(0,L))} \le C(\text{data},T),
\end{equation}
where $T<T^*$ as in \eqref{NonlinGronw} and the dependence on ``data" is as in the RHS of \eqref{spacedifff1}.
 This bound holds for the associated subsequential weak limit points and provides additional compactness below.
\begin{remark}\label{neededreg}
Denoting $w$ as the function corresponding to the weak/weak-* limit above, we see that $w \in L^2(0,T;H^4(0,L))$ and $w_t \in L^{2}(0,T;H^2_*)$; hence we obtain in the standard way  \cite{evans} the auxiliary bound  for $w \in C([0,T];H^3(0,L))$. \end{remark}

\vskip.1cm
\noindent {\bf Step 6 - Limit Passage and Weak Solution:} With higher a priori bounds in hand for smooth data $w_0 \in \mathcal D(\mathcal A),~w_1 \in H^2_*$, we proceed to pass with the limit and construct a weak solution satisfying \eqref{weakform} with $k_2 = \iota = 0$ and $ \sigma=1$ on any $[0,T]$ for $T<T^*(w_0,w_1,p)$.

From \eqref{finalreg}, Banach-Aloaglu yields existence of a subsequence $\left\{w^{n_{k}}\right\}_{k=1}^{\infty}$ and associated weak limit point \begin{equation}\label{regpoints2} w \in L^2 \left(0, T; H^4(0,L) \right) \cap  H^1 \left(0, T; H^2_{*} \right)\cap H^2 \left(0, T; L^2(0,L) \right), ~\text{ such that} \end{equation} 
\begin{equation}\label{regpoints}w^{n_k} \rightharpoonup w\in L^2 \left(0, T; \mathcal D(\mathcal A)) \right);~~~~
w^{n_k}_{t} \rightharpoonup w_{t} \in L^2 \left(0, T; H^2_{*} \right);~~~~
w^{n_k}_{tt} \rightharpoonup w_{tt} \in L^2 \left(0, T; L^2(0,T) \right),\end{equation}
with compactness of the Sobolev embeddings and Aubin-Lions ensuring strong convergence for $w_{n_k}$ in $L^2(0,T;H^2_*)$.

Now, based on Definition \ref{sol1}, in order to identify $w$ as a {\em weak solution}, it must satisfy the weak formulation \eqref{weakform} with $k_2 = \iota = 0$ and $ \sigma=1$. Identification for linear terms in \eqref{weakform} immediately follows from the above weak convergence, whereas the two [{\bf NL Stiffness}] terms require more attention. For $\phi \in H^2_{*}$, adding and subtracting mixed terms, we obtain (omitting temporal integration):
\begin{align*}
\label{PassLimit1}
\left( [w_{x}^{n_k}]^2 w_{xx}^{n_k} - w^2_{x}w_{xx}, \phi_{xx} \right) & \leq  \left( [w_{x}^{n_k}]^2 (w_{xx}^{n_k} - w_{xx} ) , \phi_{xx} \right) + \left( w_{xx} ([w_{x}^{n_k}]^2 - w^2_{x} ) , \phi_{xx} \right) \\
& \leq  ||w_{x}^{n_k}||^2_{L^{\infty}} \left(w_{xx}^{n_k} - w_{xx} , \phi_{xx} \right) + || \phi_{xx} ||~|| w_{xx}||~||[w_{x}^{n_k}]^2 - w^2_{x}||_{L^{\infty}}  \\
& \leq  ||w_{xx}^{n_k}||^2 \left(w_{xx}^{n_k} - w_{xx} , \phi_{xx} \right) + || \phi_{xx} || ~|| w_{xx}|| ~||w_{x}^{n_k} + w_{x}||_{L^{\infty}}||w_{x}^{n_k} - w_{x}||_{L^{\infty}} \\
& \leq  ||w_{xx}||^2 \left(w_{xx}^{n_k} - w_{xx} , \phi_{xx} \right) + 2~|| \phi_{xx} || ~|| w_{xx}||^3  ||w_{x}^{n_k} - w_{x}||_{1/2 + \epsilon}  \\
& ~~~~\to 0 ~~\text{as}~~ k \to \infty.
\end{align*} 
The above calculation requires no additional regularity of solutions, and follows from bounds at the baseline energy level $E_0$, i.e., $w,w^n \in L^{\infty}(0,T;H^2_*) \cap W^{1,\infty}(0,T;L^2(0,L))$. Below, we isolate the problematic nonlinear difference, and critically use additional regularity gained in the preceding steps.
\begin{align*}
\left( [w_{xx}^{n_k}]^2 w_{x}^{n_k} - w^2_{xx}w_{x}, \phi_{x} \right) & \leq  \left( [w_{xx}^{n_k}]^2 (w_{x}^{n_k} - w_{x} ) , \phi_{x} \right) + \left( w_{x} ([w_{xx}^{n_k}]^2 - w^2_{xx} ) , \phi_{x} \right) \\
& \leq   ||\phi_x||_{L^{\infty}} ||w_{xx}^{n_k}||^2||w_{x}^{n_k} - w_{x} ||_{L^{\infty}} + || \phi_x ||_{L^{\infty}} ||w_{x}||_{L^{\infty}}||[w_{xx}^{n_k}]^2 - w^2_{xx}||  \\
&\leq  || \phi_{xx} ||~||w_{xx}^{n_k}||^2 ||w_{xx}^{n_k} - w_{xx} ||+ || \phi_{xx} ||~||w_{xx}||~||w_{xx}^{n_k} + w_{xx}||_{L^{\infty}}||w_{xx}^{n_k}- w_{xx}||   \\
&\leq  || \phi_{xx} ||~||w_{xx}||^2 ||w_{xx}^{n_k} - w_{xx} || +  2~|| \phi_{xx} ||~||w_{xx}||~||w_{xxx}||~||w_{xx}^{n_k}- w_{xx}||  \\
& ~~~~\to 0 ~~\text{as}~~k \to \infty.
\end{align*}
We emphasize the need for strong convergence for $w_{xx}^{n_k}$ in $L^2(0,T;L^2(0,T))$ obtained through compactness of the higher estimates. 

\begin{remark} Algebraic manipulations of the difference ~{\small $[w_{xx}^{n_k}]^2 w_{x}^{n_k} - w^2_{xx}w_{x}$}~ reveal a clear compactness gap for limit passage at the level of only {\small $||w_{xx}^{n_k}||$} boundedness. An alternative approach for the identification of limit points for ${\small \{[w_{xx}^{n_k}]^2w_x^{n_k}\}_{k=1}^{\infty}}$ (which uniformly bounded in $L^1$), would be to utilize the Dunford-Pettis weak compactness criterion in $L^1$. However, associated multiplier estimates bring about  non-trivial commutators corresponding to the quasilinear nature of {\bf [NL Stiffness]}.\end{remark}

We conclude that the limit point $w$, as above, satisfies the weak formulation \eqref{weakform} with $k_2 = \iota = 0$ and $ \sigma=1$, and is thusly a {\em weak solution}.

\vskip.1cm
\noindent {\bf Step 7 - Strong Solution and Free Boundary Condition:} With a weak solution $w(x,t)$ in hand corresponding to smooth initial data, we have immediately that the solution is strong, by Definition \ref{strongsol1} and the regularity afforded by \eqref{regpoints}. This concludes the proof of Theorem \ref{withoutiota}.
\end{proof}
Naturally, we would like to show that the strong solution constructed above satisfies the PDE pointwisedly, as well as the higher order boundary conditions.
\begin{corollary}\label{strongisstrong} Strong solutions $w(x,t)$ as described in Definition \ref{strongsol1}, satisfy equation \eqref{dowellnon*} with $\sigma=1$ and $\iota=k_2=0$ almost everywhere in space and in time. Additionally, they satisfy the free boundary conditions: $w_{xx}(L,t)=w_{xxx}(L,t)=0$ for all $0 \leq t \leq T$.

\end{corollary}
\begin{proof}[Proof of Corollary \ref{strongisstrong}]
The weak limit $w$ constructed above satisfies:
\begin{equation}
\label{weakform3}
(w_{tt}, \phi)+D(w_{xx},\phi_{xx}) + D(w^2_{xx}w_x, \phi_x) + D(w_x^2 w_{xx}, \phi_{xx})= (p, \phi), ~~~\forall \phi \in H^2_{*},~~a.e.~t.
\end{equation}

Having in hand the regularity given in \eqref{regpoints2}, we undo integration by parts in \eqref{weakform3} evaluated on test functions  to obtain the strong form of the PDE. That is:
\begin{equation*}
\left( w_{tt} + D \partial^4_{x}w -D\partial_x\big[w_{xx}^2w_x \big]+D\partial_{xx}\big[w_{xx}w_x^2\big] - p, \phi \right)=0, ~~~\forall \phi \in C_0^{\infty}(0,L).
 \end{equation*} 
Via density, we have:
\begin{equation}\label{dumdum}
w_{tt} + D \partial^4_{x}w -D\partial_x\big[w_{xx}^2w_x \big]+D\partial_{xx}\big[w_{xx}w_x^2\big] = p ~~~a.e.~x,~~ a.e. ~t,
 \end{equation}
 and thus the PDE in \eqref{dowellnon*} is satisfied $a.e.~x$ pointwisedly for $a.e.~t$. 
 
Since $w \in H^2_*$ by construction, we must verify the free boundary conditions.
Undoing the integration by parts procedure and invoking \eqref{dumdum} results the following boundary terms: 
\begin{equation}
\label{boundary}
\phi_x(L) \left( w_{xx}(L)+w_x^2(L)w_{xx}(L) \right) - \phi(L) \left( w_{xxx}(L) + w_{x}(L) w_{xx}^2(L) +w_x^2(L)w_{xxx}(L) \right)=0, 
\end{equation}
for all  ~$\phi \in H^2_*$, holding $a.e.$ in $t$.
But, as in Remark \ref{neededreg}, $w \in C([0,T]; H^3(0,L))$, and so we can write:
$$\phi(L)(1+w_x^2(L))w_{xxx}(L) = \phi_x(L) \left( w_{xx}(L)+w_x^2(L)w_{xx}(L) \right)-\phi(L)w_{x}(L)w_{xx}^2(L),$$
where the RHS is continuous function of time.
 Now, consider the subclass of $\phi \in H_0^1\cap H_*^2 \subseteq H^2_*$. Then, $$ w_{xx}(L)\left(1 + w_x^2(L) \right)\phi_x(L)=0 ~~\text{for all such}~\phi.$$
 By the surjectivity of the trace theorem, there exists one function so that $\phi_x(L) \neq 0$, and thus $$w_{xx}(L)\left(1 + w_x^2(L)\right)=0 \implies w_{xx}(L)=0.$$
 
  Now, consider $\phi \in H^2_*$. Again, by the surjectivity of the trace theorem, there exists at least one $\phi$ so that $\phi(L)\neq 0$. Using this $\phi$ and the fact that $w_{xx}(L)=0$, \eqref{boundary} yields: $$w_{xxx}(L)\left(1 + w_x^2(L)\right)=0 \implies w_{xxx}(L)=0.$$
  
  Thus, we have verified that the free boundary terms $w_{xx}(L) = w_{xxx}(L)=0$ are satisfied. 
\end{proof}
\begin{remark} It is particularly important that strong solutions remain in $\mathscr H_s \equiv \mathcal D(\mathcal A) \times \cD(\cA^{1/2})$ for data $(w_0,w_1)$ emanating therefrom---namely, exhibiting regularity and satisfying all four boundary conditions. This, for instance, allows us to use Poincar\'e repeatedly on the solution, so, for a strong solution $w$, we have the norm equivalences: ~{\small $||\partial_x^4 w|| ~\sim~||w||_{H^4(0,L)} ~\sim~||w||_{\mathcal D(\mathcal A)}.$}
\end{remark}

\subsubsection{Uniqueness and Continuous Dependence} \label{SectionUniq}
Now, consider two strong solutions, $w$ and $v$ whose difference $z=w-v$ satisfies:
\begin{align}
z_{tt}+ D \partial_{x}^{4}z -D\partial_x\left(w_{xx}^2w_x -v_{xx}^2v_x \right)+D\partial_{x}^2\left(w_{xx}w_x^2-v_{xx}v_x^2 \right) = &~ 0 \label{eq:16},
\end{align}
as well as the strong form of the boundary conditions at $x=0$ and $x=L$ and associated initial conditions $z_0=w_0-v_0$ and $z_1=w_1-v_1$. We consider the dynamics above on $t \in [0,T]$, where $T<T^*=\min\{T^*(w_0,w_1),~T^*(v_0,v_1)\}.$
We multiply $\eqref{eq:16}$ by $z_t$ and integrate over $x\in(0,L)$.

For linear terms we have standard conserved quantities,~
$\ds  ||z_t||^2;~~ D||z_{xx}||^2.$
We now take a closer look at the nonlinear differences. Note that the regularity of strong solutions in Definition \ref{strongsol1} is sufficient---specifically $w_t \in L^2(0,T;H^2_*)$---to permit the calculations below.
\begin{enumerate}\setlength\itemsep{1em}
\item 
\label{Unique1}
$
\!
\begin{aligned}[t]
     \left ( \partial_{x}^2\left[w_{xx} w_x^2\right], z_t \right ) -  \left( \partial_{x}^2\left[v_{xx} v_x^2\right] , z_t \right)
= &~   \left ( w_{xx}w_x^2 - w_x^2 v_{xx}, z_{xxt} \right ) +  \left(w_x^2 v_{xx} - v_{xx}v_x^2 , z_{xxt} \right),
\end{aligned}
$

Examining each of the resulting terms above yields:
\begin{enumerate}[label=(\roman*)]\setlength\itemsep{.5em}
\item 

$
\!
\begin{aligned}[t]
 \left ( w_{xx}w_x^2 - w_x^2 v_{xx}, z_{xxt} \right ) = &~  \left (w_x^2 , z_{xx} z_{xxt} \right )
= ~ \frac{1}{2} \frac{d}{dt} ||w_{x}z_{xx}||^2 -  \left( w_x w_{xt}, z^2_{xx} \right) 
\end{aligned} 
$ 
\item 
$
\!
\begin{aligned}[t]
\left(w_x^2 v_{xx} - v_{xx}v_x^2 , z_{xxt} \right) =&~ \left(v_{xx} \left [ w_x^2  - v_x^2 \right ], z_{xxt} \right) = \left(v_{xx} \left [ w_x  + v_x \right ],z_{x} z_{xxt} \right).
\end{aligned} 
$ 
\end{enumerate}

\item 
\label{Unique2}
$
\!
\begin{aligned}[t]
    - \left( \partial_x\left[w_{xx}^2w_x \right], z_t \right) +  \left( \partial_x\left[v_{xx}^2v_x \right], z_t \right) 
= &  \left( w_{xx}^2w_x - w_{xx}^2v_{x}  , z_{xt} \right) +  \left( w_{xx}^2v_{x} -  v_{xx}^2v_x , z_{xt} \right)
\end{aligned}
$

Like before, we examine each term separately:
\begin{enumerate}[label=(\roman*)]\setlength\itemsep{.5em}
\item 

$
\!
\begin{aligned}[t]
 \left( w_{xx}^2w_x - w^2_{xx} v_{x} , z_{xt} \right)&  =  \left( w^2_{xx} , z_{x} z_{xt} \right) = \frac{1}{2} \frac{d}{dt} ||w_{xx}z_{x}||^2 - \left( w_{xx}w_{xxt}, z^2_{x} \right) 
\end{aligned} 
$ 
\item 
$
\!
\begin{aligned}[t]
\left( w^2_{xx} v_{x} - v_{xx}^2v_x , z_{xt} \right)& =   \left(v_x \left [ w^2_{xx}  - v_{xx}^2 \right ], z_{xt} \right) =  \left(v_x \left [ w_{xx}  + v_{xx} \right ] , z_{xx}z_{xt} \right).
\end{aligned} 
$ 
\end{enumerate}
\end{enumerate}

Combining the linear terms along with \ref{Unique1} and \ref{Unique2} we obtain:
\begin{align}
\label{wversion}
\frac{1}{2} \frac{d}{dt}& \Big [ ||z_t||^2 + D ||z_{xx}||^2 + D ||w_{x}z_{xx}||^2 +D||w_{xx}z_{x}||^2 \Big ] \\ & =  D \left( w_x w_{xt}, z^2_{xx} \right) -  D\left(v_{xx} \left [ w_x  + v_x \right ],z_{x} z_{xxt} \right) + D\left( w_{xx}w_{xxt}, z^2_{x} \right)-D \left(v_x \left [ w_{xx}  + v_{xx} \right ] , z_{xx}z_{xt} \right).\nonumber
\end{align}

The expression above cannot be directly estimated, but we exploit symmetry in the polynomial nature of the nonlinearity by swapping the roles of $w$ and $v$ in the previous calculation (equivalent to subtracting $v$ from $w$), adding the two identities, yielding the (now) symmetric identity: 

\begin{align}
\frac{d}{dt} & \Big [  ||z_{t}||^2 + D||z_{xx}||^2\Big] +\frac{D}{2}\dfrac{d}{dt}\Big[ ||w_{x}z_{xx}||^2 +  ||v_x z_{xx}||^2 +  ||w_{xx}z_{x}||^2 +  ||v_{xx}z_{x}||^2 \Big ] \label{finaluniq} \\ \nonumber
=&~ D \Big( w_x w_{xt}+ v_x v_{xt}, z^2_{xx} \Big) 
+ D\Big( w_{xx}w_{xxt} + v_{xx}v_{xxt}, z^2_{x} \Big) -D\Big(\left [ w_{xx}  + v_{xx} \right ] \left [ w_x  + v_x \right ],z_{x} z_{xxt} + z_{xx}z_{xt} \Big) .
\end{align}

Now, the third term in the above line sees $z$-regularity higher than that appearing in the ``energetic" (i.e., positive, conservative) portion of the identity. The key step is to rewrite this term, moving time derivatives onto individual trajectories treated as coefficients---so as to exploit bounds in higher norms for individual trajectories, as well as the particular quadratic factorization appearing here. 
\begin{align*} D\Big(\left [ w_{xx}  + v_{xx} \right ] \left [ w_x  + v_x \right ],z_{x} z_{xxt} + z_{xx}z_{xt} \Big) =&~ D \frac{d}{dt} \Big( \left( w_{xx}  + v_{xx} \right) \left( w_x  + v_x \right), z_{x} z_{xx}  \Big) \\[.2cm] & - D\Big( \partial_t \left [( w_{xx}  + v_{xx} ) ( w_x  + v_x \right ) ],z_{x} z_{xx} \Big). \end{align*}
\noindent Denote:
\begin{align*}
E(t) = &~ ||z_{t}||^2 + D||z_{xx}||^2+ \frac{D}{2}\Big[ ||w_{x}z_{xx}||^2 +  ||v_x z_{xx}||^2 +  ||w_{xx}z_{x}||^2 +   ||v_{xx}z_{x}||^2 \Big].
\end{align*}
Then, \eqref{finaluniq} becomes upon temporal integration:
\begin{align*}
E(t)=&~E(0)-D\left( \left( w_{xx}  + v_{xx} \right)\left( w_x  + v_x \right), z_{x} z_{xx}  \right)\Big|_0^t \\ &+ \int_0^t\Big[D \left( w_x w_{xt}+ v_x v_{xt}, z^2_{xx} \right) + D\left( w_{xx}w_{xxt} + v_{xx}v_{xxt}, z^2_{x} \right) \\
&+ D\Big(  [ w_{xxt}  + v_{xxt} ] [ w_x  + v_x] ,z_{x} z_{xx} \Big) + D\Big(  [ w_{xx}  + v_{xx} ] [ w_{xt}  + v_{xt} ] ,z_{x} z_{xx} \Big) \Big]d\tau. 
\end{align*}

The RHS terms are estimated in the following way, using the Sobolev embeddings and Poincar\'{e}, with an eye to use Gr\"onwall:
\begin{enumerate}
\item 
\label{Uniqueness1}
$
\!
\begin{aligned}[t]
D \left| \left( w_x w_{xt}+ v_x v_{xt}, z^2_{xx} \right) \right| \leq &~   ||w_x w_{xt}+ v_x v_{xt}||_{L^{\infty}} ||z_{xx}||^2 \\[.2cm]
\lesssim &~  \Big ( ||w_{xx}|| \hspace{1mm} ||w_{xxt}|| + ||v_{xx}|| \hspace{1mm} ||v_{xxt}|| \Big) ||z_{xx}||^2
\end{aligned}
$

\item 
$
\!
\begin{aligned}[t]
 D \left|  \left( w_{xx}w_{xxt} + v_{xx}v_{xxt}, z^2_{x} \right) \right| \leq & ~   ||z_x||^2_{L^{\infty}} \Big( ||w_{xx}|| \hspace{1mm} ||w_{xxt}|| + ||v_{xx}|| \hspace{1mm} ||v_{xxt}|| \Big ) \\[.2cm]
\lesssim &~ \Big ( ||w_{xx}|| \hspace{1mm} ||w_{xxt}|| + ||v_{xx}|| \hspace{1mm} ||v_{xxt}|| \Big) ||z_{xx}||^2
\end{aligned}
$

\item 
$
\!
\begin{aligned}[t]
 D \left| \left( [ w_{xxt}  + v_{xxt} ] [ w_x  + v_x ], z_{x} z_{xx} \right) \right | \leq &~  ||w_x + v_x||_{L^{\infty}} \hspace{1mm} ||z_x||_{L^{\infty}} \hspace{1mm} ||w_{xxt} + v_{xxt}|| \hspace{1mm} ||z_{xx}|| \\[.2cm]
\lesssim &~  \left( ||w_{xx}|| + ||v_{xx}|| \right) \left( ||w_{xxt}|| + ||v_{xxt}||\right) ||z_{xx}||^2
\end{aligned}
$

\item 
$
\!
\begin{aligned}[t]
D \left| \left( [ w_{xx}  + v_{xx} ][w_{xt}  + v_{xt} ] , z_{x} z_{xx} \right) \right| \leq &~   ||w_{xx} + v_{xx}|| \hspace{1mm}||z_x||_{L^{\infty}} \hspace{1mm} ||w_{xt} + v_{xt}||_{L^{\infty}} \hspace{1mm} ||z_{xx}|| \\[.2cm]
\lesssim &~ \left( ||w_{xx}|| + ||v_{xx}|| \right) \left( ||w_{xxt}|| + ||v_{xxt}|| \right)  ||z_{xx}||^2
\end{aligned}
$
\item 
\label{Uniqueness5}
$
\!
\begin{aligned}[t]
D \left| \left([w_{xx}+v_{xx}][w_x+v_x]z_x,z_{xx} \right) \right| \leq &~C_{\varepsilon_1}(w,v)||z_x||^2+\varepsilon_1 ||z_{xx}||^2  \\[.2cm]
\lesssim &~   C_{\varepsilon_1}(w,v)||z||\hspace{1mm}||z_{xx}||+\varepsilon_1||z_{xx}||^2 \\[.2cm]
\le &~ C_{\varepsilon_1,\varepsilon_2}(w,v) ||z||^2+(\varepsilon_1+\varepsilon_2)||z_{xx}||^2 \\[.2cm]
\lesssim & ~C_{\varepsilon_1,\varepsilon_2}(w,v) \left[\int_0^t||z_t||^2d\tau+||z(0)||^2\right]+(\varepsilon_1+\varepsilon_2)||z_{xx}||^2,
\end{aligned}
$ 

where above we have used interpolation and $H^2_*$ norm equivalence in the second inequality, and the fundamental theorem of calculus in the fourth line. The dependence of $C$ above is in the sense that $C(w,v) \equiv C\left(||w_{xxx}+v_{xxx}||\hskip1mm ||w_{xx}+v_{xx}||\right)\le  C\left(\ds \sup_{0 \le t \le T}\left[ ||w(t)||_3^2+||v(t)||^2_3\right]\right)$.
\end{enumerate}

Thus, choosing $\varepsilon_1,\varepsilon_2$ sufficiently small, and putting \ref{Uniqueness1}--\ref{Uniqueness5} together, we obtain:
\begin{align}\label{constanthere}
E(t) \le c (1+C(w,v)) E(0)+C(w,v)\int_0^tE(\tau)d\tau+\int_0^t K(w,v) E(\tau)d\tau.
\end{align}
We again note the dependence of $K(w,v)$ in the sense of:
$$K(w,v)\equiv K\big(||w_{xx}+v_{xx}||\hskip1mm||w_{xxt}+v_{xxt}|| \big)\le K\big(||w||_2^2,||w_t||_2^2,||v||_2^2,||v_t||_2^2\big).$$
The constant $c$ in \eqref{constanthere} does not depend on the initial data, nor the trajectories $w,v$.

Finally, we note the $C([0,T])$ boundedness (for $T<T^*(\text{data}_w,\text{data}_v)$) of the quantities $C(w,v),~K(w,v)$ from the regularity of strong solutions, along with Remark \ref{neededreg} on the individual trajectories, $(w,w_t),~(v,v_t)$.  Taking $\sup_{[0,T]}$, we obtain:
$$E(t) \le \mathcal C_1 E(0)+ \mathcal C_2 \int_0^t E(\tau)d\tau,$$
where $t \in [0,T]$ and we have the dependencies $\ds \mathcal C_i\Big(||(w_0,w_1)||_{\mathscr H_s}, ||(v_0,v_1)||_{\mathscr H_s},||p||_{H^{1}(0,T;L^2(0,L))} \Big)$.

The standard Gr\"{o}nwall lemma yields:
\begin{equation}\label{postgron}
E(t) \le \mathcal C_1 E(0) e^{\mathcal C_2 t},~~t \in [0,T].
\end{equation}

Uniqueness of strong solutions follows immediately, since if $(w_0,w_1)=(v_0,v_1)$, the times of existence are identified and $E(0)=0$ gives $z=0$ in the sense of $L^2(0,T;L^2(0,L))$ for all valid $T$.

Continuous dependence also follows from \eqref{postgron}, but is somewhat more subtle. Upon inspection, the constants above $\mathcal C_i$ are continuous, real-valued, positive functions of their arguments. Namely, the $\mathcal C_i(\cdots),~i=1,2$ are bounded when restricting to $\overline{B_R(\mathscr H_s)}$---see Remark \ref{stupiddependencies}. Hence, for  $(w_n,w_{n,t}),(w,w_t) \in \overline{B_R(\mathscr H_s)}$  we see that $z_n = w - w_n$ has the property that
$$(z_n(0),z_{n,t}(0)) \to (0,0)  \in \mathscr H~\implies~ (z_n,z_{n,t}) \to (0,0) \in C([0,T];\mathscr H).$$

\section{The Case with Nonlinear Inertia: $\sigma=\iota=1,$ $k_2>0$}\label{inertia}
\subsection{Precise Statement of the Theorem}
\begin{theorem}\label{withiota}
Take $\sigma=\iota=1$ and $k_2>0$, and consider $p \in H^3_{loc}(0,\infty;L^2(0,L))$. For ~initial data $(w_0, w_1)  \in  \mathcal D(\mathcal A^2)^2$, strong solutions exist up to some time $T^*(w_0,w_1,p)$ and are unique on their existence interval. For all $t \in [0,T^*)$, a solution obeys the energy identity
$${E}(t)+k_2\int_0^t||w_{xxt}||^2_{L^2(0,L)} = {E}(0)+\int_0^t (p,w_t)_{L^2(0,L)} d\tau,$$
where $E(t)$ is as in \eqref{energiesdef} with $\sigma=\iota=1$. 

Restricting to $B_R(\cD(\cA^2)^2)$, for any $T<T^*(R,p)$ solutions depend continuously on the data in the sense of $C([0,T];\mathscr H)$ with an estimate on the difference of two trajectories, $z=w^1-w^2$:
$$\sup_{t \in [0,T]}\big|\big|(z(t),z_t(t))\big|\big|_{\mathscr H} \le C(R,T)\big|\big|\big(z(0),z_t(0)\big)\big|\big|_{\mathscr H},~~\forall~t \in [0,T].$$
\end{theorem}

\begin{remark}
The dependence $T^*=T^*\big(||(w_0,w_1)||_{\mathcal D(\mathcal A^2) \times \mathcal D(\mathcal A^2)}, ||p||_{H^3(0,T;L^2(0,L)}\big)$.
\end{remark}

\subsection{Proof Outline}
For this proof we utilize a modified strategy from the previous section, as the presence of inertia (and damping) change the sequence of multipliers. Indeed, with the addition of damping (as per the discussion above), we can obtain a sequence of true energy estimates at various levels, and again exploit the techniques in the proof of the {\em Theorem} \ref{withoutiota} after closing estimates. Due to the structure of {\bf [NL Inertia]}, even in the presence of velocity-regularizing Kelvin-Voigt damping, further additional regularity (hence higher estimates) will be needed in the construction of solutions and their uniqueness. 
\subsection{Proof of Theorem \ref{withiota}}
\subsubsection{Existence} \label{SectionExistIne} The setup here is the same as in Section \ref{ExistenceSection}. Since the inertial term and damping ($\iota=1$ and $k_2>0$) are additional terms to the stiffness equation, we will proceed through the relevant calculations corresponding only to {\bf [NL Inertia]} ({\bf [NL Stiffness]} calculations are unchanged). {\em Kelvin-Voigt damping} appears in the final estimates with no discussion, owing to its linearity. 

\noindent {\bf Step 1 - Approximants:} Again, consider smooth data, $w_0 \in \mathcal{D}(\mathcal{A}^2)$ and $w_1 \in \mathcal{D}(\mathcal{A}^2) $, and take Fourier partial sums as  $\{w_0^n\}_{n=1}^{\infty}$ and $\{w_1^n\}_{n=1}^{\infty}$. Then, as before, we have:
\begin{align}\label{InitialData2} 
w_0^n ~ \rightarrow w_0~~\text{in } \mathcal{D}(\mathcal{A}^2);~~ w_1^n  \rightarrow w_1 ~~\text{in } \mathcal{D}(\mathcal{A}^2) 
\end{align}
and
\begin{equation*}
w^n(x, t) \coloneqq \sum_{j=1}^{n} q_j(t) s_j(x),
\end{equation*}
for $q_j(t)$ smooth functions of time. Throughout this section we freely use~ $u^n=-(1/2)\int_0^x[w^n_x]^2d\xi$.

From the weak form, $(\ref{weakform})$ (this time taken with $\iota=1$ and $k_2>0$), we construct the corresponding matrix system using the tensors $\mathcal S_{ijkl}$ from \eqref{Stensor} and
\begin{align}\label{Itensor}
\mathcal I_{ijkl}=&~\left(\int_0^x\phi_{i,x}\phi_{j,x},\int_0^x\phi_{k,x}\phi_{l,x}\right).
\end{align}

\begin{remark}
The following calculation for the inertial tensor connects $\mathcal I_{ijkl}$ back to the weak form \eqref{weakform}: {\small
\begin{align*}
\mathcal I_{ijkl} 
 = -\int_0^L \left[\left(\partial_x\int_x^L \int_0^{\xi}\phi_{i,x}\phi_{j,x}d\xi_2d\xi\right)\int_0^x\phi_{k,x}\phi_{l,x}d\xi\right]dx 
  =~\int_0^L \left[\left(\int_x^L \int_0^{\xi}\phi_{i,x}\phi_{j,x}d\xi_2d\xi\right)\phi_{k,x}\phi_{l,x}\right]dx.
\end{align*}}
\end{remark}

Analogously to \eqref{sep1}, we then have the separated form of the ODE system: {\small
\begin{align}
q''_{i}(s_i,s_j)+\left[q_i''(q_i)^2+(q_i')^2q_i\right]\mathcal I_{iiij}+k_2q_i'\left[\delta_i^4(s_i,s_j)\right]+Dq_i\left[k_i^4(s_i,s_j)\right]+Dq_i^3\left[\mathcal S_{iiij}+\mathcal S_{jiii}\right]=(p,s_j).
\label{sep2}
\end{align}}
Although this ODE system is not an evolution (it is quasilinear in time), it is polynomially nonlinear in the $q_i$'s. Thus, via the implicit function theorem, we have {\em local} solvability for $q_i''$ in terms of the other quantities and lower order terms in $q$. Therefore, local-in-time, there are $C^4(0,t^*(n))$ solutions, again noting the regularity assumption on $p$. 

\vskip.1cm \noindent{\bf Step 2 - Energy Level 0:} For this step we examine the inertial term that corresponds to {\bf Level 0} which was described in {\bf Step 2} of Section \ref{ExistenceSection} for the stiffness-only equation ($\iota=k_2=0$).
\begin{align*}
\left( \partial_x \left [ w_x \int_x^L u_{tt} \right ] , w_{t} \right ) =  - \left(  \int_x^L u_{tt},w_x w_{xt} \right ) = & - \left( \int_x^L u_{tt}, \partial_{x} \int_0^x  w_x w_{xt} \right ) 
= \left( u_{tt}, u_{t} \right ) = \displaystyle \frac{1}{2} \frac{d}{dt} ||u_{t}||^2.
\end{align*}

Denote ~
$ \ds \mathcal{E}^n_0(t) = E^n_0(t) + I^n_{0}(t) \ge 0$, where ~$\ds I^n_{0}(t) = \frac{1}{2} ||u^n_{t}||^2$ and $E^n_{0}(t)$ is as in \eqref{First Energy}. Estimating conservatively, we have:
\begin{equation}
\mathcal{E}^n_0(t) + k_2 \int_0^t ||w^n
_{xxt}||^2 d \tau \leq \mathcal{E}^n_0(0) + \frac{1}{2} \int_0^t ||p||^2 + \frac{1}{2} \int_0^t \mathcal{E}^n_0(\tau) d \tau ~~\text{for all}~~ t>0.
\end{equation}

From \eqref{InitialData2} and $u_{t} = - \int_0^x w_{x}w_{xt}$, so ~$||u_t|| \lesssim ||w||_{\cD(\cA^{1/2})}||w_t||_{\cD(\cA^{1/2})}$. It is immediate that $\left\{ \cE^n_0 (0) \right\}_{n=1}^{\infty}$ is uniform-in-$n$ controlled by ~$||(w_0,w_1)||_{\cD(\cA^{1/2})}$. Hence, the standard Gr\"{o}nwall inequality yields:
\begin{equation}
\label{FirstLevel2}
\mathcal{E}^n_{0}(t) \leq g_0 \left( p, ||w_0||_{\mathcal D(\mathcal A)},||w_1||_{H^2_{*}} \right) e^{t/2} ~~\text{for all}~~ t>0.
\end{equation}
The function $g_0$ is analogous as that described in \eqref{FirstLevel}.

\vskip.1cm
\noindent {\bf Step 3 - Uniform Boundedness of Initial Inertia}: To utilize the additional a priori bound described in the next step, we need  the quantity {\small $||w_{tt}^n(0)||^2+||u_{tt}^n(0)||^2$} to be uniformly bounded by appropriate norms on $w_0$ and $w_1$.
Our proof of uniform $L^2(0,L)$ boundedness {\small $\{w^n_{tt}(0)\}_{n=1}^{\infty}$} in {\bf Step 3} in the proof of {\em Theorem} \ref{withoutiota} {\em cannot be invoked for this calculation}, since additional terms now appear in the equation for $\iota=1,~k_2>0$.

From the equation, approximate solutions  satisfy the relation
\begin{align}
||w^n_{tt}||^2   + D(\partial_x^4 w^n,w^n_{tt})+k_2\left( \partial_x^4 w^n_t, w^n_{tt} \right)- D(\partial_x{([w^n_{xx}]^2 w^n_x)}, w^n_{tt})\nonumber  + D(\partial_{xx}{([w^n_x]^2 w^n_{xx})}, w^n_{tt}) \\ + \left( \partial_x \left [ w^n_x \int_x^L u^n_{tt} \right ] , w^n_{tt} \right) 
= (p, w^n_{tt}).  
\end{align}
Examining the inertial term:
\begin{align*}
\left( \partial_x \left [ w^n_x \int_x^L u^n_{tt} \right ] , w^n_{tt} \right) 
=& - \left(    u^n_{tt}  , \int_0^x w^n_x w^n_{xtt} \right) =  ||u^n_{tt}||^2 + \left(    u^n_{tt}  , \int_0^x \left[ w^n_{xt} \right]^2 \right),
\end{align*}
where we used the expansion of $u_{tt}$ in terms of $w$ as in \eqref{NLInertia}.
Combining everything, we have the identity:
\begin{align}
||w^n_{tt}||^2 + ||u_{tt}^n||^2 =&~  (p, w^n_{tt})- \left(    u^n_{tt}  , \int_0^x \left[ w^n_{xt} \right]^2 \right) -k_2 \left( \partial^4_x w_t^n, w^n_{tt} \right)  \label{weakfort3} \\ \nonumber
&- D(\partial_x^4 w^n,w^n_{tt})+ D(\partial_x{([w^n_{xx}]^2 w^n_x)}, w^n_{tt})- D(\partial_{xx}{([w^n_x]^2 w^n_{xx})}, w^n_{tt}) 
.  
\end{align}

Since approximate solutions (and $p$) are continuous in time, we take the time-trace at $t=0$ in \eqref{weakfort3} and use Young's inequality to obtain the estimate:
\begin{align*}
||w^n_{tt}(0)||^2 + ||u_{tt}^n(0)||^2  \leq &   ~  \delta||u_{tt}^n(0)||^2+ c_{\delta} ||w^n_{xxt}(0)||^4 +  \varepsilon||w^n_{tt}(0)||^2  \\
 &+c_{\varepsilon} \Big[  ||p(0)||^2+ ||\partial_x^4w^n_t(0)||^2 +||\partial_x^4w^n(0)||^4 ||w^n_{xxx}(0)||^2   \\
 &+||w^n_{xx}(0)||^2||w^n_{xxx}(0)||^4+ \big( 1 + ||w^n_{xx}(0)||^4 \big)  ||\partial_x^4 w^n(0)||^2 \Big] .
\end{align*}

Choosing sufficiently small $\delta$ and $\varepsilon$, and using \eqref{InitialData2}, we can finally conclude that
\begin{equation}
\label{uttbound}
||w_{tt}^n(0)||^2+||u_{tt}^n(0)||^2 \le C\left(||w_0||_{\mathcal D(\mathcal A)}, ||w_1||_{\mathcal D(\mathcal A)}, p(0)\right). 
\end{equation}
This fact will be used below in the next energy level.

\vskip.1cm \noindent
{\bf Step 4 - Energy Level 1:} \label{Level2FULL} In this step we proceed with examining the inertial term from the {\em Energy Level 1} estimate described in {\bf Step 3} of Section \ref{ExistenceSection}. Our aim is to control the conserved quantity {\small $||u_{tt}||^2$}, corresponding to a (formal) time differentiation of the equations. 
 Differentiating the inertial term in time and multiplying by $w_{tt}$ we form:
\begin{align*}
\left ( \partial_{xt} \left [ w_x \int_x^L u_{tt} \right ] \ , w_{tt} \right) = 
& - \left (  w_{xt} \int_x^L u_{tt}  , w_{xtt} \right ) - \left (  w_{x} \int_x^L u_{ttt}  , w_{xtt} \right ) \\
\equiv & \hspace{2cm} \mathcal{I}_1 ~~~~~~~~~~+~~~~~~~~~~~\mathcal{I}_2.
\end{align*}

For $\mathcal{I}_2$, we have ~$\ds \mathcal{I}_2= - \left (\int_x^L u_{ttt}  ,w_{x} w_{xtt}\right)
= - \left (u_{ttt} ,\int_0^x w_{x} w_{xtt}\right).$ 
Recalling $~\displaystyle u_{tt}(x)=-\int_0^x \left [w_{xt}^2+w_xw_{xtt} \right ]d\xi$, 
we obtain:
\begin{align*}
\mathcal{I}_2 
= ~\left (u_{ttt}, u_{tt} \right ) + \left (  u_{ttt}  ,\int_0^x w^2_{xt}  \right )  
=~ \frac{1}{2} \frac{d}{dt} ||u_{tt}||^2 +  \frac{d}{dt} \left ( u_{tt}, \int_0^x w^2_{xt} \right ) -  2 \left ( u_{tt}, \int_0^x w_{xt}w_{xtt} \right ).
\end{align*}
The second term above will be estimated so that it can be absorbed by pointwise-in-time conserved quantities. The third term is identical to $\mathcal{I}_1$, and
\begin{align*}
\mathcal{I}_1 = & \left (  \int_0^x \left [w_{xt}^2+w_xw_{xtt} \right ]  , \int_0^xw_{xt} w_{xtt} \right ) 
= ~ \frac{1}{4} \frac{d}{dt} \left | \left| \int_0^x w^2_{xt}\right | \right |^2 +\left (  \int_0^x w_xw_{xtt}   ,\int_0^x w_{xt} w_{xtt} \right ).
\end{align*}

Combining these calculations, we obtain:
\begin{equation}
\label{eq:14}
\frac{d}{dt} \left [ \frac{1}{2} ||u_{tt}||^2 + \frac{3}{4} \left | \left | \int_0^x w^2_{xt} \right | \right |^2 + \left (u_{tt}, \int_0^x w^2_{xt} \right ) \right ] = - 3\left (  \int_0^x w_xw_{xtt}   ,\int_0^x w_{xt} w_{xtt} \right ).
 \end{equation} 
Utilizing once more the approximate inextensibility relation, we can rewrite:
\begin{equation*}
\frac{d}{dt} \left [ \frac{1}{2} ||u_{tt}||^2 + \frac{3}{4} \left | \left | \int_0^x w^2_{xt} \right | \right |^2 + \left (u_{tt}, \int_0^x w^2_{xt} \right ) \right ] =3 \left ( u_{tt}  , \int_0^x w_{xt}w_{xtt} \right ) + 3\left (  \int_0^x w^2_{xt}   ,\int_0^x w_{xt} w_{xtt} \right ).
\end{equation*} 

%

\noindent Poincar\'e and the Sobolev embedding into $L^{\infty}$ yields:
\begin{align}\label{theonetalking}
\frac{d}{dt} \left [ \frac{1}{2} ||u_{tt}||^2 + \frac{3}{4} \left | \left | \int_0^x w^2_{xt} \right | \right |^2 + \left (u_{tt}, \int_0^x w^2_{xt} \right ) \right ] \leq  C_{\varepsilon_1} & \left [ ||u_{tt}||^4 +  \left | \left | \int_0^x w^2_{xt} \right | \right |^4 \right ] + C_{\varepsilon_2} ||w_{xxt}||^4 
 \\& +  \left(\varepsilon_1 + \varepsilon_2 \right)  \vspace{1mm}||w _{xxtt}||^2.\nonumber
\end{align} 

For the unsigned, conservative term on the LHS we utilize Young's inequality with precise coefficients:
$$\left | \left(u_{tt},\int_0^xw_{xt}^2\right) \right| \le~ \frac{3}{8}||u_{tt}||^2+\frac{2}{3}\left|\left|\int_0^xw_{xt}^2\right|\right|^2, $$
which is sufficient for absorption on the LHS of \eqref{theonetalking}.

Now, let us introduce more notation for the estimate resulting from the above formal calculations:~ \begin{equation}\label{higherenergy1} I^n_1(t) = \frac{1}{2} ||u^n_{tt}||^2 + \frac{3}{4} \left | \left | \int_0^x \left[ w^2_{xt} \right ] ^n \right | \right |^2 ~~~\text{and}~~~ \mathcal{E}^n_1(t) = E^n_1(t) + I^n_{1}(t),\end{equation}
with $E^n_1(t)$ given in the stiffness analysis by \eqref{secondenerg}. Compiling everything together and absorbing damping terms on the RHS, we then have that the approximate solutions $w^n$ satisfy:
\begin{align}
\label{InertiaWttEstimate}
\cE^n_1(t) + k_2 \int_0^t ||w^n_{xxtt}||^2 \leq &~  g_1\left(p_t, ||w_0||_{\mathcal D(\mathcal A)}, ||w_1||_{\mathcal D(\mathcal A)} \right) + g_2\left(p,||w_0||_{\mathcal D(\mathcal A)},||w_1||_{H^2_{*}} \right)t  \\  \nonumber   &+ C \int_0^t \left [ \cE^n_1(\tau) \right ] ^2d\tau.
\end{align}
The dependencies for $g_1$ and $g_2$ follow after the application of \eqref{FirstLevel2} and \eqref{uttbound}. Note that $C>0$ here {\em does not depend} on $w_0, w_1$ or $p$. Dependence on $p$ is taken in the sense of \eqref{timediff}.

\vskip.1cm\noindent
{\bf Step 5 - Energy Level 2:} \label{Level3Full} In contrast to what was done in the stiffness-only estimate for {\bf Step 5} of Section \ref{SectionProofOfStiffness}, we proceed to obtain an actual {\em energy estimate} for higher spatial regularity. Indeed, the inclusion of the strong damping $k_2>0$ allows us to consider improved regularity of the solution by employing the multiplier $\partial_x^4 w_{t}$, not permissible when $k_2=0$. Thus the calculations for the from Section \ref{SectionProofOfStiffness} are modified below.

We proceed by multiplying the equation by $\partial_x^4 w_{t}$ and spatially integrating, with appropriate integration by parts. Here it is important to take note of the boundary conditions associated to eigenfunctions in Section \ref{modefunx} and hence to approximants $w^n$ and all of their time derivatives as well. 

Isolating conserved quantities and gathering terms yields:
\begin{align*}
\frac{1}{2}  \frac{d}{dt}& \Big[ ||w_{xxt}||^2 +   D||\partial_x^4 w||^2   +  D ||w_x \partial_x^4 w||^2 \Big ] +k_2||\partial_x^4w_t||^2 \\ =&~  \left( p, \partial_x^4 w_t \right)-4D \left( w_{x}w_{xx}w_{xxx}, \partial_x^4 w_t \right) -   D \left( w^3_{xx}, \partial_x^4 w_t \right) 
-\left( \partial_{x} \left [ w_{x} \int_x^L u_{tt} \right ] , \partial_x^4 w_{t} \right).
\end{align*}

We first estimate quantities associated with stiffness using (as before) interpolation, the Sobolev embeddings, and Young's inequality:
\begin{enumerate}

\item 
$
\!
\begin{aligned}[t]
 4D \left| \left( w_{x}w_{xx}w_{xxx}, \partial_x^4 w_t \right) \right|
& \leq ~ C_{\delta_1}  ||w_{xx}||^8 +  C_{\delta_1} ||\partial_x^4 w||^4 + \delta_1||\partial_x^4 w_t||^2
\end{aligned}
$ 

\item 
\label{StiffnessofInertia}
$
\!
\begin{aligned}[t]
    D  \left| \left( w^3_{xx}, \partial_x^4 w_t \right) \right|
    \leq &~ C_{\delta_2} ||w_{xx}||^4_{L^{\infty}} ||w_{xx}||^2 + \delta_2||\partial_x^4 w_t||^2 \leq  C_{\delta_2} \left(||w_{xx}||^{16/3}_{L^{\infty}} + ||w_{xx}||^{8} \right)  + \delta_2||\partial_x^4 w_t||^2 .
\end{aligned}
$
\end{enumerate}
where we have used Young's Inequality $p = 4/3$ and $q=4$. Subsequently, we interpolate $||w_{xx}||^{16/3}_{L^{\infty}}$ as:
\begin{align}
\label{InterpArg2}
||w_{xx}||^{16/3}_{L^{\infty}} \leq ||w_{xx}||^{16/3}_{1/2 + \epsilon} \leq ||w_{xx}||^{10/3} ||w_{xx}||_2^{2} \lesssim ||w_{xx}||^{20/3} + ||\partial_x^4 w||^4 ,
\end{align}
where we chose $\epsilon = 1/4$ and used Young's inequality again with $p=2$ and $q=2$.

According to the above, we introduce the notation:
\begin{equation}
\label{secondFullStiffness}
E^n_2(t) = ||w^n_{xxt}||^2 +   D||\partial_x^4 w^n||^2   +  D ||w^n_x \partial_x^4 w^n||^2.
\end{equation}

We now estimate the {\em inertial} contribution above, aiming to control the term ~$||u_{xxt}||^2$: 
\begin{align*}
\left( \partial_{x} \left [ w_{x} \int_x^L u_{tt} \right ] , \partial_x^4 w_{t} \right) = & \left(  w_{xx} \int_x^L u_{tt},\partial_x^4 w_{t} \right)- \left( w_{x}u_{tt} , \partial_x^4 w_{t} \right)  \\
\equiv& ~~~~~~~~~~~~\mathcal{J}_1 ~~~~~~~~~~~~+~~~~~~~~~\mathcal{J}_2.
\end{align*}
We can directly bound $\mathcal{J}_1$ as follows:
\begin{align*}
 \left| \mathcal{J}_1 \right| \leq  C_{\delta_3} ||w_{xx}||^2_{L^{\infty}}||u_{tt}||^2 + \delta_3 ||\partial_x^4 w_{t}||^2 
\leq  C_{\delta_3} ||\partial_x^4 w||^4 + C_{\delta_3}||u_{tt}||^4 + \delta_3 ||\partial_x^4 w_{t}||^2.
\end{align*}
For $\mathcal{J}_2$, we note that $
u_{xxt} = - w_{x}w_{xxt} - w_{xx}w_{xt}$, and use this expression to integrate by parts twice:
\begin{align} 
\label{Gexpression}
\mathcal{J}_2 
=& - \left( u_{tt}, w_{xxx}w_{xxt} \right) - 2 \left( u_{xxt}, w_{xx}w_{xxt} \right) + \frac{1}{2} \frac{d}{dt} ||u_{xxt}||^2 + \left( u_{xxtt}, w_{xx}w_{xt} \right). 
\end{align}

We estimate the remaining unsigned terms:
\begin{enumerate}
\item 
$
\!
\begin{aligned}[t]
   - 2 \left( u_{xtt}, w_{xx}w_{xxt} \right) = 2 \left( u_{tt}, w_{xxx}w_{xxt} \right) + 2 \left( u_{tt}, w_{xx}w_{xxxt} \right),\\
\end{aligned}
$ 

we can combine the first term on the RHS with the first in \eqref{Gexpression}, then control each term:
\begin{enumerate}[label=(\roman*)]
\item 
$ \left| \left( u_{tt}, w_{xxx}w_{xxt} \right) \right| \lesssim ||u_{tt}||^4 +  ||\partial_x^4 w||^4+ ||w_{xxt}||^4$ \\
(where we used Young's inequality with 1 for the $u_{tt}$ term)
\item 
$
\!
\begin{aligned}[t]
  \left|  \left( u_{tt}, w_{xx}w_{xxxt} \right) \right| = \left| \left( w_{xx} u_{tt}, w_{xxxt} \right)\right|
   \leq & ~C_{\delta_4} ||\partial_x^4 w||^4+ C_{\delta_4}||u_{tt}||^4 +  \delta_4 ||\partial_x^4 w_{t}||^2.
\end{aligned}
$ 
\end{enumerate}

\item 
$
\!
\begin{aligned}[t]
   \left( u_{xxtt}, w_{xx}w_{xt} \right) = \left( \partial_t [ u_{xxt} ] , w_{xx}w_{xt} \right) = \frac{d}{dt} \left( u_{xxt}, w_{xx}w_{xt} \right) - \left( u_{xxt}, w_{xxt}w_{xt} \right) - \left( u_{xxt}, w_{xx}w_{xtt} \right),
\end{aligned}
$ \\
where each term is bounded as follows:
\begin{enumerate}[label=(\roman*)]
\item 
$
\!
\begin{aligned}[t]
   \left| \left( u_{xxt}, w_{xxt}w_{xt} \right) \right| \lesssim ||u_{xxt}||^4 + ||w_{xt}||^2_{L^{\infty}} ||w_{xxt}||^2 \lesssim ||u_{xxt}||^4 + ||w_{xxt}||^4    
\end{aligned}
$ 

\item 
$
\!
\begin{aligned}[t]
  \left| \left( u_{xxt}, w_{xx}w_{xtt} \right) \right|= \left| \left(w_{xx} u_{txx}, w_{xtt} \right) \right|
  \leq & ~C_{\varepsilon_3} ||\partial_x^4 w||^4 + C_{\varepsilon_3}  ||u_{xxt}||^4 + \varepsilon_3 ||w_{xxtt}||^2
\end{aligned}
$ 
\item 
$
\!
\begin{aligned}[t]
 \left|  \left( u_{xxt}, w_{xx}w_{xt} \right) \right|
  & \leq  \varepsilon ||u_{xxt}||^2 + C_{\varepsilon} ||w_{xt}||^2_{L^{\infty}}||w_{xx}||^2  \leq  \varepsilon ||u_{xxt}||^2 + C_{\varepsilon} ||w_{xt}||^{9/4}_{L^{\infty}} + C_{\varepsilon} ||w_{xx}||^{18}, \\
\end{aligned}
$ 
\end{enumerate}
where we used Young's inequality with $p = 9/8$ and $q=9$. Then we use interpolation for $||w_{xt}||^{9/4}_{L^{\infty}}$:
\begin{align}
\label{InterpolationTHISONE}
||w_{xt}||^{9/4}_{L^{\infty}} \leq ||w_{t}||^{9/4}_{3/2+\epsilon} \leq ||w_{t}||^{9/32} ||w_{xxt}||^{63/32} \leq C_{\varepsilon_p}||w_{t}||^{18} + \varepsilon_p||w_{xxt}||^{2},
\end{align}
where we chose $\epsilon = 1/4$ and Young's inequality with $p=64$ and $q=64/63$.

By choosing $\varepsilon$ and $\varepsilon_p$ sufficiently small, the above terms can be absorbed. Additionally, we note that from the previous energy bounds, \eqref{FirstLevel2}, $||w^n_{t}||$ and $||w^n_{xx}||$ are bounded in any power in which they appear.
\end{enumerate}

Denoting: ~$$I^n_2(t) = \frac{1}{2} ||u^n_{xxt}||^2  ~~~\text{and}~~~ \mathcal{E}^n_2(t) = E^n_2(t) + I^n_{2}(t),$$
where $E^n_2(t)$ is given by \eqref{secondFullStiffness}, we can obtain a clean estimate. It is true from \eqref{InitialData2} that, as before, $ \left\{\mathcal{E}^n_2(0) \right\}_{n=1}^{\infty}$ is uniformly bounded in terms of $||(w_0,w_1)||_{\cD(\cA)^2}$. Thus, combining \eqref{InertiaWttEstimate} with a compilation of the calculations described in this step and absorbing damping terms on the RHS, we have the estimate
\begin{align}
\mathcal{E}^n_1(t) +& \mathcal{E}^n_2(t) + k_2 \int_0^t \left[ ||w^n_{xxtt}||^2 + ||\partial_x^4 w^n_{t}||^2 \right]  d \tau   \nonumber \\
\lesssim&~ g_3\left(p_t, ||w_0||_{\mathcal D(\mathcal A)}, ||w_1||_{\mathcal D(\mathcal A)} \right)+ g_4\left(p,||w_0||_{\mathcal D(\mathcal A)},||w_1||_{H^2_{*}} \right)t+  \int_0^t \left [ \mathcal{E}^n_1(\tau) + \mathcal{E}^n_2(\tau) \right]^2 d\tau.  &  
\end{align}
We point out once again that the $C>0$ associated to `$\lesssim$' above  {\em does not depend} on $w_0, w_1$ or $p$ and  that the denoted dependence on $p$ (and its derivative) is taken in the sense of \eqref{timediff}.   
 
 Hence, disregarding the damping integral and invoking {\em nonlinear} Gr\"{o}nwall \cite{gronwall}, we obtain:
\begin{equation}
\label{NonlinGronwFULLMODEL}
\mathcal{E}^n_1(t) + \mathcal{E}^n_2(t) \leq  \frac{ g_3  + g_4t }{1- C \left[g_3 t + g_4 t^2 \right]} \equiv M_2(t)~~0 \leq t < T_1^* ~~\text{where}~~ T_1^* = \sup_{t} \left \{ C \left[g_3 t + g_4 t^2 \right] <1 \right \}. 
\end{equation}
From \eqref{NonlinGronwFULLMODEL}, we deduce that the Galerkin approximations $w^n$ satisfy a uniform-in-$n$ a priori bound on $[0,T]$ for any $T<T_1^*$:
$$0 \le \mathcal{E}^n_1(t) + \mathcal{E}^n_2(t) \le M^*_2(T) \equiv  \max_{t \in [0,T]} M_2(t).$$
This, along with \eqref{FirstLevel2}, provides uniform-in-$n$ boundedness in the associated norms of $\mathcal{E}_0$, $\mathcal{E}_1$ and $\mathcal{E}_2$ for a finite time depending on the initial data.

\vskip.1cm\noindent
{\bf Step 6 - Boundedness of Initial Jerk:} It is apparent from the expression of the {\bf [NL Inertia]} in \eqref{NLInertia} that the existence of strong solutions requires higher regularity of $w_{tt}$. We obtain this via yet another energy level, corresponding to two temporal differentiations of the equations. To begin, we again need uniform estimates of $t=0$ quantities appearing in the energy estimates. We remark that the resulting regularity of solutions obtained here is requisite also in the latter proof of uniqueness. Lastly, we note that in order to obtain this estimate (as well as that in the previous sections for $\cE_1^n$ and $\cE_2^n$) with $\iota=1$, the presence of the damping term $k_2>0$ is critical.

For the upcoming energy inequality for $\cE_3^n(t)$, we must justify boundedness in $n$ of ~{\small $\left\{||w^n_{ttt}(0)||\right\}_{n=1}^{\infty}$}, {\small$\left\{||u^n_{ttt}(0)||\right\}_{n=1}^{\infty}$}. To that end, the weak equations of motion \eqref{weakform} hold on approximants $w^n$ and can be differentiated in time for any fixed test function $\phi$. 
Then, by choosing $\phi= s_j(x)$, multiplying \eqref{weakform} by $q^{'''}_j(t)$ and summing over $j=1,2, \ldots, n$, we obtain:
\begin{align}
\label{atzero}
||w^n_{ttt}||^2 + D \left( \partial^4_x w^n_t, w^n_{ttt} \right)  + k_2 \left(  \partial^4_x w_{tt}^n, w^n_{ttt} \right) -D & \left( \partial_{xt}\left[(w^n_{xx})^2w^n_x \right],  w^n_{ttt} \right) + D  \left( \partial_{xxt}\left[w^n_{xx}(w^n_x)^2\right], w^n_{ttt} \right) \nonumber \\ & +\left( \partial_{xt}\left[w^n_x\int_x^L u^n_{tt} \right], w^n_{ttt} \right) - \left( p_{t}, w^n_{ttt} \right) = 0.
\end{align}

Differentiating directly, we have $u_{ttt} = - \int_0^x \left[ 3w_{xt}w_{xtt} + w_{x}w_{xttt} \right]d\xi,$ which yields:
\begin{align*}
\left( \partial_{xt} \left[w^n_x\int_x^L u^n_{tt} \right], w^n_{ttt} \right) =  - \left( \partial_{t} \left[w^n_x\int_x^L u^n_{tt}\right], w^n_{xttt} \right) 
= & - \left( w^n_{xt} \int_x^L u^n_{tt}, w^n_{xttt} \right) - \left( w^n_{x} \int_x^L u^n_{ttt}, w^n_{xttt} \right) \\
\equiv & ~~~~~~~~~~~~~~~\mathcal{K}_1 ~~~~~~~~~~~+~~~~~~~~~~~~ \mathcal{K}_2.
\end{align*}

For $\mathcal{K}_1$ we proceed by undoing the integration by parts which yields:
\begin{align*}
\mathcal{K}_1 =  \left( w^n_{xxt} \int_x^L u^n_{tt}, w^n_{ttt} \right)   - \left( w^n_{xt}  u^n_{tt}, w^n_{ttt} \right).
\end{align*}
These two terms can now be transferred to the right hand side and be estimated.
For $\mathcal{K}_2$ we recall the expression for $u_{ttt}$ above, and by adding and subtracting appropriate terms we have:
\begin{align*}
\mathcal{K}_2 = - \left(  u^n_{ttt}, \int_0^x w^n_{x}w^n_{xttt} \right) = ||u^n_{ttt}||^2 + 3 \left(  u^n_{ttt}, \int_0^x w^n_{xt}w^n_{xtt} \right).
\end{align*}
Grouping everything together, and absorbing {\small $||w^n_{ttt}(0)||^2$} and {\small $||u^n_{ttt}(0)||^2$} from  the RHS, we obtain:
\begin{align}
\label{threeder}
||w^n_{ttt}(0)||^2 + ||u^n_{ttt}(0)||^2 \leq h_1 \left( p_t(0), \partial_x^k w^n(0), \partial_x^lw^n_{t}(0), \partial_x^4 w^n_{tt}(0) \right),~~ k,l=1,2,3,4,
\end{align} 
with $h_1$ is polynomial in its slots. As we can see from the above expression, it is now crucial to establish the boundedness of the sequence {\small $\left\{ \partial_x^4 w^n_{tt} (0) \right\}_{n=1}^{\infty}$} in $L^2(0,L)$ in terms of the data, $(w_0,w_1) \in \cD(\cA^2)$.

To achieve this bound, we revisit the weak form  \eqref{weakform} and test with $\phi = \partial_x^8 s_{j}(x) $, then multiplying by $q''_{j}(t)$ and summing over $j=1,2, \ldots, n,$ yielding (after some integration by parts): 
\begin{align*}
||\partial_x^4 w^n_{tt}||^2    + \left(  \partial^5_x\left[w^n_x\int_x^L u^n_{tt} \right], \partial_x^4 w^n_{tt} \right) = \left(  p,  \partial_x^4 w^n_{tt} \right) - D \left(  \partial^8_x w^n, \partial_x^4 w^n_{tt} \right)  - k_2 \left(  \partial^8_x w_t^n, \partial_x^4 w^n_{tt} \right)\\  +D \left(  \partial^5_x\left[(w^n_{xx})^2w^n_x \right] , \partial_x^4 w^n_{tt} \right) -D \left(  \partial^6_{x}\left[w^n_{xx}(w^n_x)^2\right], \partial_x^4 w^n_{tt} \right).
\end{align*}
Brute force yields:
\begin{align*}
\label{fiveInertia}
\partial^5_x\left[w^n_x\int_x^L u^n_{tt} \right] =&~ \partial_x^6w^n \int_x^L u^n_{tt} - 5[\partial_x^5w ^n u^n_{tt}+w^n_{xx}  u^n_{xxxtt}] - 10 [\partial_x^4 w ^n u^n_{xtt} + w^n_{xxx}  u^n_{xxtt}]  -w^n_{x} \partial_x^4 u^n_{tt}. \\
\partial_x^4 u_{tt} =&~ - \left [ 6w_{xxt}w_{xxxt} + 2 w_{xt} \partial_x^4 w_{t} + \partial_x^4 w  w_{xtt}+3w_{xxx}w_{xxtt} + 3w_{xx}w_{xxxtt} +w_{x} \partial_x^4 w_{tt} \right ].\\
%
- \left( w^n_{x} \partial_x^4 u^n_{tt} , \partial_x^4 w^n_{tt} \right) =&~ \left(  \partial_x^4 u^n_{tt} , - w^n_{x} \partial_x^4 w^n_{tt} \right) \\ 
=&~  || \partial_x^4 u^n_{tt}  ||^2 +6 \left( \partial_x^4 u^n_{tt},  w^n_{xxt}w^n_{xxxt}\right) +2 \left( \partial_x^4 u^n_{tt},  w^n_{xt} \partial_x^4 w^n_{t}\right)  + \left(\partial_x^4 u^n_{tt},   \partial_x^4 w^n w^n_{xtt}\right)\\ &+3 \left( \partial_x^4 u^n_{tt}, w^n_{xxx}w^n_{xxtt}\right) +3 \left( \partial_x^4 u^n_{tt}, w^n_{xx}w^n_{xxxtt} \right).
\end{align*}

Combining the terms above,  we can extract {\small $||\partial_x^4 w^n_{tt}||^2$} and {\small $||\partial_x^4 u^n_{tt}||^2$} on the LHS. We group the RHS terms into different categories based on the actions that are necessary to control them. 
 \text{{\bf{\em  Type 1}} is first:}
\begin{align*}
T_1 \equiv \left(  p,  \partial_x^4 w^n_{tt} \right) - D \left(  \partial^8_x w^n, \partial_x^4 w^n_{tt} \right) - k_2 \left(  \partial^8_x w_t^n, \partial_x^4 w^n_{tt} \right)+D \left(  \partial^5_x\left[(w^n_{xx})^2w^n_x \right] , \partial_x^4 w^n_{tt} \right)  -D \left(  \partial^6_{x}\left[w^n_{xx}(w^n_x)^2\right], \partial_x^4 w^n_{tt} \right)\\ - \left( \partial_x^6w^n \int_x^L u^n_{tt},\partial_x^4 w^n_{tt} \right)  + 5 \left( \partial_x^5w^n  u^n_{tt}, \partial_x^4 w^n_{tt} \right) -6 \left( \partial_x^4 u^n_{tt},  w^n_{xxt}w^n_{xxxt}\right)  -2 \left( \partial_x^4 u^n_{tt},  w^n_{xt} \partial_x^4 w^n_{t}\right) ,
\end{align*}
where for these terms, it is clear that 
\begin{equation}
\label{Type1}
|T_1| \leq h_2\left( p, \partial_x^i w^n, \partial_x^j w^n_{t}, u^n_{tt} \right) + \varepsilon_1 ||\partial_x^4 w^n_{tt}||^2 + \delta_1 ||\partial_x^4 u^n_{tt}||^2, ~~i,j=1,2, \ldots, 8,
\end{equation}
where $h_2$ depends on $\varepsilon_1, \delta_1$ and is polynomial in its slots.
\noindent \text{{\bf{\em  Type 2}} is next:}
\begin{align}
\label{Type1.5}
T_2 \equiv 10 \left( \partial_x^4 w^n  u^n_{xtt},\partial_x^4 w^n_{tt} \right) + 10 \left( w^n_{xxx}  u^n_{xxtt},\partial_x^4 w^n_{tt} \right) + 5 \left( w^n_{xx}  u^n_{xxxtt},\partial_x^4 w^n_{tt} \right).
\end{align}

For this category we will exploit the fact that {\small $||\partial_x^4 u^n_{tt}||^2$} appears in the LHS and that {\small $\left\{ u^n_{tt} (0) \right\}_{n=1}^{\infty}$} is bounded in $L^2(0,L)$ as shown in \eqref{uttbound} which will be used in interpolation for the terms $\partial_x^i u_{tt},~i=1,2,3.$
We show how to control one of the terms appearing in \eqref{Type1.5}.
\begin{align*}
 \left| \left( w^n_{xx}  u^n_{xxxtt},\partial_x^4 w^n_{tt} \right) \right| \leq  C_{\varepsilon} ||w^n_{xxx}||^2||u^n_{xxxtt}||^2 + \varepsilon ||\partial_x^4 w^n_{tt} ||^2 
 \leq C_{\varepsilon} ||w^n_{xxx}||^{10} + C_{\varepsilon}  || u^n_{xxxtt}||^{5/2}  + \varepsilon ||\partial_x^4 w^n_{tt} ||^2,
\end{align*}
where we used Young's inequality with $p = 5$ and $q=5/4$. Then we use interpolation for $|| u^n_{xxxtt}||^{5/2}$:
\begin{align*}
|| u^n_{xxxtt}||^{5/2} \leq ||u^n_{tt}||^{5/8} ||\partial_x^4 u^n_{tt}||^{15/8} \leq C_{\varepsilon_p}||u^n_{tt}||^{10} + \varepsilon_p||\partial_x^4 u^n_{tt}||^{2},
\end{align*}
where employed Young's inequality once again with $p=16$ and $q=16/15$.

\begin{remark}
We can see from the explicit expression of $\partial_x^i u^n_{tt}$, $i=0,1,2,3$, that $$u^n_{tt}(0) = u^n_{xtt}(0)=u^n_{xxtt}(L)=u^n_{xxxtt}(L)=0.$$ Hence, Poincar\'{e}'s Inequality guarantees that $||u^n_{tt}||_{i} \sim ||\partial_x^i u^n _{tt}||$ for $i=1,2,3$.
\end{remark}
\noindent The remaining {\bf \em Type 2} are bounded analogously, yielding:
\begin{equation}
\label{Type2}
 \left| T_2 \right| \leq h_3 \left( \partial_x^i w^n, u^n_{tt} \right) + \varepsilon_2 ||\partial_x^4 w^n_{tt}||^2 + \delta_2 ||\partial_x^4 u^n_{tt}||^2, ~~i=1,2, \ldots, 5.
\end{equation}
\vskip.1cm 

Finally, we have \text{{\bf{\em  Type 3}}:} 
\begin{align*}
T_3 \equiv - \left(\partial_x^4 u^n_{tt},   \partial_x^4 w^n w^n_{xtt}\right)-3 \left( \partial_x^4 u^n_{tt}, w^n_{xxx}w^n_{xxtt}\right) -3 \left( \partial_x^4 u^n_{tt}, w^n_{xx}w^n_{xxxtt} \right).
\end{align*}
For this category, we interpolate the terms $\partial_x^i w_{tt},~ i=1,2,3,$ exploiting the fact that $\left\{ w^n_{tt} (0) \right\}_{n=1}^{\infty}$ is bounded in $L^2(0,L)$ as shown in \eqref{uttbound}. We omit these details, as the calculations are identical to those described for {\bf \em Type 2}. We obtain the bound:
\begin{equation}
\label{Type3}
 \left| T_3 \right| \leq h_4 \left( \partial_x^i w^n, w^n_{tt} \right) + \varepsilon_3 ||\partial_x^4 w^n_{tt}||^2 + \delta_3 ||\partial_x^4 u^n_{tt}||^2, ~~i=1,2, \ldots, 5.
\end{equation}

Combining \eqref{Type1}, \eqref{Type2} and \eqref{Type3}, absorbing with $\varepsilon_k,~\delta_k$ small, and taking the (valid on approximants) time trace at $t=0$, we produce the following estimate: 
\begin{equation*}
||\partial_x^4 w^n_{tt}(0)||^2 + ||\partial_x^4 u^n_{tt}(0)||^2 \leq  h \left( p(0), \partial_x^i w^n(0), \partial_x^j w^n_{t}(0), w^n_{tt}(0) , u^n_{tt}(0) \right), ~~i,j=1,2, \ldots, 8.
\end{equation*} 

By combining \eqref{InitialData2} and \eqref{uttbound}, we can {\em finally} write \eqref{threeder} as:
\begin{equation}
\label{utttbound}
||w_{ttt}^n(0)||^2+||u_{ttt}^n(0)||^2 \le C\left( p(0), p_t(0), ||w_0||_{\mathcal D(\mathcal A^2)}, ||w_1||_{\mathcal D(\mathcal A^2)}\right). 
\end{equation}

\vskip.1cm\noindent
{\bf Step 7 - Energy Level 3:} 
With the initial jerk bounded, we  proceed with the higher energy estimate corresponding to two time differentiations of the equation. The formal identity (applying $\partial_t^2$ to \eqref{dowellnon*} and multiplying by $w_{ttt}$) is:
\begin{align*}
 \frac{1}{2} \frac{d}{dt}& \big [ ||w_{ttt}||^2 +D ||w_{xxtt}||^2 +D||w_x w_{xxtt}||^2 +D||w_{xx}w_{xtt}||^2 \big ]  + k_2||w_{xxttt}||^2  \\
 = &~ D(w_{x} w_{xt},w_{xxtt}^2) + D(w_{xx} w_{xxt},w^2_{xtt}) - 4D(w_{xx}w_{xxt} w_{xt} ,w_{xttt}) - 2D(w_{x} w^2_{xxt} ,w_{xttt})\\ 
 &   - 2D(w_x w_{xx} w_{xxtt},w_{xttt})- 4D(w_{xxt}w_{x}w_{xt},w_{xxttt})  -2D(w_{xx} w^2_{xt}, w_{xxttt})-2D(w_{xx}w_{x}w_{xtt},w_{xxttt}) \\ 
 &- \left( \partial_{xtt} \left [ w_{x} \int_x^L u_{tt} \right ] , w_{ttt} \right).
\end{align*}

We bound the RHS, in line with previous sections, using the Sobolev embeddings and Young's; the estimates from stiffness terms are straightforward. 
Inertia is handled as in previous estimates. After two temporal differentiation we have:
\begin{align*}
\left( \partial_{xtt} \left [ w_{x} \int_x^L u_{tt} \right ] , w_{ttt} \right) = & - \left( \partial_{tt} \left [ w_{x} \int_x^L u_{tt} \right ] , w_{xttt} \right) \\
= & -  \left( w_{xtt} \int_x^L u_{tt}  , w_{xttt} \right) -  2 \left( w_{xt} \int_x^L u_{ttt}  , w_{xttt} \right) -  \left( w_{x} \int_x^L u_{tttt}  , w_{xttt} \right) \\
\equiv & ~~~~~~~~~~~~~~~~ \mathcal{L}_1 ~~~~~~~~~~~+~~~~~~~~~~~~~~ \mathcal{L}_2 ~~~~~~~~~~~+~~~~~~~~~~~~~\mathcal{L}_3.
\end{align*}
We bound $\mathcal{L}_1$ and $\mathcal{L}_2$ as:
\begin{enumerate}
\item
$
\!
\begin{aligned}[t]
 \left| \mathcal{L}_1 \right| \lesssim &~ ||w_{xtt}||_{L^{\infty}}||u_{tt}||~||w_{xttt}||
\leq  C_{\varepsilon_7} ||u_{tt}||^4 +C_{\varepsilon_7} ||w_{xxtt}||^4 + \varepsilon_7 ||w_{xxttt}||^2
\end{aligned}
$ 
\item
$
\!
\begin{aligned}[t]
\left| \mathcal{L}_2 \right| \lesssim &~ ||w_{xt}||_{L^{\infty}}||u_{ttt}||~||w_{xttt}||
\leq  C_{\varepsilon_8} ||w_{xxt}||^4 + C_{\varepsilon_8} ||u_{ttt}||^4  + \varepsilon_8 ||w_{xxttt}||^2.
\end{aligned}
$ 
\end{enumerate}
The term $\mathcal{L}_3$ creates the desired conserved quantity (again using the explicit representation of $u_{ttt}$): 
\begin{align*}
\mathcal{L}_3 = \left( u_{tttt}, u_{ttt} \right) + 3  \left(  u_{tttt},\int_0^x w_{xt} w_{xtt} \right ) = \frac{1}{2} \frac{d}{dt} ||u_{ttt}||^2 + 3  \left( \int_x^L u_{tttt}, w_{xt}w_{xtt} \right ).
\end{align*}
The additional term that was produced above can be manipulated as follows:
\begin{align*}
\left(  u_{tttt},\int_0^x w_{xt} w_{xtt} \right )= & \frac{d}{dt} \left( u_{ttt}, \int_0^x w_{xt}w_{xtt} \right) - \left( u_{ttt}, \int_0^x w^2_{xtt} \right) - \left( u_{ttt}, \int_0^x w_{xt}w_{xttt} \right) \\
 \equiv& ~~~~~~~~~~~~~\frac{d \mathcal{M}_1}{dt}  ~~~~~~~~~+~~~~~~~~~ \mathcal{M}_2 ~~~~~~+~~~~~~~~~~~~\mathcal{M}_3. 
\end{align*}

Now, $\mathcal{M}_2$ and $\mathcal{M}_3$ will be moved to the right hand side and estimated as follows:
\begin{enumerate}
\item
$
\!
\begin{aligned}[t]
 \left| \mathcal{M}_2 \right| \lesssim ||u_{ttt}||^2 + ||w^2_{xtt}||^2 \lesssim  ||u_{ttt}||^2+||w_{xtt}||^2_{L^{\infty}}||w_{xtt}||^2 
 \lesssim & ~||u_{ttt}||^2+||w_{xxtt}||^4
\end{aligned}
$ 
\item
$
\!
\begin{aligned}[t]
 \left| \mathcal{M}_3 \right| = \left| \left( w_{xt} \int_x^L u_{ttt}, w_{xttt} \right) \right| \leq ||w_{xt}||_{L^{\infty}}||u_{ttt}||~||w_{xttt}|| \
 \leq &~ C_{\varepsilon_9}||w_{xxt}||^4 + C_{\varepsilon_9} ||u_{ttt}||^4 + \varepsilon_9 ||w_{xxttt}||^2.
\end{aligned}
$ 
\end{enumerate}

The $\mathcal{M}_1$ is more delicate, since it must be absorbed by conservative quantities:
\begin{align*}
\label{Cone}
 \left| \mathcal{M}_1 \right| \leq   \varepsilon||u_{ttt}||^2 +C_{\varepsilon}||w_{xtt}||^2_{L^{\infty}}||w_{xt}||^2 & \leq \varepsilon||u_{ttt}||^2 +C_{\varepsilon}||w_{xtt}||^{9/4}_{L^{\infty}} +C_{\varepsilon} ||w_{xt}||^{18} \\
 & \leq \varepsilon||u_{ttt}||^2 +C_{\varepsilon, \varepsilon_p}||w_{tt}||^{18} +C_{\varepsilon} \varepsilon_p ||w_{xxtt}||^{2}+C_{\varepsilon} ||w_{xt}||^{18},
\end{align*}
accomplished as in \eqref{InterpolationTHISONE}.

Moving on, we compile the above calculations into an energy estimate, taking
\begin{equation*}
E_3(t) =  \frac{1}{2}\left [ ||w_{ttt}||^2 +D ||w_{xxtt}||^2 +D||w_x w_{xxtt}||^2 +D||w_{xx}w_{xtt}||^2 \right ] ~~~~\text{and} ~~~~ I_3(t) = \frac{1}{2} ||u_{ttt}||^2, 
\end{equation*}
and  subsequently ~ ${\small \ds \mathcal{E}_3(t) = E_3(t) + I_3(t).}$
Thus, by invoking \eqref{InitialData2} and \eqref{utttbound} to guarantee the uniform boundedness of $\cE^n_3(0)$, we obtain: 
\begin{align*}
\mathcal{E}^n_3(t) + k_2 \int_0^t ||w^n_{xxttt}||d \tau \leq g_5 \left(p,p_{t}, p_{tt}, ||w_0||_{\mathcal D(\mathcal A^2)}, ||w_1||_{\mathcal D(\mathcal A^2)} \right) + g_6\left(p,||w_0||_{\mathcal D(\mathcal A)},||w_1||_{H^2_{*}} , M_2^* \right)t \\
 +  \sum_{j=1}^{9} \varepsilon_{j} \int_0^t ||w^n_{xxttt}||^2 d\tau + C \int_0^t \left [ \mathcal{E}^n_3(\tau)\right]^2 d\tau~~~ \text{for all } t \in [0, T],
\end{align*}
where $T<T_1^*$ and $\ds M_2^*(T)$ are  as in \eqref{NonlinGronwFULLMODEL}. In addition, $C>0$ {\em does not depend} on $w_0, w_1$ or $p$.
Absorbing the damping terms, we finally obtain through another application nonlinear Gr\"onwall:
\begin{equation}
\label{GronwallforWttt}
\mathcal{E}^n_3(t) \leq  \frac{ g_5  + g_6t }{1- C \left[g_5 t + g_6 t^2 \right]} ~~0 \leq t < T_2^* ~~\text{where}~~ T_2^* = \min_t \left( \sup_{t} \left \{ C \left[g_3 t + g_4 t^2 \right] <1 \right \} , T_1^* \right) . 
\end{equation}
As before, this yields a uniform-in-$n$ a priori bound on solutions in the topology corresponding to $\cE_3$ on any $[0,T]$ for $T<T_2^*$. {We remark once again that the regularity of $p$ considered in theorem \eqref{withiota} is necessary for ensuring that the functions $g_1, g_2, \ldots, g_6$ are continuous functions in time, as required by the version of the  Gr\"onwall lemma we employ.}

\vspace{2.5cm}
\vskip.1cm \noindent{\bf Step 8 - Sufficient Regularity for $w_{t}$:} 

Regularity for the damping (with smooth data) proceeds standardly, through the equation:
\begin{align*}
|| \partial^4_{x}w^n_t|| 
\lesssim &~ ||p|| + ||w^n_{tt}|| +  ||w^n_{xxx}||~||\partial_x^4w^n||^2+||w^n_{xx}||~||w^n_{xxx}||^2+\left( 1 + ||w^n_{xx}||^2 \right) ||\partial_x^4 w^n|| + ||w^n_{xx}||~||u^n_{tt}||.
\end{align*}
Using \eqref{NonlinGronwFULLMODEL} we can deduce that \begin{equation}
\label{boundofwt}
||\partial^4_{x}w^n_t|| \text{ is bounded in } L^{\infty}(0,T; L^2(0,L)),
\end{equation} for any $T<T_2^*$.
Thus, combining \eqref{NonlinGronwFULLMODEL} and \eqref{GronwallforWttt} and \eqref{boundofwt}, we can finally obtain  a priori bounds: 
\begin{equation}
\label{finalreg2}
||w^n||_{L^{\infty}(0,T;\mathcal D(\mathcal A))}+||w^n_t||_{L^{\infty}(0,T;\mathcal D(\mathcal A))}+||w^n_{tt}||_{L^{\infty}(0,T;H^2_{*})} \le C(\text{data},T),
\end{equation}
(among other controlled norms), where ``data" indicates dependence on ~$(w_0,w_1)$ measured in norms up to that of $\cD(\cA^2)^2$.
\vskip.1cm\noindent
{\bf Step 9 - Limit Passage and Weak Solution:} With our a priori bounds in hand for smooth data $w_0 \in \mathcal D(\mathcal A^2),~w_1 \in \mathcal D(\mathcal A^2)$, we proceed to pass with the limit and construct a weak solution satisfying \eqref{weakform} with $\sigma = \iota = 1$ and $ k_2 >0$.
The boundedness of the terms in \eqref{finalreg2} yields to the existence of a subsequence $\left\{ w^{n_{k}}\right\}_{k=1}^{\infty}$ and a limit point $w \in H^1 \left(0, T; \mathcal D(\mathcal A) \right) \cap  H^2 \left(0, T; H^2_{*} \right)$, such that
\begin{equation*}
 w^{n_k} \rightharpoonup w \in L^2 \left(0, T; \mathcal D(\mathcal A) \right);~~~~
w^{n_k}_{t} \rightharpoonup w_{t} \in L^2 \left(0, T; \mathcal D(\mathcal A) \right);~~~~
w^{n_k}_{tt} \rightharpoonup w_{tt} \in L^2 \left(0, T; H^2_{*}  \right) .
\end{equation*}

We must show that $w$ satisfies the weak form \eqref{weakform}, in this case with $\sigma = \iota = 1$ and $k_2 >0$. The details corresponding to limit point identification for {\bf [NL Stiffness]} are identical to those in {\bf Step 6} of Section \ref{SectionProofOfStiffness}, thus we focus on {\bf [NL Inertia]} terms.

We first show that $u^{n_k}_{tt} \to u_{tt}$ in $L^2 \left(0, T; L^2(0,L) \right)$. To that end, we consider the differences:
\begin{align*}
||u_{tt}^{n_{k}} - u_{tt}|| \leq \left| \left| \int_x^L \left(  [w^{n_k}_{xt}]^2 - w_{xt}^2 \right)  \right| \right|^2 + \left| \left| \int_x^L \left( w^{n_k}_{x}w^{n_k}_{xtt} - w_{x}w_{xtt} \right)  \right| \right|^2  \equiv ~ \mathcal{Y}_1 + \mathcal{Y}_2.
\end{align*} 
We will show that both $\mathcal{Y}_1$ and $\mathcal{Y}_2$ go to zero as $k\to \infty$.
\begin{align*}
\mathcal{Y}_1 \lesssim ||w^{n_k}_{xt} + w_{xt}||^2_{L^{\infty}} ||w^{n_k}_{xt} - w_{xt}||^2 \lesssim  ||w_{xxt}||^2 ||w^{n_k}_{xt} - w_{xt}||^2 
 ~~\to 0 ~~\text{as}~~ k \to \infty.
\end{align*}
\begin{align*}
\mathcal{Y}_2 \leq ||w^{n_k}_{x} (w^{n_k}_{xtt} - w_{xtt})||^2 + ||w_{xtt} (w^{n_k}_{x} - w_{x})||^2 & \leq ||w^{n_k}_{x}||^2_{L^\infty} ||w^{n_k}_{xtt} - w_{xtt}||^2 + ||w_{xtt}||^2_{L^\infty}  ||w^{n_k}_{x} - w_{x}||^2 \\
& \lesssim ||w_{xx}||^2 ||w^{n_k}_{xtt} - w_{xtt}||^2 + ||w_{xxtt}||^2  ||w^{n_k}_{x} - w_{x}||^2 \\
& ~~~~\to 0 ~~\text{as}~~ k \to \infty.
\end{align*}

Now, in order to pass to the limit for the {\bf [NL Inertia]} term we need to show that 
\begin{equation*}
\left( w^{n_k}_{x} \int_x^L u^{n_k}_{tt} , \phi_x \right) \to \left( w_{x} \int_x^L u_{tt} , \phi_x \right)~~~~\text{for all } \phi \in H^2_{*}.
\end{equation*}
\begin{align*}
\left| \left( w^{n_k}_{x} \int_x^L u^{n_k}_{tt} - w_{x} \int_x^L u_{tt}, \phi_x \right) \right| & \leq \left| \left( w^{n_k}_{x} \int_x^L [ u^{n_k}_{tt} - u_{tt}], \phi_x \right) + \left(  ( w^{n_k}_{x} - w_{x} )\int_x^L u_{tt}, \phi_x \right) \right|\\
& \leq ||\phi_x||_{L^\infty}||w^{n_k}_{x}||~||u^{n_k}_{tt} - u_{tt}|| + ||\phi_x||_{L^\infty}||u_{tt}||~|| w^{n_k}_{x} - w_{x}|| \\
& \leq ||\phi_x||_{L^\infty}||w_{x}||~||u^{n_k}_{tt} - u_{tt}|| + ||\phi_x||_{L^\infty}||u_{tt}||~|| w^{n_k}_{x} - w_{x}||\\
& ~~~\to 0 ~~\text{as}~~ k \to \infty.
\end{align*}

Hence, $w$ satisfies the weak formulation \eqref{weakform} with  $\sigma = \iota = 1$ and $ k_2 >0$.
With a weak solution $w(x,t)$ in hand corresponding to smooth data, we have by Definition \ref{strongsol2} that the solution is strong, via the estimate \eqref{finalreg2}  that provides the necessary regularity for $w,~w_t,~w_{tt}$.

And thus we have proven theorem \ref{withiota}.
\begin{corollary} \label{StrongInertia} Strong solutions $w$, described in Definition \ref{strongsol2}, satisfy equation \eqref{dowellnon*} with $\sigma = \iota = 1$ and $ k_2 >0$ in the sense of $L^2(0,T;L^2(0,L))$. Additionally, they satisfy $w_{xx}(L,t)=w_{xxx}(L,t)=0$ for all $0 \leq t \leq T$.
\end{corollary} 

\begin{proof}[Proof of Corollary \ref{StrongInertia}] The weak form is now satisfied by the constructed limit:
\begin{align}
\label{weakform4}
(w_{tt}, \phi)+D(w_{xx},\phi_{xx}) + k_2(w_{xxt},\phi_{xx}) + D(w^2_{xx}w_x, \phi_x) + D(w_x^2 w_{xx}, \phi_{xx}) - & \left( w_{x} \int_x^L u_{tt}, \phi_{x} \right) = (p, \phi), \nonumber \\
& ~~~\forall \phi \in H^2_{*},~~a.e.~t.
\end{align}

Reversing integration by parts, yields (on test functions):
\begin{equation*}
\left( w_{tt} + D \partial^4_{x}w +k_2 \partial^4_{x}w_t  -D\partial_x\big[w_{xx}^2w_x \big]+D\partial_{xx}\big[w_{xx}w_x^2\big] + \partial_x \left [w_{x} \int_x^L u_{tt} \right] - p, \phi \right)=0, ~~~\forall \phi \in C_0^{\infty}(0,L).
 \end{equation*} 
By density, we have the equation holding in $L^2(0,L)$, as desired:
\begin{equation}\label{dumdum*}
w_{tt} + D \partial^4_{x}w +k_2 \partial^4_{x}w_t -D\partial_x\big[w_{xx}^2w_x \big]+D\partial_{xx}\big[w_{xx}w_x^2\big] + \partial_x \left [ w_{x} \int_x^L u_{tt} \right] = p ~~~a.e.~x,~~ a.e. ~t.
 \end{equation}

The solution resides in $H_*^2$, but we must show the natural boundary conditions $w_{xx}(L,t) = w_{xxx}(L,t)=0$. 
The argument proceeds as before, invoking \eqref{dumdum*} and yielding, upon integration by parts: 
\begin{align}
\label{boundaryful}
\phi_x(L) \left( (1+w_x^2(L))w_{xx}(L) +w_{xxt}(L) \right) - \phi(L)  \big((1+w_x^2(L))w_{xxx}(L) +  w_{x}(L)& w_{xx}^2(L)  +w_{xxxt}(L) \big)=0, \nonumber \\ & \forall \phi \in H^2_*,
\end{align}
where we interpret the time derivatives above distributionally. 
Considering $\phi \in H_0^1\cap H^2_* \subseteq H^2_*$, we see ~$
\phi_x(L) \left( (1+w_x^2(L))w_{xx}(L) +w_{xxt}(L) \right)=0.$
 There exists one such function so that $\phi_x(L) \neq 0$, and thus 
 \begin{equation*}
 w_{xxt}(L) +\left(1 + w_x^2(L)\right) w_{xx}(L)=0.
 \end{equation*}
 
 Now, since $w \in H^1(0,T; \cD(\cA))$ for smooth solutions, we have $w \in C([0,T];\cD(\cA))$. Hence $w_{xx}(L,t)$, $w_{xxx}(L,t)$ are continuous functions of time, so we have a linear ODE of the form $f'(t) + g(t)f(t)=0$, with classical solution
 \begin{equation*}
 w_{xx}(L, t) = w_{xx}(L,0) e ^{- \int_0^t (1+w_{x}^2(L,s))ds}.
 \end{equation*}
As $w_0 \in \mathcal D(\cA)$, $w_{xx}(L,0)=0$ and thus $w_{xx}(L,t)=0$ for all $t \in (0,T)$.
  
The same argument now applies for $\phi \in H^2_*$, yielding  $$w_{xxxt}(L) +\left(1 + w_x^2(L)\right) w_{xxx}(L)=0,$$
from which we deduce that $w_{xxx}(L,t)=0$ for all $t \in (0,T)$.
\end{proof}

\subsubsection{Uniqueness and Continuous Dependence}
Consider $w$ and $v$ to be two strong solutions of \eqref{dowellnon*} with  $\sigma = \iota = 1$ and $ k_2 >0$ and let $z \equiv w-v$. Using the multiplier $z_t$ on \eqref{dowellnon*} we obtain:
\begin{align}
\frac{1}{2} \frac{d}{dt} \Big [ ||z_t||^2 + D ||z_{xx}||^2 + D ||w_{x}z_{xx}||^2 +D||w_{xx}z_{x}||^2 \Big ] +k_2 ||z_{xxt}||^2 + \left(  \partial_x\left[w_x\int_x^L \overline{u}_{tt} - v_x\int_x^L \hat{u}_{tt} \right] , z_{t} \right) \nonumber \\  =  D \left( w_x w_{xt}, z^2_{xx} \right) -  D\left(v_{xx} \left [ w_x  + v_x \right ],z_{x} z_{xxt} \right) + D\left( w_{xx}w_{xxt}, z^2_{x} \right)-D \left(v_x \left [ w_{xx}  + v_{xx} \right ] , z_{xx}z_{xt} \right),
\end{align}
where
$$\overline{u}_{tt}(x)=-\int_0^x \left [w_{xt}^2+w_xw_{xtt} \right ]d\xi \hspace{5mm} \text{and} \hspace{5mm} \hat{u}_{tt}(x)=-\int_0^x \left [v_{xt}^2+v_xv_{xtt} \right ]d\xi.$$

The presence of strong damping allows us to estimate the RHS in a straightforward manner (without the subtlety needed in Section \ref{SectionUniq}):
\begin{enumerate}
\item  $ D  \left| \left( w_x w_{xt}, z^2_{xx} \right) \right| \lesssim ||w_{xx}||~||w_{xxt}||~||z_{xx}||^2$
\item
$
\!
\begin{aligned}[t]
 D  \left| \left(v_{xx} \left [ w_x  + v_x \right ],z_{x} z_{xxt} \right) \right|  \leq  C_{\varepsilon_1} ||v_{xx}||^2  ||w_{xx}+ v_{xx}|| ^2 ||z_{xx}||^2 + \varepsilon_1 ||z_{xxt}||^2
\end{aligned}
$ 
\item
$
\!
\begin{aligned}[t]
D  \left| \left( w_{xx}w_{xxt}, z^2_{x} \right) \right| \lesssim ||w_{xx}||~||w_{xxt}||~||z_{xx}||^2
\end{aligned}
$ 
\item
$
\!
\begin{aligned}[t]
 D  \left| \left(v_x \left [ w_{xx}  + v_{xx} \right ] , z_{xx}z_{xt} \right) \right| \leq C_{\varepsilon_2} ||v_{xx}||^2||w_{xxx} + v_{xxx}||^2||z_{xx}||^2 + \varepsilon_2||z_{xxt}||^2.
\end{aligned}
$ 
\end{enumerate}

For the inertial term, we have:
\begin{align*}
 \left(  \partial_x\left[w_x\int_x^L \overline{u}_{tt} - v_x\int_x^L \hat{u}_{tt} \right] , z_{t} \right) = & -\left(  (w_x - v_{x})\int_x^L \overline{u}_{tt}  , z_{xt} \right) - \left(  v_x \int_x^L \left [ \overline{u}_{tt} -  \hat{u}_{tt} \right ] , z_{xt} \right) \\
\equiv &~~~~~~~~~~~~~~~~~~ \mathcal{N} ~~~~~~~~~~~~~~+~~~~~~~~~~~~~~~~~ \mathcal{O}.
\end{align*}
Firstly:
\begin{align*}
 \left| \mathcal{N} \right|  =    \left| \left( \overline{u}_{tt}  , \int_0^x z_{x} z_{xt} \right) \right| 
 \lesssim ||z_{xt}||_{L^{\infty}}  || \overline{u}_{tt} || ~ ||z_{x}||
\leq  \varepsilon_3 ||z_{xxt}||^2 + c_{\varepsilon_3} ||\overline{u}_{tt}||^2||z_{xx}||^2.
\end{align*}
The second term $\mathcal{O}$ yields:
\begin{align*}
\mathcal{O} =   - \left(  \overline{u}_{tt} -  \hat{u}_{tt}   , \int_0^x v_x  z_{xt} \right) = & \left(\int_0^x \left [w_{xt}^2 -  v_{xt}^2 \right ] , \int_0^x v_x  z_{xt} \right)  + \left( \int_0^x \left [w_xw_{xtt} - v_xv_{xtt} \right]  , \int_0^x v_x  z_{xt} \right) \\
\equiv &~~~~~~~~~~~~~~~~~~ \mathcal{O}_1 ~~~~~~~~~~~~~~~~+~~~~~~~~~~~~~~~~~~ \mathcal{O}_2,
\end{align*}
where
\begin{align*} \mathcal{O}_2=\left( \int_0^x \left [z_xw_{xtt}  +v_xz_{xtt} \right]  , \int_0^x v_x  z_{xt} \right)=&~  \frac{1}{2} \frac{d}{dt} \left | \left | \int_0^x v_{x}z_{xt} \right | \right |^2 - \left( \int_0^x  v_{xt} z_{xt}    , \int_0^x v_x  z_{xt} \right) \\
&+\left( \int_0^x z_{x}w_{xtt}, \int_0^x v_x  z_{xt} \right).\end{align*}

The conserved $d / dt$ quantity will remain on the LHS, with the rest moved to the RHS and estimated: 
\begin{align*}
 \left| \left( \int_0^x  v_{xt} z_{xt}    , \int_0^x v_x  z_{xt} \right) \right|
\leq&~ \varepsilon_4 ||v_{xt}||^2_{L^{\infty}}  || z_{xt} ||^2  +  c_{\varepsilon_4} \left | \left | \int_0^x  v_{x} z_{xt} \right | \right |^2 \\
\leq&~ \varepsilon_4 ||v_{xxt}||^2~|| z_{xxt} ||^2  + c_{\varepsilon_4} \left | \left | \int_0^x  v_{x} z_{xt} \right | \right |^2\\
 \left| \left( \int_0^x z_{x}w_{xtt}, \int_0^x v_x  z_{xt} \right) \right|
\lesssim&~  ||w_{xtt}||^2  || z_{xx} ||^2  + \left | \left | \int_0^x  v_{x} z_{xt} \right | \right |^2.
\end{align*}
\begin{remark} The above calculation demonstrates the necessity of forming an energy identity (namely the {\bf Energy Level 3} formed in the proof of {\em Theorem} \ref{withiota}) that provides higher spatial regularity for $w_{tt}$.
\end{remark}

Lastly we have:
\begin{align*}
 \left| \mathcal{O}_1 \right| \leq  \left| \left(\int_0^x \left ( w_{xt} +  v_{xt}  \right ) z_{xt} , \int_0^x v_x  z_{xt} \right) \right|
\leq \varepsilon_5 \left(  ||w_{xxt} + v_{xxt}|| \right) ^2 ||z_{xxt}||^2 + c_{\varepsilon_5} \left | \left | \int_0^x  v_{x} z_{xt} \right |  \right |^2.
\end{align*}

Defining:
\begin{equation*}
\mathcal{E}(t) = \frac{1}{2} \left[ ||z_t||^2 + D ||z_{xx}||^2 + D ||w_{x}z_{xx}||^2 +D||w_{xx}z_{x}||^2 + \left | \left | \int_0^x v_{x}z_{xt} \right | \right |^2 \right ],
\end{equation*}
combining estimates for $\mathcal N$ and $\mathcal O$, and recalling ~{\small $\overline{u}_{tt}(x)=-\int_0^x \left [w_{xt}^2+w_xw_{xtt}\right]$}, we obtain:
\begin{align}
\label{UniquenessEsti}
\mathcal{E}(t) + k_2 \int_0^t ||z_{xxt}||^2 \leq \mathcal{E}(0) + \int_0^t K(w,v) \mathcal{E}(\tau) d\tau + \sum_{i=1}^5 \varepsilon_i \int_0^t C(w,v) ||z_{xxt}||^2,
\end{align}
where the above dependencies are of the following sense:
\begin{equation*}
K(w,v)= K\left( ||w||_3^2, ||w_{t}||_2^2, ||v||_3^2, ||w_{tt}||_1^2 \right)~~~\text{and}~~~C(w,v) = \left(||w||_2^2, ||w_{t}||_2^2 \right).
\end{equation*}

The regularity of strong solutions in the inertial case with data in $\mathcal D(\mathcal A^2)^2$ (see e.g., \eqref{GronwallforWttt}) provide $w_{ttt} \in L^{\infty}(0,T;L^2(0,T))$ (with a bound in terms of the data), and with $w_{tt} \in L^{\infty}(0,T;H^2_*)$; thus we have $w_{tt} \in C([0,T];H^1(0,L))$. From this, and the energy estimate for inertial solutions (\eqref{finalreg2} with Remark \ref{neededreg}), we obtain $C([0,T])$ boundedness (for $T<T^*(\text{data}_w,\text{data}_v)$) of the quantities $C(w,v),~K(w,v)$ the individual trajectories $(w,w_t),~(v,v_t)$.  Taking $\sup_{[0,T]}$ and choosing $\varepsilon_i$ sufficiently small (depending on the data), we obtain:
$$\cE(t) \le \mathcal C_1 \cE(0)+ \mathcal C_2 \int_0^t \cE(\tau)d\tau,$$
where $t \in [0,T]$ and we have the dependencies $\ds \mathcal C_i\Big(||(w_0,w_1)||_{\mathscr H^I_s}, ||(v_0,v_1)||_{\mathscr H^I_s},||p||_{H^{2}(0,T;L^2(0,L))} \Big)$.
The standard Gr\"{o}nwall lemma yields:
\begin{equation*}
\cE(t) \le \mathcal E(0) e^{\mathcal C_1 t},~~t \in [0,T].
\end{equation*}

Uniqueness and continuous dependence follow as in Section \ref{SectionUniq} for stiffness-only dynamics, i.e., in the sense that~ {\small $||z||_2^2+||z_t||^2 \lesssim \cE(t)$}.

\section{Global Solutions for Sufficiently Small Data}\label{global}

\subsection{Precise Statement of the Theorem}

\begin{theorem}\label{th:main3}
Suppose $\iota=\sigma=1$ with $k_2>0$, and take $p \equiv 0$. Then there exists a number $Q>0$ such that if ~$||(w_0,w_1)||_{\cD(\cA^2)\times \cD(\cA^2)} \le Q$, then the corresponding strong solution $(w,w_t)$ of \eqref{dowellnon*}--\eqref{dowellnon2*} has time of existence $T^*(w_0,w_1)=+\infty$ and there exist $M,\omega>0$ depending only on $Q$ such that 
$$||(w(t),w_t(t))||_{\cD(\cA^2)\times \cD(\cA^2)}^2 \le M\exp(-\omega t).$$
\end{theorem}
We note that the above theorem will obtain unproblematically in the case of $\iota=0$ and $k_2>0$, i.e., when nonlinear inertia is neglected and Kelvin-Voigt damping is included. On the other hand, it is clear the result should be possible with weaker damping. See the second point in Section \ref{open}.

\subsection{Outline of Proof}
We proceed as in \cite{ig1,xiang1} to obtain global existence indirectly via the Barrier method, which exploits the superlinearity in the problem. Using the damping, we will employ stabilization type multipliers at every energy level to obtain an inequality of the form in the theorem below, which we take from \cite{ig1}:
\begin{theorem}\label{decayingone}
Suppose that $X : [0,\infty) \to [0,\infty)$ is a continuous function such that there is a $T>0$ so that $X(T) <\infty$ and
\begin{equation}
\label{ExponentialDecay}
X(T) + C_1 \int_0^T X(s)ds \leq C_2X(0)+C_3 \sum_{i=1}^{N_1} X^{\alpha_i}(0) +C_4 \sum_{i=1}^{N_2}   X^{\beta_i}(T) + C_5 \sum_{i=1}^{N_3} \int_0^T  X^{\gamma_i}(s)ds,
\end{equation}
where $\alpha_i >1,~i=1,2, \ldots, N_1$, $\beta_j>1,~j=1,2,\ldots,N_2$ and $\gamma_k>1, k=1,2,\ldots, N_3$. Then there exists a $\mathscr C$ depending on $C_i, N_i, \alpha_i,\beta_i$ so that if $X(0) \leq \epsilon \le \mathscr C$, then 
$$X(t) \le \dfrac{ \epsilon}{\mathscr C} \exp\left({-\mathscr C t}\right).$$
\end{theorem}

For us, $X(t)$ will be the sum of (the majority of) the norms appearing in the formal energy identities we have constructed thus far. Any continuous function that satisfies the integral inequality has the desired property: exponential decay for sufficiently small initial conditions $X_i(0)$, which in turn yields global-in-time existence \cite{ig1}.

To form an inequality of the form \eqref{ExponentialDecay} we will utilize the previous calculations we have obtained for the energy estimates in Section \ref{inertia}, along with additional estimates based on equipartition multipliers at each level. We will attempt to form \eqref{ExponentialDecay} for {\em each}  $X_i(t),~i=0,1,2,3$  separately, where $X_i$'s correspond to each energy level we have defined before, and sum the results. To streamline exposition of comparable calculations in the earlier sections, we demonstrate the detailed calculations for $X_0(t)$. For $X_i(T),~i=1,2,3$, we will only highlight deviations from the details in the proof of theorems \ref{withoutiota} and \ref{withiota}. 

\subsection{Proof of Theorem $\ref{th:main3}$}
\begin{proof}[\unskip\nopunct]
{\bf Step 1 - Inequality for ${X}_0$:}
Recalling the estimate for $\cE_0$ in {\bf Step 2} of Section \ref{inertia}, we redefine: 
\begin{equation*}
X_{0}(t) = ||w_{t}(t)||^2 +  ||w_{xx}(t)||^2 + ||w_{x}w_{xx}(t)||^2 + ||u_{t}(t)||^2
\end{equation*}
and we have immediately the inequality:
\begin{equation}
\label{ZerothAddu}
X_0(T) + k_2 \int_0^T ||w_{xxt}||^2  \le X_0(0).
\end{equation}

It is crucial to retain the inertial term to appear under the integral sign, thus we augment \eqref{ZerothAddu} to obtain:
\begin{equation}
X_0(T) +  \int_0^T \left[ k_2 ||w_{xxt}||^2 +  ||u_{t}||^2 \right]  \le  X_0(0)+  \int_0^T ||u_{t}||^2.
\end{equation}
The inertial term appearing on the RHS will now have to be estimated. Note that if {\em we bound it above by $X_0(0)$}, it will appear under the time integral and such a bound would be inconsistent with the form of inequality \eqref{ExponentialDecay}. Rather, we estimate as:
\begin{align}
\label{InterpArg1}
||u_{t}||^2 \lesssim ||w_{x}w_{xt}||^2 \lesssim ||w_{xx}||^2 ||w_{xt}||^2 \lesssim ||w_{xx}||^6 + ||w_{xt}||^3 ,
\end{align}
where we've used Young's $p = 3$ and $q=3/2$. We interpolate~ {\small $||w_{xt}||^3$} as follows:
\begin{align}
\label{InterpArg2}
||w_{xt}||^3 \leq ||w_{t}||^{3/2}||w_{xxt}||^{3/2} \leq c_{\varepsilon_1}||w_{t}||^6 + \varepsilon_1 ||w_{xxt}||^2,
\end{align}
again using Young's with $p=4$ and $q=4/3$. Thus we have the $\varepsilon$ for absorption, and we obtain:
\begin{equation}
\label{E0oneNOTYET}
X_0(T) + c_1 \int_0^T \left[  ||w_{xxt}||^2 + ||u_{t}||^2 \right]  \leq X_0(0)+  c_2 \int_0^T \left [ ||w_{xx}||^6 + ||w_{t}||^6  \right].
\end{equation}

Invoking norm equivalence between $H^2(0,L)$ and $H^2_*$, \eqref{E0oneNOTYET} becomes:
\begin{equation}
\label{X0one}
X_0(T) +  c_1 \int_0^T \left[  ||w_{t}||^2 + ||u_{t}||^2 \right]  \leq X_0(0) +  c_2 \int_0^T \left [ ||w_{xx}||^6 + ||w_{t}||^6  \right].
\end{equation}

Now, the equipartition (stability) multiplier for this level is $w$; multiplying \eqref{dowellnon*} by the solution and integrating by parts in space and time, we obtain:
\begin{align*}
\frac{k_2}{2} ||w_{xx}(T)||^2 +& D \int_0^T \left[ ||w_{xx}||^2 + 2||w_x w_{xx}||^2 \right] - \int_0^T \left[ ||w_{t}||^2 + ||u_{t}||^2 \right] \\
=& ~
\frac{k_2}{2} ||w_{xx}(0)||^2 - \int_0^L ww_{t} \big|_0^T - 2  \int_0^L uu_{t} \big|_0^T.
\end{align*}

We note that from ~{\small $u(x,t)=-\frac{1}{2}\int_0^xw^2_x(\xi,t)d\xi$}, we have:
\begin{align*} ||u(t)||^2 \lesssim&~ \big|\big| \int_0^xw_x(\xi,t)  \left[\int_0^{\xi}w_{xx}(\zeta,t)d\zeta\right] d\xi\big|\big|^2 
\lesssim~||w_x(t)w_{xx}(t)||^2 \lesssim X_0(t),\end{align*}
via Fubini and Jensen's inequality, having then extended the integrals to $x \in [0,L]$. Hence,~ {\small $||u(T)||^2+||u(0)||^2 \leq c_3 X_{0}(0).$} The RHS can then be estimated straightforwardly, yielding:
\begin{align}
\label{X0two}
\frac{k_2}{2} ||w_{xx}(T)||^2 + D\int_0^T \left[ ||w_{xx}||^2 + 2||w_x w_{xx}||^2 \right] - \int_0^T \left[ ||w_{xxt}||^2 + ||u_{t}||^2 \right] \leq  c_3  X_{0}(0) .
\end{align}

We then take an appropriate linear combination of \eqref{X0one} and \eqref{X0two} (with constants depending on the damping coefficient $k_2$), and eliminate the negative terms appearing in \eqref{X0two}. Then, by possible adjustments of the constants, we have:
\begin{align}
\label{X0final}
X_0(T) + C_1 \int_0^T X_0 \leq C_2 X_0(0) + C_3 \int_0^T X_0^6. 
\end{align} 

\vskip.1cm \noindent
{\bf Step 2 - Inequality for $X_1$ and $X_2$:} In this step we will proceed by forming the inequality that corresponds to ~$X_1 +X_2$. As we will see later, there will be terms in the $X_2$ estimate that will need to by absorbed by some appearing in $X_1$.
 We define:
\begin{equation*}
X_1(t) = ||w_{tt}(t)||^2 + ||w_{xxt}(t)||^2 + ||w_{xt}(t)w_{xx}(t)||^2 + ||w_{x}(t)w_{xxt}(t)||^2 + ||u_{tt}(t)||^2.
\end{equation*}
\begin{remark} Note $X_1$ does not include the quantity ~{\small $\left | \left | \int_0^x  w^2_{xt} \right | \right |^2$}, as ~$\mathcal{E}_1$ does. As it can be seen from the following calculations, the aforementioned norm is not needed in obtaining \eqref{ExponentialDecay}. \end{remark}

Following similar calculations as in {\bf Step 4} in the proof of { Theorem} \ref{withiota}, we obtain:
\begin{align}
\label{InertiaDifferent}
\left ( \partial_{xt} \left [ w_x \int_x^L u_{tt} \right ] \ , w_{tt} \right) = \frac{1}{2} \frac{d}{dt} ||u_{tt}||^2 +  \frac{d}{dt} \left ( u_{tt}, \int_0^x w^2_{xt} \right ) -  3 \left ( u_{tt}, \int_0^x w_{xt}w_{xtt} \right ).
\end{align}
The conserved quantity of \eqref{InertiaDifferent} will remain to the LHS, while the remaining terms will be moved to the RHS and be estimated, after we proceed with integration in time, as follows:
\begin{enumerate}\setlength\itemsep{.5em}
\item 
$
\!
\begin{aligned}[t]
 \left| \left ( u_{tt}(T), \int_0^x w^2_{xt}(T) \right ) \right| \leq \delta_1 ||u_{tt}(T)||^2 + c_{ \delta_1} ||w_{xxt}(T)||^4 
\end{aligned}
$ 

\item 
$
\!
\begin{aligned}[t]
 \left| \left ( u_{tt}(0), \int_0^x w^2_{xt}(0) \right ) \right| & \lesssim ||u_{tt}(0)||^2 + ||w_{xxt}(0)||^4 
\end{aligned}
$ 

\item 
$
\!
\begin{aligned}[t]
 \left| \left ( u_{tt}, \int_0^x w_{xt}w_{xtt} \right ) \right| \leq  \delta_2 ||u_{tt}||^2 + c_{ \delta_2} ||w_{xxt}||^6 + c_{ \delta_2, \varepsilon_1} ||w_{tt}||^6 + c_{ \delta_2} \varepsilon_1 ||w_{xxtt}||^2.
\end{aligned}
$
\end{enumerate}

In addition to the above calculations, we add the term {\small $\int_0^T ||u_{tt}||^2$} to both sides of the inequality. On the RHS it will be estimated via:
\begin{align*}
||u_{tt}||^2 \lesssim ||w_{xxt}||^2 ||w_{xt}||^2 + ||w_{xxt}||^2||w_{x}||~||w_{xtt}|| + ||w_{xx}||^2||w_{xtt}||^2.
\end{align*}

Using Young's and interpolation, as in \eqref{InterpArg1} and \eqref{InterpArg2}, and directly invoking the stiffness calculations in {\bf Step 4} in the proof of { Theorem} \ref{withoutiota}, we arrive at:
\begin{align}
\label{X1one}
X_1(T) +  \int_0^T \left[  k_2||w_{xxtt}||^2 + ||u_{tt}||^2 \right]  \leq c_1 \left(X_1(0) + X^2_1(0)+ X_0^2(0) + X_0^{16}(0) \right) + c_2 X^2_1(T) \nonumber \\   + c_3 \int_0^T  \left[  X^2_1 + X^3_1  + X^4_1 + X^2_0 + X^3_0  \right] 
 + \delta \int_0^T  ||u_{tt}||^2  + \varepsilon \int_0^T ||w_{xxtt}||^2,
\end{align}
where $\varepsilon$ and $\delta$ collect the various $\varepsilon_i$'s and $\delta_i$'s corresponding to earlier applications of Young's inequality.

For the time-differentiated version of the equations, $w_t$ acts as the equipartition multiplier. After the appropriate calculations, and straightforward estimation, we obtain:
\begin{align}
\label{X1two}
\frac{k_2}{2} &||w_{xxt}(T)||^2 + D\int_0^T \left[ ||w_{xxt}||^2 + ||w_{xt} w_{xx}||^2 +||w_x w_{xxt}||^2 \right] - \int_0^T \left[ ||w_{tt}||^2 + ||u_{tt}||^2 \right] \nonumber \\
\leq&~ c_1 (X_1(0) + X_0(0))  + \varepsilon_1 ||w_{tt}(T)||^2 + \delta_1||u_{tt}(T)||^2  + c_2 \int_0^T \left[ X_1^2 + X^0_1 \right ] + \delta_2 \int_0^T ||u_{tt}||^2.
\end{align}

Now, we define:
\begin{equation*}
X_2(t) = ||w_{xxt}(t)||^2 + ||\partial_x^4w(t)||^2 + ||w_x(t)\partial_x^4w(t)||^2 + ||u_{xxt}(t)||^2.
\end{equation*}
Then, duplicating the calculations in {\bf Step 5} in the proof of {Theorem} \ref{withiota} and adding ~ {\small $\int_0^T ||u_{xxt}||^2$} to both sides we have:
\begin{align}
\label{X2one}
X_2(T) +  \int_0^T \left[ k_2 ||\partial_x^4 w_t||^2 + ||u_{xxt}||^2 \right ]  \leq c_1 \left(X_2(0) +  X_0^9(0)  \right)  + c_2 \int_0^T \left[X_2^2 + X_1^2 + X_0^{10/3} + X_0^4    \right ] \nonumber \\
 +\varepsilon_2 \int_0^T  ||w_{xxtt}||^2 . 
\end{align}

\begin{remark}
The term $||w_{xxtt}||^2$ appearing on the RHS of the above inequality is the reason why we chose to have the calculations of $X_1$ and $X_2$ combined.
\end{remark}

To complete the estimate for $X_2(t)$,  we proceed by employing $\partial_x^4w$ as a multiplier. The calculations corresponding to stiffness are described in {\bf Step 5} in the proof of {Theorem} \ref{withoutiota}. Inertial terms are handled through differentiation and spatial integration by parts:
\begin{align*}
\int_0^T \left( \partial_{x} \left [ w_{x} \int_x^L u_{tt} \right ] , \partial_x^4 w \right) = & \int_0^T \left(  w_{xx} \int_x^L u_{tt},\partial_x^4 w \right)- \int_0^T \left( w_{x}u_{tt} , \partial_x^4 w \right)  \\
= \int_0^T \left(  w_{xx} \int_x^L u_{tt},\partial_x^4 w \right) &- \int_0^T \left( u_{tt}, w_{xx}w_{xxx} \right) - 2 \int_0^T  \left( u_{ttx}, w^2_{xx}\right) - \int_0^T \left( u_{ttxx}, w_{x} w_{xx}\right).
\end{align*}
The only non-trivial term to estimate is the last; integrate by parts {\em in $t$} and note $\partial_t \left(- w_{x}w_{xx} \right) = u_{xxt}$:
\begin{align*}
- \int_0^T \left( u_{ttxx}, w_{x} w_{xx}\right) = -w_{x}w_{xx}u_{xxt} \big|_0^T - \int_0^T ||u_{xxt}||^2.
\end{align*}

Combining, we obtain:
\begin{align}
\label{X2two}\nonumber
\frac{k_2}{2} ||\partial_x^4 w(T)||^2 +& D \int_0^T \left[ ||\partial_x^4 w||^2 + ||w_x \partial_x^4 w||^2 \right ]  -  \int_0^T \left[ ||w_{xxt}||^2 + ||u_{xxt}||^2 \right ] \\ 
\leq & ~c_1 \left( X_2(0) + X_1(0) +  X_0(0) + X_0^2(0)\right) \nonumber + \varepsilon_3 ||w_{xxt}(T)||^2 + \delta_3 ||u_{xxt}(T)||^2  \\ &+ c_2\int_0^T \left[ X_2 ^2 + X_1 ^2 + X_0^2 + X_0^9 +  X_0^{17}  \right ] +\varepsilon_4 \int_0^T ||\partial_x^4 w||^2.
\end{align}

As before, we add \eqref{X1one} to \eqref{X2one}, and we add \eqref{X1two} to \eqref{X2two}; we then choose an appropriate linear combination of the sums for absorption of negative integral terms; we then  choose $\varepsilon_i$, $\delta_i$ appropriately, and invoke norm equivalence for $H^2_*$, yielding the estimate for $X_1 + X_2$:
%
%
\begin{align}
\label{X1X2final}
 X_2(T) + &X_1(T) +  C_1 \int_0^T \left[ X_2 + X_1 \right] \nonumber \\  \leq &~C_2  \Big( X_2(0)  + X_1(0)+ X^2_1(0) + X_0(0)+ X_0^2(0)  + X_0^9(0) + X_0^{16}(0)      \Big) \nonumber  \\ 
& + C_3 X^2_1(T) + C_4 \int_0^T \left[ X_2^2 + X_2^4 + X_1^2 + X^3_1 + X^4_1 + X_0^2 + X^3_0 + X_0^9 + X_0^{17} \right ].
\end{align}
{\bf Step 3 - Inequality for $X_3$:} 
Define:
\begin{equation*}
X_3(t) = ||w_{ttt}(t)||^2 + ||w_{xxtt}(t)||^2 + ||w_x w_{xxtt}(t)||^2 + ||w_{xx}(t)w_{xtt}(t)||^2 + ||u_{ttt}(t)||^2.   
\end{equation*} 
The estimate corresponding to {\em two time} differentiations of \eqref{dowellnon*} with the multiplier $w_{ttt}$ can be directly formed from the existing calculations for {\bf Step 7} in the proof of {\em Theorem} \eqref{withiota}.
\begin{align}
\label{X3one}
 X_3(T) +  c_1 \int_0^T  \left[  ||w_{ttt}||^2 + ||u_{ttt}||^2 \right]  
\leq  c_2 \left( X_3(0) + X^9_2(0) + X^9_1(0) \right) + c_3 \left( X_2^9(T) + X_1^9(T) \right) \\\nonumber + c_4 \int_0^T \left[X_3^2 + X_2^2 +  X_2^4 + X_1^2 + X_1^3 + X_1^4 + X_0^2 + X_0^3 + X_0^4  \right] ,
\end{align}
where we added {\small $\int_0^T ||u_{ttt}||^2$} to both sides and proceeded as in earlier estimates in this section.

In this case, $w_{tt}$ is the equipartition multiplier, and calculations  corresponding to stiffness are duplicated from {\bf Step 7} in the proof of {\em Theorem} \eqref{withiota}. The inertial term calls for a slightly altered approach:
\begin{align*}
\left( \partial_{xtt} \left [ w_{x} \int_x^L u_{tt} \right ] , w_{tt} \right) =  -  \left( w_{xtt} \int_x^L u_{tt}  , w_{xtt} \right) -  2 \left( w_{xt} \int_x^L u_{ttt}  , w_{xtt} \right) -  \left( w_{x} \int_x^L u_{tttt}  , w_{xtt} \right).
\end{align*}

The first two terms above can be treated similarly to $\mathcal J_i$ of {\bf Step 5} in the proof of {\em Theorem} \eqref{withiota}, and for the latter we write:
\begin{align*}
\int_0^T \left(  u_{tttt}  , -  \int_0^x w_{x}w_{xtt} \right) = \int_0^T \left( u_{tttt}, u_{tt} \right) + \int_0^T \frac{d}{dt}  \left( u_{ttt}, \int_0^x w^2_{xt} \right) - 2 \int_0^T \left( u_{ttt}, \int_0^x w_{xt}w_{xtt} \right).
\end{align*}
The first term will be integrated by parts {\em in time} and the following two will be estimated as above.

Hence, assembling everything together we have:
\begin{align}
\label{X3two}
\frac{k_2}{2} ||w_{xxtt}(T)||^2 & + c_1 \int_0^T \left[ ||w_{xxtt}||^2 + ||w_{x}w_{xxtt}||^2 + ||w_{xx}w_{xtt}||^2 \right] - c_2 \int_0^T \left[ ||w_{ttt}||^2 + ||u_{ttt}||^2 \right] \nonumber \\
& \leq c_3 \left( X_3(0) + X_1^2(0) + X_1(0) \right) + c_4 X_1(T)   +\varepsilon_1||w_{ttt}(T)||^2 + \varepsilon_2||u_{ttt}(T)||^2 \nonumber + c_5 X_1^2(T)   \\
& ~~~+ c_6 \int_0^T \left[X_3^2 + X_2^2 + X_2^4 + X_1^2 + X_1^4 + X_0^2 + X_0^4   \right] .
\end{align}

Combining \eqref{X3one} with \eqref{X3two} in an appropriate linear combination, we obtain:
\begin{align}
\label{X3final}
X_3(T) + C_1 \int_0^T X_3  \leq C_2 \left(X_3(0) + X^9_2(0) + X_1(0) + X^2_1(0) + X^9_1(0) \right)  + C_3 \left( X_1^9(T) + X_2^9(T) \right) \nonumber\\ + C_4 X_1(T) 
+ C_5 \int_0^T \left[X_3^2 + X_2^2 +  X_2^4 + X_1^2 + X_1^3 + X_1^4 + X_0^2 + X_0^3 + X_0^4    \right].
\end{align}
Here we remark that the term $X_1(T)$ appearing above is bounded by \eqref{X1one}.

Finally, we note that the bound above depends on the boundedness of the quantity $X_3(0)$ which contains the term $||w_{ttt}(0)||^2+||u_{ttt}(0)||^2$. This term does not explicitly appear as data, however, it is directly bounded by the data $||(w_0, w_1) ||^2_{\cD(\cA^2)^2}$, which can be shown directly on approximate solutions, as was the focus of {\bf Step 6} in Section \ref{withiota}. 

\vskip.1cm\noindent
{\bf Step 4 - Global Estimate:} With the constituent inequalities in hand from {\bf Steps 1--3}, we  form: 
\begin{equation*}
X(T) = X_0(T) + X_1(T) + X_2(T) + X_3(T).
\end{equation*}
This quantity is nonnegative, and continuous due to the regularity of constructed solutions.  We then add \eqref{X0final}, \eqref{X1X2final} and \eqref{X3final}, and with minor algebraic manipulations, we obtain:
\begin{equation}
\label{XGlobal}
X(T) + C_1 \int_0^T X(s)ds \leq C_2 X(0) +C_3 F_0 + C_4 \left( X^2(T) + X^9(T) \right) + C_5 \int_0^T F(s)ds,
\end{equation}
where
\begin{equation*}
F_0 =X^2(0) + X^9(0) + X^{16}(0)
\end{equation*} 
and
\begin{equation*}
F(s) = X^2(s) + X^3(s) + X^{10/3}(s) + X^4(s) + X^6(s) + X^9(s) + X^{17}(s).
\end{equation*} 
This final estimate \eqref{XGlobal} is of the form in Theorem \ref{decayingone}, which concludes the proof of Theorem \ref{th:main3} \end{proof}

\section{Comments, Open Problems, and Future Work}\label{open}
We briefly state and discuss some open problems and directions for future work.
\begin{itemize}
\item {\bf The existence of finite energy, weak solutions} seems to be a challenging one. It is clear that additional compactness is needed for the identification of limit points associated only to the stiffness portion of the dynamics. Compensated compactness requirements (e.g., those in $L^1$) might be adapted to the nonlinear structures, though it is unclear if such an approach would be more expedient than the higher order energy methods employed here.
\item {\bf The elimination or weakening of damping} seems a natural course. In our estimates, it is clear that the regularizing effects of Kelvin-Voigt damping are stronger than explicitly needed in the construction of solutions and estimation of inertial terms.  On the other hand, weak damping of the form $k_0w_t$ is clearly too weak to address inertial terms. Unfortunately, for cantilevered beams, the physical interpretation of $A^{1/2}w_t \sim k_1\partial_x^2w_t$ damping is unclear---see the discussions in \cite{dghw} and \cite{HHWW}. Additionally, it is a question for future work to utilize weaker (than $Aw_t$)damping to obtain global solutions with sufficiently small data for \eqref{dowellnon*} with $ \sigma, k_2=1$ and $\iota = 0$.
\item Explicit {\bf proof of blow up} for large data in this quasilinear system would nicely complement our local existence results. Currently, numerical evidence indicates that large data quickly leads to non-physical solutions.
\item {\bf The introduction of non-conservative forces} as discussed in the introduction (e.g., with application to piezoelectric energy harvesting) is a natural next step. In fact, the earlier work \cite{dghw} addresses a piston-theoretic beam, as does the more recent \cite{follower,McHughIFASD2019}. However, exploiting the superlinearity of the nonlinear stiffness to provide a rigorous framework for long-time behavior of trajectories---or even constructing limit cycle oscillations---is a desirable future goal.
\item {\bf The 2-D cantilever model}, invoking inextensible elasticity (see the engineering references \cite{dowell4,inext2}), is the topic of forthcoming work. This challenging mathematical problem was untouchable before establishing the theory in this treatment. Difficulties for the 2-D problem include the challenging mixed, clamped-free-type plate boundary conditions, as well as the loss of the 1-D Sobolev embeddings (which were used profusely and non-trivially) in this treatment. Closing estimates will require {\em even higher} differentiations of the equations, resulting in further involved calculations beyond the numerous pages here.
\end{itemize}

\section{Acknowledgements} \label{ack}
The authors are generously supported by NSF-DMS-1907620, ``Experiment, Theory, and Simulation of Aeroelastic Limit Cycle Oscillations for Energy Harvesting Applications".


\begin{thebibliography}{99}\scriptsize

\bibitem{antman} Antman, S.S., 2005. Nonlinear Problems of Elasticity, volume 107 of Applied Mathematical Sciences, Springer. New York.


\bibitem{antman2} Antman, S.S. and Seidman, T.I., 2005. The parabolic-hyperbolic system governing the spatial motion of nonlinearly viscoelastic rods. {\em Archive for Rational Mechanics and Analysis}, 175(1), pp.85--150.

\bibitem{bolotin} Bolotin, V.V., 1963. {\em Nonconservative problems of the theory of elastic stability}. Macmillan.


\bibitem{che-tri:89:PJM} Chen, S.P. and Triggiani, R., 1989. Proof of extensions of two conjectures on structural damping for elastic systems. {\em Pacific J. of Mathematics}, 136(1), pp.15-55.

\bibitem{survey2} Chueshov, I., Dowell, E.H., Lasiecka, I. and Webster, J.T., 2016. Nonlinear Elastic Plate in a Flow of Gas: Recent Results and Conjectures. {\em Applied Mathematics \& Optimization}, 73(3), pp.475-500.

\bibitem{springer} Chueshov, I. and Lasiecka, I., 2010. {\em Von Karman Evolution Equations: Well-posedness and Long Time Dynamics}. Springer Science \& Business Media.

\bibitem{ciarlet} Ciarlet, P.G., 1997. {\em Mathematical Elasticity: Volume II: Theory of Plates}. Elsevier.

\bibitem{dghw} Deliyianni, M., Gudibanda, V., Howell, J. and Webster, J.T., 2020. Large deflections of inextensible cantilevers: modeling, theory, and simulation. {\em Mathematical Modelling of Natural Phenomena}, 15, p.44.

\bibitem{dowell} Dowell, E.H., Clark, R. and Cox, D., 2004. {\em A modern course in aeroelasticity} (Vol. 3). Dordrecht: Kluwer academic publishers.

\bibitem{inext1} Dowell, E. and McHugh, K., 2016. Equations of motion for an inextensible beam undergoing large deflections.{\em J. of Applied Mechanics}, 83(5), p.051007.

\bibitem{gronwall} Dragomir, S.S., 2003. {\em Some Gronwall type inequalities and applications}. Nova Science.


\bibitem{DOWELL} Dunnmon, J.A., Stanton, S.C., Mann, B.P. and Dowell, E.H., 2011. Power extraction from aeroelastic limit cycle oscillations. {\em J. of Fluids and Structures}, 27(8), pp.1182-1198.

\bibitem{lions} Duvant, G. and Lions, J.L., 2012. {\em Inequalities in mechanics and physics} (Vol. 219). Springer Science \& Business Media.

\bibitem{energyharvesting} Erturk, A. and Inman, D.J., 2011.{\em Piezoelectric energy harvesting}. John Wiley \& Sons.
 
 \bibitem{evans} Evans, L.C., 2010.{\em Partial differential equations} (Vol. 19). American Mathematical Soc..
 
\bibitem{fab-han:01:DCDS} Fabiano, R.H. and Hansen, S.W., 2001. Modeling and analysis of a three-layer damped sandwich beam. In {\em Conference Publications} (Vol. 2001, No. Special, p. 143). American Institute of Mathematical Sciences.

\bibitem{HHWW} Howell, J., Huneycutt, K., Webster, J.T. and  Wilder, S., 2019. (In)stability of a cantilevered piston-theoretic Beam, {\em Mathematics in Engineering}.

\bibitem{HTW} Howell, J.S., Toundykov, D. and Webster, J.T., 2018. A Cantilevered Extensible Beam in Axial Flow: Semigroup Well-posedness and Postflutter Regimes.{\em SIAM J. on Mathematical Analysis}, 50(2), pp.2048-2085.

\bibitem{huang} Huang, L., 1995. Flutter of cantilevered plates in axial flow. {\em J. of Fluids and Structures}, 9(2), pp.127-147.

\bibitem{ig1} Ignatova, M., Kukavica, I., Lasiecka, I. and Tuffaha, A., 2014. On well-posedness and small data global existence for an interface damped free boundary fluid-structure model. {\em Nonlinearity}, 27(3), p.467.

\bibitem{book} Kaltenbacher, B., Kukavica, I., Lasiecka, I., Triggiani, R., Tuffaha, A. and Webster, J.T., 2018. {\em Mathematical Theory of Evolutionary Fluid-Flow Structure Interactions}. Springer International Publishing.
	
\bibitem{koch} Koch, H. and Lasiecka, I., 2002. Hadamard well-posedness of weak solutions in nonlinear dynamic elasticity-full von Karman systems. In {\em Evolution equations, semigroups and functional analysis} (pp. 197-216). Birkh\"{a}user, Basel.

\bibitem{old}  Kou\'{e}mou-Patcheu, S., 1997. Global existence and exponential decay estimates for a damped quasilinear equation. {\em Communications in Partial Differential Equations}, 22(11-12), pp.2007--2024.

\bibitem{lagleug} Lagnese, J.E. and Leugering, G., 1991. Uniform stabilization of a nonlinear beam by nonlinear boundary feedback.{\em J. of Differential Equations}, 91(2), pp.355-388.

\bibitem{xiang1} Lasiecka, I., Pokojovy, M. and Wan, X., 2019. Long-time behavior of quasilinear thermoelastic Kirchhoff--Love plates with second sound.{\em Nonlinear Analysis}, 186, pp.219-258.

\bibitem{redbook} Lasiecka, I. and Triggiani, R., 2000.{\em Control theory for partial differential equations: Volume 1, Abstract parabolic systems: Continuous and approximation theories} (Vol. 1). Cambridge University Press.

\bibitem{follower} McHugh, K.A. and Dowell, E.H., 2019. Nonlinear Response of an Inextensible, Cantilevered Beam Subjected to a Nonconservative Follower Force.{\em J. of Computational and Nonlinear Dynamics}, 14(3), p.031004.

\bibitem{McHughIFASD2019} McHugh, K.A., Beran, P., Freydin, M. and Dowell, E.H., 2019. Flutter and Limit Cycle Oscillations of a Cantilevered Plate in Supersonic/Hypersonic Flow. {\em Proceedings of IFASD.}

\bibitem{paidoussis} Paidoussis, M.P., 1998. {\em Fluid-structure interactions: slender structures and axial flow} (Vol. 1). Academic press.

\bibitem{beamdamping} Russell, D.L., 1991. A comparison of certain elastic dissipation mechanisms via decoupling and projection techniques. {\em Quarterly of Applied Mathematics}, 49(2), pp.373-396.
  
\bibitem{dowellmaterial} Sayag, M.R. and Dowell, E.H., 2019. Nonlinear Structural, Inertial and Damping Effects in an Oscillating Cantilever Beam. {\em In Nonlinear Dynamics, Volume 1 }(pp. 387-400). Springer, Cham.

\bibitem{semler2} Semler, C., Li, G.X. and Paidoussis, M.P., 1994. The non-linear equations of motion of pipes conveying fluid.{\em J. of Sound and Vibration}, 169(5), pp.577-599.

\bibitem{piezomass} Stanton, S.C., Erturk, A., Mann, B.P., Dowell, E.H. and Inman, D.J., 2012. Nonlinear non-conservative behavior and modeling of piezoelectric energy harvesters including proof mass effects.{\em J. of Intelligent Material Systems and Structures}, 23(2), pp.183-199.
 
\bibitem{stoker} Stoker, J.J., 1947.{\em Nonlinear elasticity}. Gordon and Breach.

\bibitem{dowell4} Tang, D., Gibbs, S.C. and Dowell, E.H., 2015. Nonlinear aeroelastic analysis with inextensible plate theory including correlation with experiment.{\em AIAA J.}, 53(5), pp.1299-1308.

\bibitem{inext2} Tang, D., Zhao, M. and Dowell, E.H., 2014. Inextensible beam and plate theory: computational analysis and comparison with experiment.{\em J. of Applied Mechanics}, 81(6), p.061009.

\bibitem{wonkrieg} Woinowsky-Krieger, S., 1950. The effect of an axial force on the vibration of hinged bars.{\em J. Applied Mechanics}, 17(1), pp.35--36.

\end{thebibliography}
\end{document}